\documentclass[11pt,leqno]{amsart}
\pdfoutput=1

\usepackage{latexsym}
\usepackage{amssymb}

\usepackage{ifpdf}
\usepackage{graphicx}

\ifpdf
\usepackage[usenames,dvipsnames]{xcolor}
\usepackage[pdfusetitle,
bookmarks=true,bookmarksnumbered=true,bookmarksopen=true,bookmarksopenlevel=1,
breaklinks=false,pdfborder={0 0 
  0},backref=false,colorlinks=true,linkcolor=NavyBlue,citecolor=NavyBlue,urlcolor=black]
{hyperref}
\pdfpageheight\paperheight
\pdfpagewidth\paperwidth

\fi
\usepackage{amsmath,amsthm,amsfonts}
\usepackage[capitalise]{cleveref}
\crefformat{equation}{(#2#1#3)}

\usepackage[shortlabels]{enumitem}

\usepackage{array}
\newcommand\undermat[3]{%
  \makebox[0pt][l]{$\smash{\underbrace{\phantom{%
    \begin{matrix}#2\end{matrix}}}_{\text{$#1$}}}$}#3}

\newcommand{\myol}[2][3]{{}\mkern#1mu\overline{\mkern-#1mu#2}}
\newcounter{conditioncounter}
\newcommand*{\cond}[1]{\renewcommand{\theconditioncounter}{#1}\refstepcounter{conditioncounter}\textbf{(\theconditioncounter)}\label{cond:#1}}
\newcommand*{\condu}[1]{\renewcommand{\theconditioncounter}{#1}\refstepcounter{conditioncounter}\textbf{(\underline\theconditioncounter)}\label{condu:#1}}
\newcommand*{\condo}[1]{\renewcommand{\theconditioncounter}{#1}\refstepcounter{conditioncounter}($\overline{\mbox{\theconditioncounter}}$)\label{condo:#1}}
\newcommand*{\refcond}[1]{(\ref{cond:#1})}
\newcommand*{\refcondu}[1]{(\underline{\ref{condu:#1}})}
\newcommand*{\refcondo}[1]{($\myol{\mbox{\ref{condo:#1}}}$)}
\newcommand*{\refcondot}[1]{($\overline{\mbox{\ref{condo:#1}}}$)}

\newcommand{\T}{\mathbb{T}}
\newcommand{\N}{\mathbb{N}}
\newcommand{\hati}{\hat\imath}
\newcommand{\floor}[1]{\lfloor #1\rfloor}
\newcommand{\ceil}[1]{\lceil #1\rceil}
\newcommand{\norm}[1]{\left\|#1\right\|}

\newcommand{\normhs}[1]{\left\|#1\right\|_{\mathrm{HS}}}

\newcommand{\C}{\mathbb{C}}
\renewcommand{\epsilon}{\varepsilon}

\newcommand{\ux}{{\underline{x}}}
\newcommand{\uE}{{\underline{E}}}
\newcommand{\upsi}{{\underline{\psi}}}

\newcommand{\oE}{{\overline{E}}}
\newcommand{\uiota}{{\underline{\iota}}}
\newcommand{\ox}{{\overline{x}}}

\let\deg\relax
\DeclareMathOperator{\deg}{deg}
\DeclareMathOperator{\spec}{spec}

\newcommand{\cI}{\mathcal{I}}
\newcommand{\fC}{\mathfrak{C}}
\newcommand{\fd}{\mathfrak{d}}
\newcommand{\fE}{\mathfrak{E}}
\newcommand{\fG}{\mathfrak{G}}
\newcommand{\fc}{\mathfrak{c}}
\newcommand{\fg}{\mathfrak{g}}

\newcommand{\fT}{\mathfrak{T}}
\newcommand{\fH}{\mathfrak{H}}

\newcommand{\fe}{\mathfrak{e}}
\newcommand{\fn}{\mathfrak{n}}
\newcommand{\fp}{\mathfrak{p}}
\newcommand{\fq}{\mathfrak{q}}

\newcommand{\uN}{\underline N}

\newcommand{\uPi}{\underline \Pi}

\newcommand{\car}{\mathop{\rm{Car}}\nolimits}

\newcommand{\Car}{\mbox{\rm Car}}

\newcommand{\nn}{\nonumber}

\newcommand{\cB}{{\mathcal{B}}}
\newcommand{\cD}{{\mathcal{D}}}

\newcommand{\cG}{{\mathcal{G}}}

\newcommand{\cS}{{\mathcal{S}}}

\newcommand{\cP}{{\mathcal{P}}}

\newcommand{\IC}{{\mathbb{C}}}

\newcommand{\Res}{\mathop{\rm{Res}}\nolimits}

\renewcommand{\mod}{{\rm mod\ }}

\newcommand{\be}{\begin{eqnarray}}
\newcommand{\ee}{\end{eqnarray}}

\newcommand{\R}{{\mathbb R}}
\newcommand{\tor}{{\mathbb T}}
\newcommand{\Z}{{\mathbb Z}}

\newcommand{\dist}{\mbox{\rm dist}}
\newcommand{\mes}{{\rm mes}}

\newcommand{\Proj}{{\rm Proj}}

\renewcommand{\mod}{{\rm{mod}\, }}

\newtheorem{theorem}{Theorem}[section]
\newtheorem{lemma}[theorem]{Lemma}
\newtheorem{cor}[theorem]{Corollary}
\newtheorem{prop}[theorem]{Proposition}

\newtheorem{thma}{Theorem}

\newtheorem{thmb}{Theorem}

\newtheorem{thmc}{Theorem}

\newtheorem{thmd}{Theorem}

\newtheorem{thme}{Theorem}

\theoremstyle{definition}
\newtheorem{defi}[theorem]{Definition}
\newtheorem{definition}[theorem]{Definition}

\theoremstyle{remark}
\newtheorem{remark}[theorem]{Remark}

\numberwithin{equation}{section}

\begin{document}
\title[On the Spectrum of Multi-Frequency Quasiperiodic Operators]
{On the Spectrum of Multi-Frequency  Quasiperiodic Schr\"odinger Operators
with Large Coupling}
\date{}
\author{Michael Goldstein, Wilhelm Schlag, Mircea Voda}

\address{Department of Mathematics, University of Toronto, Toronto, Ontario, Canada M5S 1A1}
\email{gold@math.toronto.edu}

\address{Department of Mathematics, The University of Chicago, 5734 S. University Ave., Chicago, IL 60637, U.S.A.}
\email{schlag@math.uchicago.edu}

\address{Department of Mathematics, The University of Chicago, 5734 S. University Ave., Chicago, IL 60637, U.S.A. and Department of Mathematics, University of Toronto, Toronto, Ontario, Canada M5S 1A1}
\email{mvoda@uchicago.edu}

\thanks{The first author was partially supported by an NSERC grant. The second
  author was partially supported by the NSF, DMS-1500696.}

\begin{abstract} We study multi-frequency quasiperiodic
  Schr\"{o}dinger operators on $\Z$. We
  prove that for a large  real analytic potential satisfying certain restrictions the spectrum
  consists of a single interval. The result is a consequence
  of a criterion for the spectrum to contain an
  interval at a given location that we establish non-perturbatively in
  the regime of positive Lyapunov exponent.
\end{abstract}

\maketitle\tableofcontents

\section{Introduction}\label{sec:intro}
In the last 40 years after the groundbreaking paper ~\cite{DinSin75}
the theory of quasiperiodic Schr\"odinger operators has been developed
extensively, see the monograph \cite{Bou05} for an overview and
\cite{JitMar16} for a survey of the more recent results.  For shifts
on a one-dimensional torus $ \T $ most of the results have been
established non-perturbatively, i.e., either in the regime of almost
reducibility or in the regime of positive Lyapunov exponent, and
Avila's global theory, see \cite{Avi15}, gives a qualitative spectral
picture, covering both regimes, for generic potentials. One of the
main results of the one-dimensional theory is the fact that the
spectrum is a Cantor set. For the case of the almost Mathieu operator
(corresponding to a cosine potential), this result has been proved for
any non-zero coupling and any irrational shift, see 
\cite{Pui04} and \cite{AviJit09,AviJit10}. For general analytic
potentials in the regime of positive Lyapunov
exponent with generic shift the Cantor structure of the spectrum has
been obtained in \cite{GolSch11}.

On the other
hand, shifts on a multidimensional torus $ \T^d $ turned out to be
harder to analyze and the theory is less developed, even in the
perturbative setting.
In particular, not much is known about the geometry of the spectrum
for multidimensional shifts.  In their pioneering paper
\cite{ChuSin89}, Chulaevsky and Sinai conjectured that in contrast to
the shift on the one-dimensional torus, for the two-dimensional shift
the spectrum can be an interval for generic large smooth
potentials. In this paper we prove this conjecture for large analytic potentials.

Heuristically, gaps in the spectrum of the one-frequency
operators are created by horizontal ``forbidden
zones'' appearing at the points of  intersection of the graphs of shifted finite
scale eigenvalues parametrized by phase,
see~\cite{Sin87,GolSch11}. In contrast to this, the heuristic
principle underlying \cite{ChuSin89} is that for multiple frequencies, the intersection curves
of the  graphs of shifted finite scale 
eigenvalues may not be too flat, thus preventing the appearance of the  horizontal  ``forbidden
zones'' and stopping the formation of gaps.  It is clear that some genericity
assumption on the potential function is needed for this to be true,
since potentials like $V(x,y)=v(x)$ lead to flat intersection curves
and have Cantor spectrum. Furthermore, the largeness of the potential
is also needed. Indeed, it is known that for small potentials with
atypical frequency vector the spectrum has gaps, see \cite{Bou02}. 

Implementing such an argument, appears to be very challenging for a
number of reasons. First, the analytical techniques available in
finite volume are less favorable (mainly the large deviation theorems
and everything that depends on them) as compared to the
case of one frequency. In particular, it is difficult to implement an
approach based on finite scale localization as in
\cite{GolSch11}. This is due to the fact that it is hard to handle
long chains of resonances and to control the intersections of the resonant
curves with the level sets of the eigenvalues. Second, it is
inevitable that the intersection curves of the graphs of shifted
finite scale eigenvalues flatten near the absolute extrema and
handling this situation seems to be a delicate matter.


In \cite{GolSchVod16} we addressed some of the issues regarding the
analytical techniques, including establishing finite scale
localization. We will use most of the basic tools from
\cite{GolSchVod16}. However, for the purpose of this paper one would
need a refined version of finite scale localization, beyond what is achieved in that paper.
We analyze the spectrum of the operator $H_N(x)$, $x\in \tor^d$, on a finite interval $[1,N]$ subject to Dirichlet boundary conditions.
To keep this spectrum under control requires resolving the following problem.
Given $E$ let $\mathcal{R}_N(E)$ be the set of all phases $x$ such that $E$ is in the spectrum of the 
operator $H_N(x)$. One has to identify phases $x\in \mathcal{R}_N(E)$ for which $x+n\omega$ is not too close to $\mathcal{R}_N(E)$
as $n$ runs in the interval $N\ll n< N^A$, $A\gg 1$. This issue, commonly referred to as {\em double resonances},  is well-known. 
Similar strategies, leading to the formation of intervals in the
spectrum, have been implemented for the skew-shift
in~\cite{Kru12} and for continuous two-dimensional Schr\"odinger
operators in~\cite{KarSht14}.
The main new device that we develop in this work,  consists of an elimination of double resonances for \emph{all}
shifts $x+h$, and not just the ``arithmetic ones'' $x+n\omega$. Of course the shift $h$ cannot be too small.
Although this problem looks  less accessible, it  turns out to provide more control on the  resonant set $\mathcal{R}_N(E)$ of the previous scale. 
The level sets $V(x)=E$ of the potential in question must satisfy the requirements of this more general elimination in order to launch the multi-scale analysis.
This is exactly the origin of our main condition on the potential, see Definition~\ref{defi:genericU} below. 

Furthermore, in order to show that the spectrum is actually an
interval,  we develop a Cartan type estimate that
controls the intersections of the level sets of an analytic function
near a non-degenerate extremum with their shifts.


The core of our approach is \emph{non-perturbative} and works in the
regime of positive Lyapunov exponent. More precisely, we develop two
non-perturbative inductive schemes, one leading to the formation of
intervals in the bulk of the spectrum and the other leading to
intervals at the edges of the spectrum. We will only use the largeness
of the potential to check that the initial inductive conditions are satisfied.

We introduce some notation and definitions that we need to state our
main result. We work with operators 
\begin{equation}\label{eq:H-lambda}
  [H_\lambda (x)\psi](n)= -\psi(n+1)-\psi(n-1) + \lambda V(x+n\omega)\psi(n),
\end{equation}
with $\lambda>0$ being a real parameter, and with the potential $V$ a real analytic function on
the torus $\tor^d$, $ \T=\R/\Z $, $ d\ge 2 $. We 
assume that the frequency vector $ \omega\in \T^d $ obeys
the standard Diophantine condition
\begin{equation}\label{eq:vecdiophant}
  \|k\cdot\omega\|\ge\frac{a}{|k|^{b}}\qquad\text{for all nonzero }k\in\Z^d,
\end{equation}
where $ a>0 $, $ b>d $ are some constants, $ \norm{\cdot} $ denotes
the usual norm on $\tor$, and $ |\cdot| $ denotes the sup-norm on
$ \Z^d$.  Unless otherwise stated, throughout the paper $ a,b $ will refer to the constants
from \cref{eq:vecdiophant}. In this paper we don't use
elimination of frequencies and our results apply to any Diophantine
frequency $ \omega $. To simplify notation, we omit dependence on $
\omega $ from notation whenever possible. The dependence on frequency
will still be reflected by having some of the constants depend on $
a,b $.

\begin{defi}\label{defi:genericU}
  We let $ \fG $ be the class of real-analytic functions $ V $ on
  $ \T^d $, $ d\ge 2 $, for which there exist constants
  $ \fc_0=\fc_0(V,d)\in(0,1) $, $ \fc_1=\fc_1(d)\in(0,1) $,
  $ \fC_0=\fC_0(V,d)> 1 $, such that the following properties hold.

  \medskip\noindent (i) $ V $ is a Morse function, i.e., all its critical
  points are non-degenerate.

  \medskip\noindent (ii) $ V $ attains each global extremum at just one point.

  \medskip\noindent (iii) Given $ h\in \T^d $, let
  \begin{equation*}
	g_{V,h,i,j}(x)= \det \begin{bmatrix}
      \partial_{x_i} V(x) & \partial_{x_j} V(x)\\
      \partial_{x_i} V(x+h) & \partial_{x_j} V(x+h)
    \end{bmatrix}.
  \end{equation*}
  For any $ i\neq j $,
  $K\ge \fC_0$, and any $\|h\|\ge \exp(-\mathfrak{c}_0K)$ we have
  \begin{equation*}
    \mes \{x_{\hati}\in \T^{d-1}: \min_{x_i}\left( |V(x+h)-V(x)|+|g_{V,h,i,j}(x)|
    \right)<\exp(-K)\}\le \exp(-K^{\fc_1}),
  \end{equation*}
  where $ x_{\hati}=(x_1,\dots,x_{i-1},x_{i+1},\dots,x_d) $.

  \medskip\noindent (iv) For any $ i $, $ K\ge \fC_0 $,
  $ \eta\in \R $, and $ h_0\in \R^d $, $ \norm{h_0}=1 $, we have
  \begin{gather*}
    \mes \{x_{\hati}\in \T^{d-1}:\min_{x_i} \left( |V(x)-\eta|+|\langle \nabla V(x),h_0 \rangle| \right)
    <\exp(-K)\}\le \exp(-K^{\fc_1}).
  \end{gather*}
\end{defi}

Recall that $ \spec H_\lambda(x) $ is known not to depend on the
phase. We will use the notation $ \cS_\lambda:=\spec H_\lambda (x) $.

\begin{thma}\label{thm:A}
  There exists $\lambda_0=\lambda_0(V,a,b,d)$ such that the following
  statements hold for
  $\lambda\ge \lambda_0$.

  \medskip\noindent
  (a) Assume that $ V $ attains its global minimum at exactly one
  non-degenerate critical point $ \ux $. Then there exists $ \uE\in \R
  $, $ |\lambda^{-1}\uE-V(\ux)|<\lambda^{-1/4} $, such that
  \begin{equation*}
	[\uE,\uE+\lambda\exp(-(\log\lambda)^{1/2})]\subset
    \cS_\lambda\quad \text{ and } \quad (-\infty,\uE)\cap \cS_\lambda=\emptyset.
  \end{equation*}
  An analogous statement holds relative to the global maximum of $ V $
  (using the notation $ \ox,\oE $).

  \medskip\noindent
  (b)  Assume that $V\in \mathfrak{G}$ and let $ \uE,\oE $ be as in
  (a). Then $ \cS_\lambda=[\uE,\oE] $.
\end{thma}

\begin{remark}\label{rem:thmAquantify}
  (a) The constant $ \lambda_0(V,a,b,d) $ can be expressed explicitly,
  see the proof of \cref{thm:A}.

  \smallskip\noindent
  (b)
  The genericity of the assumptions on $ V $ will be addressed in
  \cite{Vod17}. More precisely, the following result will be established. Consider  real trigonometric
  polynomials of the form
  \begin{equation*}
    V(x)=\sum_{m\in \mathbb{Z}^d:|m|\le n} c_{m}e^{2\pi i m\cdot x},\quad x\in \mathbb{R}^d
  \end{equation*}
  of a given cumulative degree $n\ge 1$,
  $|m|:=\sum_{1\le j\le d}|m_j|$.  Then for almost all vectors
  $(c_m)_{|m|\le n}$ one has $V\in \mathfrak{G}$.  

  \smallskip\noindent (c) For the completeness of our paper we include
  a particular example of potential $ V\in \fG $ that can be obtained
  by the  methods from \cite{Vod17}. Namely, in \cref{sec:example}, we
  show that
  \begin{equation*}
    V(x,y)=\cos(2\pi x)+s\cos(2\pi y)
  \end{equation*}
  satisfies the
  assumptions of \cref{defi:genericU} for all $ s\in \R\setminus \{
  -1,0,1 \} $. We note that as $ s $ approaches $ \{ -1,0,1 \} $ our
  explicit value for $ \lambda_0 $ diverges to $ \infty $ and the
  geometry of the spectrum cannot be decided by continuity. Of course,
  for $ s=0 $ the spectrum is a Cantor set. However, for $ s=\pm 1 $,
  part (a) of \cref{thm:A} still applies and guarantees the existence
  of intervals at the edges of the spectrum.
  
  \smallskip\noindent (d) The measure estimates from conditions
  (iii) and (iv)  of \cref{defi:genericU} are Cartan type estimates (see
  \cref{sec:Cartan}). We note that one cannot apply Cartan's estimate
  directly to the functions from this conditions. Instead, the
  estimates can be obtained by applying Cartan's estimate to some
  resultants associated with these functions, see \cref{sec:example}.
\end{remark}

\smallskip As mentioned above, the derivation of \cref{thm:A} is based on two
non-perturbative statements in the regime of positive Lyapunov
exponent. Namely, \cref{thm:B} produces an interval in the spectrum in
the vicinity of a spectral value at which certain finite scale
conditions hold, and \cref{thm:C} shows that the spectrum is an
interval under certain additional finite scale conditions. Since the
conditions are rather technical, we do not state the theorems
here. Their statements can be found in \cref{sec:main-thm}.
The inductive conditions and their corresponding inductive theorems
are discussed
in \cref{sec:bulk} (see \cref{thm:D}) and \cref{sec:edges} (see \cref{thm:E}). In \cref{sec:A-to-DE} we show
how these conditions hold at large coupling, given a potential as in
\cref{thm:A}. Throughout the paper we will employ the basic tools
discussed in \cref{sec:basic-tools} for the non-perturbative regime
and in \cref{sec:perturbative-refinements} for large coupling. The
Cartan type estimate that we use to handle the edges of the spectrum
is discussed in \cref{sec:Cartan-Morse}.

\section{Basic Tools}\label{sec:basic-tools}

In this section we discuss some basic results that we will use
throughout the paper. The results will apply to a family of discrete
Schr\"odinger operators,
\begin{equation}  \label{eq:schr100}
  [H(x)\psi](n)= -\psi(n+1)-\psi(n-1)+ V(x+n\omega)\psi(n).
\end{equation}
with $ V $ real-analytic on $ \T^d $ and $ \omega $ as in
\cref{eq:vecdiophant}. Note that we omit the coupling constant
$ \lambda $ because the results of this section are non-perturbative.
We also assume that $ V $ extends complex analytically to
\begin{equation*}
  \tor^d_\rho:=\{x+iy : x\in\tor^d,\quad y\in \mathbb{R}^d,\quad |y|<\rho \},
\end{equation*}
with some $ \rho>0 $. Note that we use $ |\cdot| $ to denote the
sup-norm on $ \R^d $ and $ \norm{\cdot} $ to denote the Euclidean norm
on $ \R^d $. At the same time when we apply it to shifts on $ \T^d $, $
\norm{\cdot} $ will stand for the usual norm on $ \T^d $. It is well-known that for any real-analytic
function on $ \T^d $, such $ \rho=\rho(V) $ exists. To simplify some
later estimates we also assume $ \rho\le 1 $. Throughout the
paper, with the exception of \cref{sec:Cartan-Morse}, we reserve $
\rho $ for this constant.

We recall some standard notation. Given an interval $ [a,b]\subset \Z $, the transfer matrix is defined by
\begin{equation*}
  M_{[a,b]}(x,E)= \prod_{n=b}^a \begin{bmatrix}
	V(x+n\omega)-E & -1 \\
    1 & 0
  \end{bmatrix}.
\end{equation*}
We let $H_{[a,b]}(x)$ be the restriction of $H(x)$ to
the interval $[a,b]$ with Dirichlet boundary conditions and we
denote the corresponding Dirichlet determinant by $ f_{[a,b]}(x,E):= \det
(H_{[a,b]}(x)-E) $. We use $ E^{[a,b]}_j(x) $, $
\psi_j^{[a,b]}(x,\cdot) $ to denote the eigenpairs of $
H^{[a,b]}(x) $, with $ \psi_j^{[a,b]}(x,\cdot) $ being $
\ell^2 $-normalized. The
transfer matrix is related to the Dirichlet determinants
through the following  formula
\begin{equation}\label{eq:M-f}
  M_{[a,b]}(x,E) = \begin{bmatrix}
    f_{[a, b]} (x,E) & - f_{[a+1, b]} (x,E)\\
    f_{[a, b-1]} (x,E) & - f_{[a+1, b-1]}(x,E)
 \end{bmatrix}.
\end{equation}
We let $ M_N:=M_{[1,N]} $, $ H_N:=H_{[1,N]} $, $ f_N:=f_{[1,N]} $.
The {\em Lyapunov exponent} is defined by
\begin{equation*}
  L(E)=\lim_{N\to\infty} L_N(E)= \inf_{N} L_N(E)
  ,\quad L_N(E) = \frac{1}{N} \int_{\tor^d} \log\|M_N(x,E)\|\,dx.
\end{equation*}

Most of the results in this section {\em do not use} the fact that $V$
assumes only real values on the torus $\mathbb{T}^d$ and therefore
they also hold on $ \T^d+iy $, $ |y|<\rho/2 $, by replacing $ V $ with
$ V(\cdot+iy)$. In particular, this applies to all the results up to
and including \cref{cor:4.6zeros}. Of course, when we change the
potential, we also need to adjust the Lyapunov exponents. To this end
we define
\begin{equation}\label{eq:lapyomega}
\begin{split}
  L_N(y,E) &= \frac{1}{N}\int_{\mathbb{T}^d}  \log \|M_N(x+iy,E)\|\,dx, \\
  L(y,E) &= \lim_{N\to \infty} L_N(y,E).
\end{split}
\end{equation}

We will use some standard conventions. Unless stated otherwise,
the constants denoted by $ c,C $ might have different values each time
they are used. We let $ a\lesssim b $ denote
$ a\le Cb $ with some positive $ C $, $ a\ll b $ denote
$ a\le C b $ with a sufficiently large positive $ C $, and
$ a\simeq b $ stand for $ a\lesssim b $ and $ b\lesssim a $. 
It will be clear from the context what the implicit constants are allowed to
depend on. To emphasize the dependence on some parameter we may use it
as a subscript for the above symbols (e.g., $ a\simeq_d b $).

Our constants will depend on $ \omega $, $ V $, $ E $, $ d $, and $
\gamma $, where $ \gamma>0 $ will stand for a lower bound on the
Lyapunov exponent. The dependence on $ \omega $ will be through the
parameters $ a,b $ from \cref{eq:vecdiophant}. The dependence on $ V $
will be through $ \rho $ and
\begin{equation*}
  \norm{V}_\infty:=\sup \{ |V(z)|: z\in \T_{3\rho/4}^d \}.
\end{equation*}
The dependence on $ E $ will be uniform on bounded sets. In most cases
we leave the dependence on $ d $ implicit and, unless stated otherwise, all constants may depend on the
dimension $ d $.

When we work in the perturbative setting we will need to replace $ V $
by $ \lambda V $ and we will need explicit knowledge of the dependence
on $ \lambda $. This means that we need to keep track explicitly of
the dependence on $ \norm{V}_\infty $, $ E $ (because the range of
energies we need to consider depends on $ V $), and $ \gamma $ (note
that $ \rho $ remains unchanged when we introduce the coupling
constant). To this end we will use the quantity
\begin{equation*}
  S_{V,E}:=\log(3+\norm{V}_\infty+|E|).
\end{equation*}
This definition is motivated by the fact that
\begin{equation*}
  \norm{\begin{bmatrix}
      V(x+n\omega)-E & -1\\
      1 & 0
    \end{bmatrix}}\le 1+\norm{V}_\infty+|E|
\end{equation*}
and therefore
\begin{gather}
  \label{eq:monodr1}
  0\le \log \|M_N(x,E)\|\le N\log(1+\norm{V}_\infty+|E|),\\
  \label{eq:LEupperb1}
  0\le L_N(E)\le \log(1+\norm{V}_\infty+|E|).
\end{gather}
The choice of the absolute constant in  the definition of $ S_{V,E} $
is for the convenience of having $ S_{V,E}>1 $.
Since
\begin{equation*}
  \spec H_N(x)\subset [-2-\norm{V}_\infty,2+\norm{V}_\infty],
\end{equation*}
it will actually be enough to work with $ |E|\le \norm{V}_\infty+4
$ and when we want to suppress the dependence on $ E $ we will use
\begin{equation}\label{eq:SV}
  S_V:=\log(3+\norm{V}_\infty).
\end{equation}
Note that $ S_{V,E}\simeq S_V $ for $ |E|\le \norm{V}_\infty+4 $.

We will make repeated use of the observation that using the mean value
theorem and Cauchy estimates, we have
\begin{equation}\label{eq:stability-x}
  |E_j^{[a,b]}(x)-E_j^{[a,b]}(x_0)|\le \norm{H_{[a,b]}(x)-H_{[a,b]}(x_0)}
  \le C_\rho\norm{V}_\infty |x-x_0|.
\end{equation}
We will also use the following basic identity
\begin{equation}\label{eq:basic-identity}
  \spec  H_{m+[a,b]}(x)=\spec  H_{[a,b]}(x+m\omega).
\end{equation}

\subsection{Large Deviations Estimates}
We recall the Large Deviations Theorem (LDT) for the transfer
matrix. We refer to \cite{Bou05} and \cite{GolSch01} for two different
approaches to its proof. The particular formulation we give here is
based on \cite{GolSch01} (see Corollary 9.2 therein).

\begin{theorem}\label{thm:anyLDT}
  Assume $ E\in \C $.
  There exist $\sigma =\sigma(a,b)$, $\tau=\tau(a,b)$,
  $ \sigma,\tau\in (0,1) $, $ C_0=C_0(a,b,\rho) $,
  such that for $N\ge 1$ one has
  \begin{equation*}
    \mes \left\{ x\in\tor^d : |\log\|M_N(x,E)\|-NL_N(E)| > C_0S_{V,E}N^{1-\tau} \right\}
    < \exp(-N^{\sigma}).
  \end{equation*}
\end{theorem}

In \cite{GolSch08} it was shown (see Proposition 2.11 therein) that
in the the regime of positive Lyapunov exponent, the large deviations
estimate extends to the entries of the transfer matrix.

\begin{theorem}\label{thm:DirLDT}
  Assume $ E\in \C $, and $ L(E)> \gamma >0 $.
  There exist $\sigma =\sigma(a,b)$, $\tau=\tau(a,b)$,
  $ \sigma,\tau\in (0,1) $,
  such that for $N\ge N_0(V,a,b,E,\gamma)$ one has
  \begin{equation*}
    \mes \left\{ x\in\tor^d : |\log |f_N(x,E)|-NL_N(E)| > N^{1-\tau} \right\}
    < \exp(-N^{\sigma}).
  \end{equation*}
\end{theorem}

Note that the large deviations estimates also hold with any other
smaller choices of the actual exponents $ \sigma,\tau $.
The sharpness of these exponents plays no role for
us, so we will also assume without loss of generality that the exponents are the same in both
statements and  $\sigma\ll\tau\ll 1$.

We claim that by inspecting the proof from \cite{GolSch08} it can be
seen that the constant $ N_0 $ from \cref{thm:DirLDT} can be chosen to
be $ (S_{V,E}+ \gamma^{-1})^{C} $, $ C=C(a,b,\rho) $. In fact, all
the large constants in our statements can be chosen of this form
(though not optimally). Since the proof in \cite{GolSch08} is quite
lengthy and intricate, and we only need to be explicit about $ N_0 $
in the perturbative setting, we will give a simpler proof of the (LDT)
for determinants at large coupling in \cref{sec:perturbative-refinements}.

The usefulness of the (LDT) is enhanced by the following result, known
as the Avalanche Principle.
\begin{prop}[{\cite[Prop.~2.2]{GolSch01}}]
  \label{prop:AP}
  Let $A_1,\ldots,A_n$ be a sequence of  $2\times 2$--matrices whose determinants satisfy
  \begin{equation}
    \label{eq:detsmall}
    \max\limits_{1\le j\le n}|\det A_j|\le 1.
  \end{equation}
  Suppose that
  \be
  &&\min_{1\le j\le n}\|A_j\|\ge\mu>n\mbox{\ \ \ and}\label{large}\\
  &&\max_{1\le j<n}[\log\|A_{j+1}\|+\log\|A_j\|-\log\|A_{j+1}A_{j}\|]<\frac12\log\mu\label{diff}.
  \ee
  Then
  \begin{equation}
    \Bigl|\log\|A_n\ldots A_1\|+\sum_{j=2}^{n-1} \log\|A_j\|-\sum_{j=1}^{n-1}\log\|A_{j+1}A_{j}\|\Bigr|
    < C\frac{n}{\mu}
    \label{eq:AP}
  \end{equation}
  with some absolute constant $C$.
\end{prop}

To apply the Avalanche Principle one needs to be in the positive
Lyapunov exponent regime and to be able to compare the Lyapunov
exponents $ L_N $ at different scales. This can be achieved through
the following result.
\begin{prop}[{\cite[Lem.~10.1]{GolSch01}}]\label{prop:uniform}
  Assume $ E\in \C $, and
  $L(E)>\gamma>0$.  Then for any $ N\ge 2 $,
  \[ 0\le L_N(\omega,E)-L(\omega,E)< C_0\frac{(\log
      N)^{1/\sigma}}{N}\]
  where $C_0=C_0(V,a,b,E,\gamma)$ and $ \sigma $ is as in (LDT).
\end{prop}

The constant $ C_0 $ from the previous proposition can be evaluated
explicitly by inspecting its proof in \cite{GolSch01}. However, we
will obtain an explicit perturbative version of this result in
\cref{sec:perturbative-refinements}.

The remaining results that we state without proof in this section are proved in 
\cite{GolSchVod16}. The specific constants from their statements are
obtained by a simple inspection of the proofs in
\cite{GolSchVod16}. Note that in the choice of constants we favour
simplicity over sharpness. Some of the constants will depend on the
constants $ N_0 $ from \cref{thm:DirLDT} and $ C_0 $ from \cref{prop:uniform}. To keep
track of this we fix
\begin{equation}\label{eq:B0}
  B_0:=N_0+ C_0.
\end{equation}

As a consequence of the (LDT) and the submean value property for
subharmonic functions one gets the following uniform upper estimate.
\begin{prop}[{\cite[Prop.~2.13]{GolSchVod16}}]\label{prop:logupper}
  Let $ E\in \C $ and $ \tau $ as in (LDT). Then for all $N\ge 1$,
  \begin{equation*}
    \sup\limits_{x\in\tor^d} \log \| M_N (x,E) \| \le
    NL_N(E)+ C_0 S_{V,E} N^{1-\tau},
  \end{equation*}
  with $ C_0=C_0(a,b,\rho) $.
\end{prop}

To extend the uniform upper estimate to a complex neighborhood of $
\T^d $ we need the following result.
\begin{lemma}[{\cite[Cor.~2.12]{GolSchVod16}}]\label{cor:liplap}
  Let $ E\in \C $. For any $ N\ge 1 $ we have
  \begin{equation*}
    |L_N(y,E) - L_N(E)|
    \le C_\rho S_{V,E}\sum_{i=1}^d |y_i|.
  \end{equation*}
  In particular, the same bound holds with $ L $ instead of $ L_N $.
\end{lemma}

\begin{cor}\label{cor:logupper}
  Let $ E\in \C $ and $ \tau $ as in (LDT). Then for all $N\ge 1$ and
  all $ y\in \R^d $, $ |y|<\min(\rho/2,1/N) $,
  \begin{equation}\label{eq:uniform-estimate}
    \sup\limits_{x\in\tor^d} \log \| M_N (x+iy,E) \| \le
    NL_N(E)+ C_0 S_{V,E} N^{1-\tau},
  \end{equation}
  with  $ C_0=C_0(a,b,\rho) $. In particular we also have
  \begin{equation*}
    \sup\limits_{x\in\tor^d} \log | f_N (x+iy,E) | \le
    NL_N(E)+ C_0 S_{V,E} N^{1-\tau}.
  \end{equation*}
\end{cor}
\begin{proof}
  The conclusion follows by applying \cref{prop:logupper} with $
  V(x+iy) $ instead of $ V(x) $ and by using \cref{cor:liplap}.
\end{proof}

Next we recall a way of obtaining off-diagonal decay for Green's
function. We use the notation $ \cG_{[a,b]}(x,E):=(H_{[a,b]}(x)-E)^{-1} $.
\begin{lemma}[{\cite[Lem. 2.24]{GolSchVod16}}]\label{lem:Green}
  Assume $ x_0\in \tor^d $,
  $ E_0\in \C $, and $ L(E_0)> \gamma >0 $.  Let $ K\in \R $
  and $ \tau $ as in (LDT). There exists
  $ C_0=C_0(a,b,\rho) $ such that if
  $ N\ge (B_0+ S_{V,E_0}+ \gamma^{-1})^{C} $, $ C=C(a,b,\rho) $, and
  \begin{equation}\label{eq:fN_unter}
    \log \big | f_N(x_0,E_0) \big | > NL_N(\omega_0,E_0) - K,
  \end{equation}
  then for any $ (x,E)\in \T^d\times \C $ with
  $|x - x_0|,\ |E - E_0| < \exp(-(K+C_0N^{1-\tau}))$, we have
  \begin{align} \label{eq:Gjk}
    \big | \cG_{[1, N]} (x,E;j,k)\big | & \le \exp\left(-
               \frac{\gamma}{2}|k - j|+K+2C_0N^{1-\tau}\right), \\
    \big \| \cG_{[1, N]} (x,E) \big \| & \le
                                                 \exp(K+3C_0N^{1-\tau}). \label{eq:normG}
  \end{align}
\end{lemma}

\subsection{Cartan's Estimate}\label{sec:Cartan} We recall the
definition of Cartan sets from \cite{GolSch08}. We use the notation $
\cD(z_0,r)=\{ z\in \C: |z-z_0|<r \} $.
\begin{defi}\label{defi:cartansets} Let $H \ge 1$.  For an arbitrary set $\cB \subset \cD(z_0,
1)\subset \IC$ we say that $\cB \in \car_1(H, K)$ if $\cB\subset
\bigcup\limits^{j_0}_{j=1} \cD(z_j, r_j)$ with $j_0 \le K$, and
\begin{equation}
\sum_j\, r_j < e^{-H}\ .
\end{equation}
If $d\ge 1$ is an  integer and $\cB \subset
\prod\limits_{j=1}^d \cD(z_{j,0}, 1)\subset \IC^d$, then we define
inductively that $\cB\in \car_d(H, K)$ if for any $1 \le j \le d$ there
exists $\cB_j \subset \cD(z_{j,0}, 1)\subset \IC, \cB_j \in \car_1(H,
K)$ so that $\cB_z^{(j)} \in \car_{d-1}(H, K)$ for any $z \in \IC
\setminus \cB_j$,  here $\cB_z^{(j)} = \left\{(z_1, \dots, z_d) \in \cB:
z_j = z\right\}$.
\end{defi}

The definition is motivated by the following
generalization of the usual Cartan estimate to several variables. Note
that given a set $ S $ that has a centre of symmetry, we will let
$ \alpha S $, $ \alpha>0 $, stand for the set scaled with respect to
its centre of symmetry.

\begin{lemma}[{\cite[Lem.~2.15]{GolSch08}}]\label{lem:high_cart}
 Let $\varphi(z_1, \dots, z_d)$ be an analytic function defined
on a polydisk $\cP = \prod\limits^d_{j=1} \cD(z_{j,0}, 1)$, $z_{j,0} \in
\IC$.  Let $M \ge \sup\limits_{z\in\cP} \log |\varphi(z)|$,  $m \le
\log|\varphi(z_0) |$,
$z_0 = (z_{1,0},\dots, z_{d,0})$.  Given $H
\gg 1$ there exists a set $\cB \subset \cP$,  $\cB \in
\car_d\left(H^{1/d}, K\right)$, $K = C_d H(M - m)$,  such that
\begin{equation}\label{eq:cart_bd}
  \log | \varphi(z)| > M-C_d H(M-m)
\end{equation}
for any $z \in \frac{1}{6}\cP\setminus \cB$.
Furthermore, when $ d=1 $ we can take
$ K=C(M-m) $ and keep only the disks of $ \cB $ containing a zero of $\phi$ in them.
\end{lemma}

We note that the definition of the Cartan sets gives implicit information
about their measure.
\begin{lemma}\label{lem:Cartan-measure}
  If  $ \cB\in \Car_d(H,K) $ then
  \begin{equation*}
    \mes_{\C^d}(\cB)\le C(d) e^{-H} \text{ and } \mes_{\R^d}(\cB\cap \R^d)\le C(d) e^{-H}.
  \end{equation*}
\end{lemma}
\begin{proof}
  The case $ d=1 $ follows immediately from the definition of
  $ \Car_1 $. The case $ d>1 $ follows by induction, using Fubini and
  the definition of $ \Car_d $.
\end{proof}

The following simple corollary of the Cartan estimate will allow us to
upgrade estimates from $ \T^d $, where we can take advantage of the
fact that $ H(x) $ is self-adjoint, to some complex neighborhood of $ \T^d $.
\begin{cor}\label{cor:high_cart}
  Let $\varphi(z_1, \dots, z_d)$ be an analytic function defined
  on a polydisk $\cP = \prod\limits^d_{j=1} \cD(x_{j,0},1)$, $x_{j,0} \in
  \R $.  Assume $\sup_\cP \log |\varphi(z)|\le 0$ and
  $\log
  |\varphi(x)|\le m<0$ for any $x\in \cP\cap \R^d$ .
  Then for any $ z\in \frac{1}{24}\cP $,
  \begin{equation*}
	\log
  |\varphi(z)|< c_0m, 
  \end{equation*}
  with some $c_0\ll_d 1$.
\end{cor}
\begin{proof} Assume, to the contrary, that there exists
  ${z}_0=(z_{j,0})$, $|{z}_0-x_0|<1/24$, such
  that $\log \bigl |\varphi({z}_0)\bigr |\ge c_0m$, with $ c_0 $
  to be specified later.
  Take $H\gg 1$ and find
  $\cB \subset \prod\limits^d_{j=1} \cD(x_{j,0},1/2)$,
  $2(\cB-z_0) \in \car_d\left(H^{1/d}, K\right)$,
  $K = c_0C_d H|m|$, such that
  \begin{equation}\label{eq:cart_bd11}
    \log \bigl | \varphi(z)\bigr | > -c_0C_d H|m|
  \end{equation}
  for any $z \in \prod^d_{j=1} \cD(z_{j,0},1/12)\setminus \cB$.
  Note that since $|z_0-x_0|<1/24$,
  \begin{gather*}
    \mes_{\mathbb{R}^d}\big(\prod^d_{j=1} \cD(z_{j,0},1/12)
    \cap \mathbb{R}^d \big)\ge  c_1(d),\ c_1>0.
  \end{gather*}
  On the other hand
  \begin{gather*}
    \mes_{\R^d}\big(\cB\cap \R^d\big)\le C(d) \exp(-H^{\frac{1}{d}})\ll c_1,
  \end{gather*}
  provided $H\gg 1$. So, there exists $x \in\big( \prod^d_{j=1} \cD(z_{j,0},1/12)\setminus \cB\big)
  \cap \R^d $. This implies $\log \bigl | \varphi(x)\bigr | > -c_0C_d H|m|>{m \over 2}$,
  provided we choose $c_0\ll 1$ appropriately. This contradicts our assumptions.
\end{proof}

Another simple consequence of Cartan's estimate is the following
statement that we refer to as the \emph{spectral form} of (LDT).
\begin{cor}[{\cite[Cor.~2.21]{GolSchVod16}}]\label{cor:4.6zeros}
  Assume $ x\in \T^d $, $ E\in \C $, and $ L(E)> \gamma >0 $. Let
  $ \sigma,\tau $ as in (LDT). If $ N\ge (B_0+ S_{V,E})^C $, $ C=C(a,b,\rho) $, and
  \begin{equation*}
    \norm{(H_N(x)-E)^{-1}}\le \exp(N^{\sigma/2}),
  \end{equation*}
  then
  \begin{equation*}
    \log|f_N(x,E)|> NL_N(E)-N^{1-\tau/2}.
  \end{equation*}
\end{cor}

\subsection{Poisson's Formula}

Recall that for any solution $ \psi $ of the difference equation $ H(x)\psi=E\psi $,
Poisson's formula reads
\begin{equation}\label{eq:poissonC}
  \psi(m) = \cG_{[a, b]} (x,E;m, a)\psi(a-1) + \cG_{[a, b]} (x,E;m,b)\psi(b+1),
  \quad m \in [a, b].
\end{equation}
With the help of Poisson's formula one  gets the following {\em covering lemma}.
\begin{lemma}[{\cite[Lem. 2.22]{GolSchVod16}}]\label{lem:Poissoncover}
  Let $ x\in \mathbb{T}^d$, $ E\in\mathbb{R} $, and
  $ [a,b]\subset \mathbb{Z} $. If for any $ m\in[a,b] $, there exists an interval
  $ I_m=[a_m,b_m]\subset[a,b] $ containing $m$ such that
  \begin{equation}\label{eq:Poissoncover-condition}
    (1- \delta_{a,a_m}) \left| \cG_{I_m}(x,E;m,a_m) \right|
    +(1- \delta_{b,b_m}) \left| \cG_{I_m}(x,E;m,b_m) \right|
    <1,
  \end{equation}
  then $ E\notin \spec H_{[a,b]}(x) $ (here $ \delta_{\cdot,\cdot} $
  stands for the Kronecker delta).
\end{lemma}

We refer to the next result as the {\em covering form}  of (LDT).
\begin{lemma}[{\cite[Lem. 2.25]{GolSchVod16}}]\label{lem:Greencoverap0}
  Assume $ N\ge 1 $, $ x_0\in \tor^d $, $ E_0\in \R $, and
  $ L(E_0)> \gamma >0 $.  Let $ \sigma,\tau $  as in (LDT).
  Suppose that for each point $ m\in [1,N] $ there exists an interval $ I_m\subset [1,N] $
  such that:
  \begin{enumerate}
  \item $ \dist(m,[1,N]\setminus I_m)\ge |I_m|/100 $,
  \item $ |I_m|\ge (B_0+ S_{V,E_0}+ \gamma^{-1})^C $, $ C=C(a,b,\rho) $,
  \item $ \log|f_{I_m}(x_0,E_0)|> |I_m|L_{|I_m|}(E_0)-|I_m|^{1-\tau/4} $.
  \end{enumerate}
  Then for any $ (x,E)\in \T^d\times \C $ such that
  \begin{equation*}
    |x-x_0|,\ |E-E_0|< \exp(- 2 \max_m |I_m|^{1-\tau/4}),
  \end{equation*}
  we have
  \begin{equation*}
    \dist(E,\spec H_N(x))\ge \exp(- 2 \max_m |I_m|^{1-\tau/4}).
  \end{equation*}
\end{lemma}

We also give another formulation of the covering form of (LDT) that is
better suited for the setting of this paper.

\begin{lemma}\label{lem:Greencoverap1}
  Assume $ x_0\in \tor^d $, $ S\subset \R $, and $ L(E)>\gamma>0 $
  for $ E\in S $.  Let $ \sigma $ as in (LDT), and $ a<b
  $ integers.
  Suppose that for each point $ m\in [a,b] $ there exists an interval $ J_m $
  such that $ m\in J_m $ and:
  \begin{enumerate}
  \item $ \dist(m,\partial J_m)\ge |J_m|/100 $,
  \item $\dist (\spec H_{J_m}(x_0),S)\ge \exp(-K)$,
    with  $K< \frac{1}{2} \min_m |J_m|^{\sigma/2}$,
  \item $ K\ge (B_0+ S_V+ \gamma^{-1})^C $, $
    C=C(a,b,\rho) $ (here $ a,b $ are as in \cref{eq:vecdiophant}),
  \end{enumerate}
  Let $ J=\bigcup_{m\in [a,b]} J_m $.
  Then for any $ |x-x_0|<\exp(-2K) $ 
  we have
  \begin{equation*}
    \dist(\spec H_J(x),S)\ge \frac{1}{2}\exp(-K).
  \end{equation*}
\end{lemma}
\begin{proof}
  It is enough to consider the case $ S=\{ E_0 \} $ because the full
  result follows by applying this particular case to each $ E_0\in S
  $. Furthermore, we can assume $ |E_0|\le \norm{V}_\infty+4 $,
  because otherwise the conclusion holds trivially.

  First we need to set up some intervals for which we will be able to apply
  the covering lemma.  Let $ J_m=[c_m,d_m] $. Then
  \begin{equation*}
	J=[c,d],\quad c=\inf_m c_m,\quad d=\sup_m d_m.
  \end{equation*}
  Let
  \begin{equation*}
	m_-=\sup \{ m\in[a,b] : c_m=c \},\quad m_+=\inf \{ m\in[a,b] : d_m=d \},
  \end{equation*}
  \begin{equation*}
	I_m=\begin{cases}
      J_{m_-} &,m\in[c,m_-]\\
      J_m &, m\in[m_-,m_+]\\
      J_{m+} &, m\in[m_+,d]
    \end{cases}.
  \end{equation*}
  Then $ \dist(m,J\setminus I_m)\ge |I_m|/100 $.

  Take $ m\in [c,d] $. 
  Using (2) and (3) (also recall \cref{eq:stability-x}), for any
  \begin{equation*}
    |x-x_0|< \exp(-2K),\quad \ |E'-E_0|\le \frac{1}{2}\exp(-K)
  \end{equation*}
  we have
  \begin{equation*}
	\dist (\spec H_{I_m}(x),E')\ge \frac{1}{4}\exp(-K)>\exp(-|I_m|^{\sigma/2})
  \end{equation*}
  Combining  the spectral form of (LDT) from
  Corollary~\ref{cor:4.6zeros} with \cref{lem:Green} we get
  \begin{equation*}
    \left| \cG_{I_m}(x,E';m,k) \right|\le
    \exp \left( -\frac{\gamma}{2}|m-k|+\frac{3}{2}|I_m|^{1-\tau/2} \right).
  \end{equation*}
  Using (1) and (3) (which implies $ |I_m|\gg 1 $), the assumptions of
  \cref{lem:Poissoncover} are satisfied, and therefore
  $ E'\notin \spec H_N(x) $ for any
  $ |E'-E_0|\le \frac{1}{2}\exp(-K) $.  This yields the conclusion.
\end{proof}

\begin{remark}
  Obviously, for the covering forms of (LDT) it is enough to have a
  collection of intervals that overlap near their edges for a fraction
  of their size. We will use
  this observation tacitly when we invoke the above results.
\end{remark}

In connection with the estimates given by the covering form of (LDT)
we recall the following elementary criterion for an energy not to be
in the spectrum.

\begin{lemma}[{\cite[Lem. 2.39]{GolSchVod16}}]\label{lem:elemspec1}
  If for some $ x\in \T^d $, $ E\in \R $, $ \rho>0 $, there exist sequences
  $N'_s\to\infty$, $N''_s\to +\infty$ such that
  \[
    \dist(E,\spec H_{[-N_s',N_s'']}(x))\ge \rho,
  \]
  then
  \begin{equation*}
    \dist(E,\spec H(x))\ge \rho.
  \end{equation*}
\end{lemma}

\subsection{Finite Scale Localization}

The covering and spectral forms of (LDT) can be used to obtain
localization of the eigenfunctions on a finite interval.  The
following result is a version of \cite[Prop. 3.1]{GolSchVod16} that
is better suited to the setting of \cref{sec:bulk} and
\cref{sec:edges}.

\begin{prop}\label{prop:localization}
  Let $ x_0\in \tor^d$, $ E_0\in \R $, and assume $ L(E_0)>\gamma>0
  $. Let $ \sigma $ as in (LDT) and $ 0<\beta<\sigma/2 $.  Let
  $ N\ge N_0 $ be integers.  Assume that for any $ 3N_0/2<|m|\le N $
  there exists an interval $ J_m $ such that $ m\in J_m $,
  $ \dist(m,\partial J_m)\ge N_0-N_0^{1/2} $, $ |J_m|\le 10N_0  $, and
  \begin{equation*}
    \dist ( \spec H_{J_m}(x_0),E_0)\ge \exp(-N_0^{\beta}).
  \end{equation*}
  Let
  \begin{equation*}
	[-N',N'']=[-3N_0/2,3N_0/2]\cup \bigcup_{3N_0/2<|m|\le N} J_m.
  \end{equation*}
  Then the following holds provided $ N_0\ge (B_0+ S_V+
  \gamma^{-1})^C $, $ C=C(a,b,\rho,\beta) $.
  If 
  \[
    |x-x_0|<\exp(-2N_0^{\beta}),\quad |E_k^{[-N',N'']}(x)-E_0|<\frac{1}{4}\exp(-N_0^\beta),
  \]
  then
  \[
    |\psi_k^{[-N',N'']}(x,n)|<\exp \left( -\gamma | n|/10 \right),
    \quad |n|\ge 3N_0/4.
  \]
\end{prop}
\begin{proof}
  Take $ x,E=E_k^{[-N',N'']}(x) $, satisfying the assumptions, and
  without loss of generality assume
  $ n\ge 3N_0/4 $. Let $ d=\ceil{n-N_0/2} $. Note that
  $ d>n/3 $.
  Let
  \begin{equation*}
	J=\bigcup \{ J_m : m\in [n-d,n+d+N_0]\cap(3N_0/2,N]   \}
  \end{equation*}
  (we add $ N_0 $ to make sure $ 3N_0/2< n+d+N_0 $, so that the
  intersection is not empty).
  Note that by the assumptions on $ J_m $ we have $
  m+[-(N_0-N_0^{1/2}),N_0-N_0^{1/2}]\subseteq J_m  $, 
  $ N_0< |J|\lesssim d $, $ n\in J $, and $ \dist(n,[-N',N'']\setminus
  J)\ge d $. Using the covering form of (LDT),
  \begin{equation*}
	\dist(H_J(x),E)\ge \frac{1}{4}\exp(-N_0^\beta)>\exp(-|J|^{\sigma/2}),
  \end{equation*}
  and by the spectral form of (LDT),
  \begin{equation}\label{eq:LDE-J}
    \log|f_J(x,E)|\ge |J|L(E)-|J|^{1-\tau/2}.
  \end{equation}
  Using \cref{lem:Green} and Poisson's formula we get
  \begin{equation*}
    \left| \psi_k^{[-N',N'']}(x,n) \right|\le 2 \exp \left( -\frac{\gamma}{2} d+C|J|^{1-\tau/2} \right)
    < \exp \left( -\frac{\gamma}{3} d \right)<\exp \left( -\frac{\gamma}{10} n \right)
  \end{equation*}
  (recall that $ \psi $ is normalized).
\end{proof}

Next we discuss the stability of localized eigenpairs when we increase
the scale. Again, the particular set-up is motivated by the setting of \cref{sec:bulk} and
\cref{sec:edges}. We will use the following elementary lemma from basic perturbation theory.
\begin{lemma}[{\cite[Lem. 2.40]{GolSchVod16}}]\label{lem:eigenvector-stability} Let
  $A $ be a $ N\times N $ Hermitian matrix.  Let
  $E,\epsilon \in \mathbb{R}$, $ \epsilon>0 $, and suppose there exists
  $\phi\in \mathbb{\R}^N$, $\|\phi\|=1$, such that
  \begin{equation}\label{eq:2contraction1}
    \begin{split}
      \|(A-E)\phi\|< \varepsilon.
    \end{split}
  \end{equation}
  Then the following statements hold.

  \vspace{0.5em}
  \noindent {\normalfont (a)}
  There exists a normalized eigenvector $\psi$ of $A$ with an eigenvalue $E_0$
   such that
  \begin{equation}\label{eq:2contraction1aaaM}
    \begin{split}
    E_0\in (E-\varepsilon\sqrt2,E+\varepsilon\sqrt 2),\\
      |\langle \phi,\psi\rangle|\ge (2N)^{-1/2}.
    \end{split}
  \end{equation}

  \vspace{0.5em}
  \noindent {\normalfont (b)}  If in addition there exists $ \eta>\epsilon $ such that the subspace of the
  eigenvectors of $A$ with eigenvalues falling into the interval
  $(E-\eta,E+\eta)$ is at most of dimension one,
  then there exists a normalized eigenvector $\psi$ of $A$ with an eigenvalue
  $E_0\in (E-\varepsilon,E+\varepsilon)$, such that
  \begin{equation}\label{eq:2contraction1aaa}
    \begin{split}
      \|\phi-\psi\|<\sqrt{2}\eta^{-1}\varepsilon.
    \end{split}
  \end{equation}
\end{lemma}
\begin{prop}\label{prop:stabilization}
  We use the notation and assumptions of \cref{prop:localization}.  We
  further assume that there exist integers $ |N_0'-N_0|<N_0^{1/2} $,
  $ |N_0''-N_0|<N_0^{1/2} $, and $ k_0 $, such that the following
  conditions hold:
  \begin{gather*}
	\tag{i}|E_{k_0}^{[-N'_0,N''_0]}(x_0)-E_0|< \exp(-2N_0^\beta),\\
    \tag{ii}|E_j^{[-N'_0,N''_0]}(x_0)-E_{k_0}^{[-N'_0,N''_0]}(x_0)|>  \exp(-N_0^\beta),
	\quad j \neq k_0,\\
    \tag{iii}|\psi_{k_0}^{[-N'_0,N''_0]}(x_0,-N'_0)|,
    |\psi_{k_0}^{[-N'_0,N''_0]}(x_0,N''_0)|< \exp(-2N_0^\beta).
  \end{gather*}
  Then there exist $ E_k^{[-N',N'']},\psi_k^{[-N',N'']} $, such that the following
  estimates hold for any  $|x-x_0|<\exp(-2N_0^{\beta})$,   provided
  $N_0\ge (B_0+ S_V+ \gamma^{-1})^C$, $ C=(a,b,\rho,\beta) $:
  \begin{gather*}
	\tag{1}|E_k^{[-N',N'']}(x)-E_{k_0}^{[-N'_0,N''_0]}(x)|< \exp(-\gamma N_0/20),\\
    \tag{2}|E_j^{[-N',N'']}(x)-E_k^{[-N',N'']}(x)|> \frac{1}{8}\exp(-N_0^{\beta}),\quad j\neq k,\\
    \tag{3}|\psi_k^{[-N',N'']}(x,n)|<\exp \left( -\gamma | n|/10 \right),
    \quad |n|\ge 3N_0/4,\\
    \tag{4}\|\psi_k^{[-N',N'']}(x,\cdot)-\psi_{k_0}^{[-N'_0,N''_0]}(x,\cdot)\|<\exp(-\gamma N_0/20).\\
  \end{gather*}
  Furthermore, if we additionally have
  \begin{equation}\label{eq:Jm-below}
    \dist ( \spec H_{J_m}(x_0),(-\infty,E_0])\ge \exp(-N_0^{\beta}), \quad 3N_0/2<|m|\le N
  \end{equation}
  ($ J_m $ as in \cref{prop:localization}) and
  \begin{equation*}
	\tag{ii'} E_j^{[-N'_0,N''_0]}(x_0)-E_{k_0}^{[-N'_0,N''_0]}(x_0)>  \exp(-N_0^\beta),\quad j\neq k_0,
  \end{equation*}
  then
  \begin{equation*}
	\tag{2'} E_j^{[-N',N'']}(x)-E_k^{[-N',N'']}(x)> \frac{1}{8}\exp(-N_0^{\beta}),\quad j\neq k.
  \end{equation*}
\end{prop}
\begin{proof}
  Due to condition (iii),
  \begin{equation*}
    \|(H_{[-N',N'']}(x_0)-E_{k_0}^{[-N_0',N_0'']}(x_0))\psi_{k_0}^{[-N_0',N_0'']}(x_0,\cdot)\|
    \lesssim \exp(-2N_0^\beta),
  \end{equation*}
  where we naturally extend $ \psi_{k_0}^{[-N_0',N_0'']} $ to $ [-N',N''] $
  by adding zero entries. Part $(a)$ in Lemma~\ref{lem:eigenvector-stability}
  applies and we get that there exists $ k=k(x_0) $ such that
  \begin{equation*}
	|E_k^{[-N',N'']}(x_0)-E_{k_0}^{[-N'_0,N''_0]}(x_0)|\lesssim \exp(-2N_0^\beta).
  \end{equation*}
  Then for $ |x-x_0|<\exp(-2N_0^\beta) $ (recall \cref{eq:stability-x}) we have
  \begin{gather*}
	|E_k^{[-N',N'']}(x)-E_{k_0}^{[-N'_0,N''_0]}(x)|\ll \exp(-N_0^\beta),\\
    |E_k^{[-N',N'']}(x)-E_0|\ll \exp(-N_0^\beta).
  \end{gather*}
  Due to the last estimate,
  Proposition~\ref{prop:localization} applies and (3) follows.  This implies
  \begin{gather*}
    \|(H_{[-N'_0,N''_0]}(x)-E_k^{[-N',N'']}(x)){\psi}_k^{[-N',N'']}(x,\cdot)\|\lesssim
    \exp(-\gamma (N_0-N_0^{1/2})/10).
  \end{gather*}
  Due to condition (ii), part (b) in
  Lemma~\ref{lem:eigenvector-stability} applies with
  $H_{[-N'_0,N''_0]}(x)$ in the role of $A$ and $
  \eta=c\exp(-N_0^\beta) $, $ c\ll 1 $. This yields (1) and (4).
  To prove (2) assume to the contrary that there exist
  $j\neq k$ and $ x $ such that
  \[
    |E_j^{[-N',N'']}(x)-E_k^{[-N',N'']}(x)|\le \frac{1}{8}\exp(-N_0^{\beta}).
  \]
  It follows that
  \begin{gather*}
	|E_j^{[-N',N'']}(x)-E_{k_0}^{[-N'_0,N''_0]}(x)|<\frac{1}{4} \exp(-N_0^\beta),\\
    |E_j^{[-N',N'']}(x)-E_0|<\frac{1}{4} \exp(-N_0^\beta).
  \end{gather*}
  Proposition~\ref{prop:localization} applies and we get
  \[
    |\psi_j^{[-N',N'']}(x,n)|<\exp \left( -\gamma | n|/10 \right),
      \quad |n|\ge 3N_0/4.
  \]
  Now just as above we have
  \[
    \|\psi_j^{[-N',N'']}(x,\cdot)-\psi_{k_0}^{[-N_0',N_0'']}(x,\cdot)\|
    <\exp(-\gamma N_0/20)
  \]
  and hence
  \[
    \|\psi_j^{[-N',N'']}(x,\cdot)-\psi_k^{[-N',N'']}(x,\cdot)\|\lesssim \exp(-\gamma N_0/20)<1.
  \]
  Since, $\psi_k^{[-N',N'']}(x,\cdot),\psi_j^{[-N',N'']}(x,\cdot)$ are normalized eigenvectors with
  different eigenvalues
  \[
    \|\psi_j^{[-N',N'']}(x,\cdot)-\psi_k^{[-N',N'']}(x,\cdot)\|^2=2.
  \]
  This contradiction verifies (2).

  Finally, we check (2'). Clearly all the estimates obtained so far
  hold with the extra assumptions.  Assume to the contrary that there exist $
  j\neq k $ and $ x $ such that (4') fails. By (4) we must have
  \begin{equation*}
	E_j^{[-N',N'']}(x)< E_k^{[-N',N'']}(x)- \frac{1}{8}\exp(-N_0^{\beta})
  \end{equation*}
  It follows that
  \begin{gather*}
	E_j^{[-N',N'']}(x)< E_{k_0}^{[-N'_0,N''_0]}(x)-\frac{1}{4} \exp(-N_0^\beta),\\
    E_j^{[-N',N'']}(x)< E_0-\frac{1}{4} \exp(-N_0^\beta).
  \end{gather*}
  By (ii') and \cref{eq:Jm-below} (recall \cref{eq:stability-x}) we get
  \begin{gather*}
    \dist(\spec H_{[-N_0',N_0'']}(x),E_j^{[-N',N'']}(x))>\frac{1}{4}\exp(-N_0^\beta),\\
	\dist(\spec H_{J_m}(x),E_j^{[-N',N'']}(x))>\frac{1}{2}\exp(-N_0^\beta).
  \end{gather*}
  It follows from \cref{lem:Greencoverap1} that $
  E_j^{[-N',N'']}(x)\notin \spec H_{[-N',N'']}(x) $. This
  contradiction concludes the proof.
\end{proof}

\subsection{Semialgebraic Sets}

Recall that a set $ \cS\subset\R^n $ is called semialgebraic if it is
a finite union of sets defined by a finite number of polynomial
equalities and inequalities. More precisely, a semialgebraic set
$ \cS\subset\R^n $ is given by an expression
\begin{equation*}
  \cS=\cup_j \cap_{\ell\in L_j} \{ P_\ell s_{j\ell} 0  \},
\end{equation*}
where $ \{P_1,\ldots,P_s\} $ is a collection of polynomials of $ n $
variables,
\begin{equation*}
  L_j\subset\{1,\ldots,s\}\text{ and }s_{j\ell}\in\{>,<,=\}.
\end{equation*}
If the degrees of the polynomials are
bounded by $ d $, then we say that the degree of $ \cS $ is bounded by
$ sd $. See \cite[Ch.~9]{Bou05} for more information on
semialgebraic sets.

In our context, semialgebraic sets can be introduced by approximating
the analytic potential $ V $ with a polynomial $ \tilde V $. More
precisely, given $ N\ge 1 $, by truncating $ V $'s Fourier series and
the Taylor series of the trigonometric functions, one can obtain a
polynomial $ \tilde V $ of degree less than
\begin{equation*}
  C(d,\rho)(1+\log \norm{V}_\infty)N^4
\end{equation*}
such that
\begin{equation}\label{eq:V-tilde}
  \sup_{x\in\T^d}|V(x)-\tilde V(x)|\le \exp(-N^2).
\end{equation}
If we let $ \tilde H $ be the operator with the truncated potential $
\tilde V $,  we have
\begin{equation}\label{eq:H-tilde}
  \sup_{x\in \T^d} \norm{H_{[a,b]}(x)-\tilde H_{[a,b]}(x)}\le \exp(-N^2)
\end{equation}
for any $ [a,b]\subseteq \Z $.

Our use of semialgebraic sets will be limited to applying the following result.
\begin{lemma}[{\cite[Cor. 9.6]{Bou05}}]\label{lem:sa-covering}
  Let $ \cS\subset [0,1]^n $ be semialgebraic of degree $ B $. Let
  $ \epsilon>0 $ be a small number and $ \mes_n (\cS)<\epsilon^n
  $.
  Then $ \cS $ may be covered by at most
  $ B^C \left( \frac{1}{\epsilon} \right)^{n-1} $ balls of radius
  $ \epsilon $
\end{lemma}

\subsection{Resultants}

We briefly recall the definition of the resultant of two univariate
polynomials and some of the basic properties that we will use in \cref{sec:example}.  Let
\begin{equation*}
  P(z) = a_nz^n + a_{n-1} z^{n-1} + \cdots + a_0,\qquad Q(z) = b_mz^m + b_{m-1} z^{m-1} + \cdots + b_0,
\end{equation*}
be polynomials,
$a_i, b_j \in \IC$, $a_n\neq 0$, $b_m\neq 0$.  Let $\zeta_i$,
$1 \le i \le n$ and $\eta_j$, $1 \le j \le m$ be the zeros of $P$
and $Q$ respectively.  The resultant of $P$ and $Q$ is the quantity
\begin{equation}\label{eq:resdef}
  \Res(P, Q) = a_n^mb_m^n\prod_{i,j} (\zeta_i - \eta_j).
\end{equation}
The resultant can be expressed explicitly in terms of the
coefficients (see \cite{Lan02}):
\begin{equation}\label{eq:resdef1}
  \Res(P,Q)=
  \left |
    \begin{array}{cccccc}
      a_n     &        &         & b_m     &        &            \\
      a_{n-1} & \ddots &         & b_{m-1} & \ddots &            \\ 
      \vdots  & \ddots & a_n     & \vdots  & \ddots & b_m        \\
      a_0     & \ddots & a_{n-1} & b_0     & \ddots & b_{m-1}    \\
              & \ddots & \vdots  &         & \ddots & \vdots     \\
      \kern-2em\undermat{m}{ a_{n-1} & \ddots & a_{n-1}}{ &  & a_0}
       & \kern-2em \undermat{n}{b_{n-1} & \ddots & b_{n-1}}{ &  & b_0} 
    \end{array}
  \right |
  \vspace{2em}
\end{equation}

\begin{lemma}\label{lem:resultant}
  Let $ P,Q,\zeta_i,\eta_j $ as above and $ r_P=\max_i |\zeta_i| $,
  $ r_Q=\max_j |\eta_j| $.  If there exists $z$ such that
  \begin{equation}\label{eq:max-bound}
	 \max (|P(z)|,|Q(z)|)<\min(|a_n|,|b_m|)\delta^{\max(m,n)},
  \end{equation}
  for some $ \delta\in (0,1) $,
  then
  \[
    \big|\Res(P, Q)\big|< 2|a_n|^m |b_m|^n (2r)^{mn-1}\delta,
  \]
  with  $r=\max (r_P,r_Q)$.
\end{lemma}
\begin{proof}
  For \cref{eq:max-bound} to hold there must exist
  $ \zeta_{i_0}, \eta_{j_0} $ such that $ |z-\zeta_{i_0}|<\delta $, $ |z-\eta_{j_0}|<\delta $ and therefore,
  using \cref{eq:resdef},
  \begin{equation*}
    |\Res(P,Q)|\le  |a_n|^m |b_m|^n (2r)^{mn-1}|\zeta_{i_0}-\eta_{j_0}|<|a_n|^m |b_m|^n (2r)^{mn-1}2\delta.
  \end{equation*}
\end{proof}

For the application of the previous lemma in \cref{sec:example} we
will also need a couple of auxiliary results. First, recall the
following elementary bound for the location of zeros of a polynomial
due to Cauchy (see \cite[Thm. (27,2)]{Mar66}).
\begin{lemma}\label{lem:Cauchy-bound}
  All the zeros of a polynomial $ P(z)=a_n z^n+a_{n-1} z^{n-1}\dots+a_0 $, $
  a_n\neq 0 $, $ n\ge 1 $, are located in the disk $ |z|<1+\max_{k<n}|a_k/a_n| $.
\end{lemma}
Second, we will need the following consequence of Cartan's estimate.

\begin{lemma}\label{lem:Cartan-P} Let $P(z)=a_n
  z^n+a_{n-1}z^{n-1}+\dots+a_0$, $ n\ge 1 $, $ a_n\neq 0 $, $ M=\max_i
  |a_i| $. There exists an absolute constant $ C_0 $ such that for any $ H\gg 1
  $, we have
  \begin{equation*}
	\mes \{ x\in[0,2\pi]: \log |P(\exp(ix))|<\log M-C_0nH \}< \exp(-H/2).
  \end{equation*}
\end{lemma}
\begin{proof}
  Using Cauchy estimates,
  \begin{equation*}
	M\le \max_{|z|=1}|P(z)|.
  \end{equation*}
  In particular, there exists $ z_0 $, $ |z_0|=1 $, such that $
  \log|P(z_0)|\ge \log M $. At the same time
  \begin{equation*}
	\sup_{|z|\le 100} |P(z)|\le 2M 100^n.
  \end{equation*}
  Given $ H\gg 1 $,
  by Cartan's estimate, there exists $ \cB=
  \bigcup_{k=1}^{k_0}\cD(\zeta_k,r_k) $, $ \sum_k r_k\lesssim\exp(-H) $, such
  that
  \begin{equation*}
	\log|P(z)|\ge \log(2M100^n)-CH(\log(2M100^n)-\log M)\ge \log M-C'nH,
  \end{equation*}
  for any $ z\in \cD(0,2)\setminus \cB $. The conclusion follows.
\end{proof}

\section{Basic Tools at Large Coupling}\label{sec:perturbative-refinements}

In this section we discuss some results that rely on having a large
coupling constant. So, we work with operators of the form
\cref{eq:H-lambda}. As in the previous section we assume that $ V $
extends complex analytically to $ \T^d_\rho $. Furthermore, we assume
that $ V $ is not constant.

Our first goal is to give an explicit expression for the constant $ B_0
$ from the previous section (recall \cref{eq:B0}). To this end we will
obtain, in \cref{prop:LLBasic}, a version of \cref{thm:DirLDT} and \cref{prop:uniform} at large
coupling.

Let
\begin{equation}\label{eq:uiota}
  \uiota=\uiota(V):=\inf_{x\in \T^d}\sup \{ |V(x')-V(x)| : x'\in \T^d,\ |x'-x|\le \rho/100 \}.
\end{equation}
Since $ V $ is continuous and non-constant we have $ \uiota >0 $.

\begin{lemma}\label{lem:LLCartan1}
  Let $ \eta\in \C $. For any $ H\gg 1 $ we have
  \begin{equation*}
	\mes\{x\in\T^d: |\log|V(x)-\eta||> H_{V,\eta}H\}\le C(d)\exp(-H^{1/d}),
  \end{equation*}
  with
  \begin{equation*}
	H_{V,\eta}=C(d)(1+\max(0,\log(\norm{V}_\infty+|\eta|))+\max(0,\log \uiota ^{-1})).
  \end{equation*}
\end{lemma}
\begin{proof}
  Given $ x_0\in \T^d $ there exists $ x_0'\in \T^d $ such that $
  |x_0-x_0'|\le \rho/100 $ and either
  \begin{equation*}
	|V(x_0)-\eta|\ge \uiota /2\quad\text{ or }\quad|V(x_0')-\eta|\ge \uiota /2.
  \end{equation*}
  The conclusion follows by \cref{lem:high_cart},
  \cref{lem:Cartan-measure}, and a covering argument.
\end{proof}

To keep track of the dependence of the various constants on the
potential we introduce
\begin{equation}\label{eq:TV}
  T_{V}=2+\max(0,\log\norm{V}_\infty)+\max(0,\log \uiota ^{-1}).
\end{equation}
Note that $ S_{\lambda V}\le 2\log\lambda $, when $ \log \lambda\gg T_V $.
In what follows we will restrict ourselves to ``spectral'' values of $ E $, that is, we
will assume 
$ |E|\le \lambda\norm{V}_\infty+4 $.

\begin{lemma}\label{lem:LLrate1}
  There exists $ \lambda_0(V)=\exp((T_{V})^C) $,
  $ C=C(d) $, such that the following hold for $
  \lambda\ge \lambda_0 $ and $ |E|\le \lambda\norm{V}_\infty+4 $. 
  For any
  $N\le\exp((\log \lambda)^{\frac{1}{4d}})$ we have
  \begin{gather*}
    |L_N(E)-2 L_{2}(E)+L_1(E)|\lesssim \frac{(\log\lambda)^{\frac{1}{2}}}{N} ,\\
    |L_N(E)-\log \lambda|\lesssim (\log \lambda)^{\frac{1}{2}},
  \end{gather*}
  and there exists a set $ \cB_N $, $ \mes (\cB_N) <\exp(-(\log
  \lambda)^{\frac{1}{3d}}) $, such that
  \begin{equation}\label{eq:fN-MN-lambda}
    \big|\log |f_{N}(x,E)|-\log \| M_N (x,E) \|\big |\lesssim (\log \lambda)^{1/2},
  \end{equation}
  for any $x\notin\cB_N$. 
\end{lemma}
\begin{proof}
  Denote by $ \cB $ the set from \cref{lem:LLCartan1} with
  $ \eta=E/\lambda $ and $ H=(\log \lambda)^{\frac{1}{3}+\epsilon} $,
  $ \epsilon\ll 1 $. Set
  $\cB_N=\bigcup_{1\le j\le N} \big(\cB-j\omega\big)$. Note that we
  have $ (\log \lambda)^{1/2}\ge H_{V,\eta}H $ and
  \begin{equation*}
	\mes(\cB_N)\le NC(d)\exp(-(\log \lambda)^{(\frac{1}{3}+\epsilon)\frac{1}{d}})
    < \exp(-(\log \lambda)^{\frac{1}{3d}}).
  \end{equation*}
  For $ x\notin \cB_N $, $ 1\le
  j\le N $, 
  \begin{equation*}
	\left| \log|\lambda V(x+j\omega)-E|-\log\lambda \right|\le (\log \lambda)^{\frac{1}{2}}
  \end{equation*}
  and therefore
  \begin{gather}
    \big|\log |f_\ell(x+(j-1)\omega,E)|-\ell \log\lambda \big|\lesssim (\log \lambda)^{\frac{1}{2}},
    \ \ell=1,2,\label{eq:M-lambda}\\
    \big|\log \|M_\ell(x+(j-1)\omega,E)\| -\ell\log \lambda\big|\lesssim (\log \lambda)^{{\frac{1}{2}}},
    \ \ell=1,2.\label{eq:f-lambda}
  \end{gather}  
  Applying the avalanche principle we get that for any $ x\notin \cB_N $,
  \begin{equation}\label{eq:MN-AP}
    \log \| M_N (x,E) \|= \sum_{j=0}^{N-2} \log \| M_{2} (x+j\omega,E) \|
    - \sum_{j=1}^{N-2} \log \| M_1 (x+j\omega,E) \| + O(\lambda^{-\frac{1}{2}})\\
  \end{equation}
  and
  \begin{multline}\label{eq:fN-AP}
	    \log|f_N(x,E)|\\
    =\log\norm{M_2(x,E)\begin{bmatrix}
        1 & 0\\
        0 & 0
      \end{bmatrix}}+
    \sum_{j=1}^{N-3} \log \| M_{2} (x+j\omega,E) \|+\log\norm{\begin{bmatrix}
        1 & 0\\
        0 & 0
      \end{bmatrix}M_2(x+(N-2)\omega,E)}\\
    - \sum_{j=1}^{N-2} \log \| M_1 (x+j\omega,E) \| + O(\lambda^{-\frac{1}{2}}).
  \end{multline}
  We used the fact that
  \begin{equation}\label{eq:fN-MN}
	\log|f_N(x)|=\log\norm{\begin{bmatrix}
        1 & 0\\
        0 & 0
      \end{bmatrix}M_N(x)\begin{bmatrix}
        1 & 0\\
        0 & 0
      \end{bmatrix}}
  \end{equation}
  (recall \cref{eq:M-f}).
  It follows that \cref{eq:fN-MN-lambda} holds.
  Integrating \cref{eq:MN-AP} yields
  \begin{equation*}
	|NL_N(E)-(N-1)2L_2(E)+(N-1)L_1(E)|\le C\lambda^{-\frac{1}{2}}+4\mes(\cB_N)S_{\lambda V}
    \le \exp(-(\log\lambda)^{\frac{1}{4d}}).
  \end{equation*}
  By integrating \cref{eq:M-lambda} we get
  \begin{equation*}
    |L_1(E)-\log\lambda|, |L_{2}(E)-\log\lambda|\lesssim (\log\lambda)^{\frac{1}{2}}
    +(S_{\lambda V}+\log\lambda)\exp(-(\log\lambda)^{\frac{1}{3d}})\lesssim (\log\lambda)^{\frac{1}{2}}.
  \end{equation*}
  Therefore
  \begin{equation*}
    |L_N(E)-2 L_{2}(E)+L_1(E)|\le \exp(-(\log\lambda)^{\frac{1}{4d}})+\frac{2(L_1(E)-L_2(E))}{N}
    \lesssim \frac{(\log\lambda)^{\frac{1}{2}}}{N}
  \end{equation*}
  and
  \begin{equation*}
	|L_N(E)-\log \lambda|\lesssim \frac{(\log\lambda)^{\frac{1}{2}}}{N}+(\log \lambda)^{\frac{1}{2}}
    \lesssim (\log \lambda)^{\frac{1}{2}}.
  \end{equation*}
\end{proof}

We use the avalanche principle to extend by induction the estimates of the previous
lemma for arbitrarily large $ N $.

\begin{lemma}\label{lem:condLDTf}
  Let $ E\in \C $, and $ \sigma,\tau $ as in \cref{thm:anyLDT}. There
  exist $ \ell_0(a,b,\rho) $  and $
  \lambda_0(V)=\exp((T_{V})^C) $, $ C=C(d) $, such that
  the following hold for $ \lambda\ge \lambda_0 $,
  $ \ell\ge \ell_0 $, and $ |E|\le \lambda\norm{V}_\infty+4 $.  Assume that for any
  $ \ell\le \ell',\ell''\le 4\ell $ we have
  \begin{gather}
	|L_{\ell'}(E)-L_{\ell''}(E)|\le \frac{(\log \lambda)\log \ell}{\ell},
    \qquad  L_{\ell'}(E)\ge \frac{1}{2}\log \lambda,\label{eq:ell-Lyapunov-assumption}
    \\
    \mes \left\{x\in \T^d: \big|\log |f_{\ell'}(x,E)|-\ell' L_{\ell'}(E)\big|
      >S_{\lambda V}(\ell')^{1-\tau/2}\right\}<\exp(-(\ell')^{\sigma/2})\label{eq:ell-LDT-assumption}.
  \end{gather}
  Then for $ \ell^{10}\le N\le \ell^{100} $, $ N\le N',N''\le 4N $, we have
  \begin{gather*}
	|L_{N'}(E)-L_{N''}(E)|\le \frac{(\log \lambda)\log N}{N},\\
    L_{N'}(E)\ge L_\ell(E)- \frac{2(\log\lambda)\log \ell}{\ell}
    -\frac{(\log\lambda)\log N'}{3N'},\\
	\mes \left\{x\in \T^d: \big|\log |f_{N'}(x,E)|-N' L_{N'}(E)\big|
      >S_{\lambda V}(N')^{1-\tau/2}\right\}<\exp(-(N')^{\sigma/2}).    
  \end{gather*}
\end{lemma}
\begin{proof} We first prove the statements pertaining to the Lyapunov exponents.
  The derivation follows the method in \cite[Lemma 4.2]{GolSch01}. We
  omit some details. We also suppress $E$ from most of the notation.
  To shorten the presentation we consider the case $N=n\ell$, $ n\in
  \N $,  only.
  By \cref{thm:anyLDT} and \cref{eq:ell-Lyapunov-assumption} we have
  \begin{equation}\label{eq:M-ell}
    \log \|M_{\ell}(x+j\ell\omega)\|\ge \ell L_{\ell}-C_0S_{\lambda V}\ell^{1-\tau}
    \ge \frac{1}{4}\ell\log\lambda
  \end{equation}
  and
  \begin{multline}\label{eq:M-2ell}
	\log \|M_{\ell}(x+j\ell\omega)\|+\log \|M_{\ell}(x+(j+1)\ell\omega)\|
    -\log \|M_{2\ell}(x+j\ell\omega)\|\\
    \le 2\ell(L_\ell-L_{2\ell})+2C_0S_{\lambda V}\ell^{1-\tau}+C_0S_{\lambda V}(2\ell)^{1-\tau}
    < \frac{1}{8}\ell\log\lambda,
  \end{multline}
  for any $0\le j\le N$, $x\notin \cB$, $\mes (\cB)\le
  2n\exp(-\ell^{\sigma})\le \exp(-\ell^{\sigma}/2)$.
  With these estimates in hand the avalanche principle kicks in and yields
  \begin{equation}\label{eq4.4}
    \log \| M_N (x) \|
    = \sum_{j=0}^{n-2}  \log \| M_{2\ell} (x+j\ell\omega) \|
    - \sum_{j=1}^{n-2} \log \| M_\ell (x+j\ell\omega) \| + O(\exp(-(\ell\log\lambda)/8)),
  \end{equation}
  for any $x\notin \cB$. 
  Recalling \cref{eq:LEupperb1} and integrating~\eqref{eq4.4} over~$x$
  yields
  \begin{equation*}
    \left|L_N-\frac{n-1}{n}2 L_{2\ell}+\frac{n-2}{n}L_\ell\right|
    \le \frac{1}{N}C\exp(-(\ell\log\lambda)/8)+4\mes(\cB)S_{\lambda V}
    \le \exp(-c_0\ell^{\sigma}/4)\log\lambda.
  \end{equation*}
  Therefore
  \begin{equation*}
    |L_N-2L_{2\ell}+L_\ell|\le \exp(-c_0\ell^{\sigma}/4)\log\lambda+\frac{2}{n}(L_\ell-L_{2\ell})
    \le \frac{3(\log \lambda)\log\ell}{N} \le \frac{(\log\lambda)\log N}{3N}.
  \end{equation*}
  The same estimate also holds for general $ N $ (not just $ N=n\ell
  $) and $ N\le N',N''\le 4N
  $. This implies the estimates for the Lyapunov exponents.

  Next, we consider the statement about the determinants. The main
  tool here is the application of the avalanche principle to expand
  $\log |f_N|$.
  The argument is very close to the one in
  \cite[Corollary 3.10]{GolSch08}.  Again we omit some details and
  assume $N=n\ell$, $ n\in \N $. On top of \cref{eq:M-ell} and
  \cref{eq:M-2ell}, using \cref{thm:anyLDT} and \cref{eq:ell-LDT-assumption} we have
  \begin{gather*}
    \log\norm{M_\ell(x)\begin{bmatrix}
        1 & 0\\
        0 & 0
      \end{bmatrix}}\ge \log|f_\ell(x)|\ge \ell L_\ell
    -S_{\lambda V}\ell^{1-\tau/2}\ge \frac{1}{4}\ell \log\lambda,\\
    \log\norm{\begin{bmatrix}
        1 & 0\\
        0 & 0
      \end{bmatrix}M_\ell(x+(n-1)\ell\omega)}\ge \log|f_\ell(x+(n-1)\ell\omega)|
    \ge \frac{1}{4}\ell \log\lambda,\\
    \log \|M_{\ell}(x)\|+\log \|M_{\ell}(x+\omega)\|
    -\log \|M_{2\ell}(x)\textstyle{\begin{bmatrix} 1 & 0\\ 0 & 0\end{bmatrix}}\|
    < \frac{1}{8}\ell\log\lambda,\\
    \log \|M_{\ell}(x+(n-2)\ell\omega)\|+\log \|M_{\ell}(x+(n-1)\ell\omega)\|
    -\log \|\textstyle{\begin{bmatrix} 1 & 0\\ 0 & 0\end{bmatrix}}M_{2\ell}(x+(n-2)\ell\omega)\|
    < \frac{1}{8}\ell\log\lambda
  \end{gather*}
  for any $x\notin \cB'$, $\mes(\cB')\le
  4\exp(-\ell^{\sigma/2})$.
  So we can apply the avalanche principle to expand $\log |f_N(x)|$
  for $ x\notin \cB\cup \cB' $ (similarly to \cref{eq:fN-AP}).
  Combining this with \eqref{eq4.4} we get
  \begin{multline}\label{eq4.4f}
    \log |f_N(x)|=\log \| M_N (x) \|
    + \log\left\|M_{2\ell}(x)\begin{bmatrix} 1 & 0\\ 0 &
        0\end{bmatrix}\right\|-\log \|M_{2\ell}(x)\|\\
    +\log\left\|\begin{bmatrix} 1 & 0\\ 0 &
        0\end{bmatrix}M_{2\ell}(x+(n-2)\ell\omega)\right\|
      -\log\|M_{2\ell}(x+(n-2)\ell\omega)\|
    + O(\exp(-(\ell\log\lambda)/8)\\
    \ge \log \| M_N (x) \|-2S_{\lambda V}(2\ell)^{1-\tau/2}-2C_0S_{\lambda V}(2\ell)^{1-\tau}\ge
    NL_N-S_{\lambda V}N^{1-\tau}
  \end{multline}
  for any $x\notin \cB\cup\cB'$ (recall that $ \tau\ll 1 $).
  In particular,
  for any $x_0\in \T^d$ there exists $ x_1\in \T^d $, $|x_1-x_0|\ll
  \rho N^{-1}$  such that
  $\log |f_N(x_1)|\ge NL_N-S_{\lambda V}N^{1-\tau}$.
  On the other hand due to \cref{cor:logupper}
  \begin{equation*}
    \sup\limits_{x\in\tor^d,|y|<\rho N^{-1}} \log |f_N (x+iy)| \le
    NL_N+C(a,b,\rho)S_{\lambda V} N^{1-\tau}.
  \end{equation*}
  Applying Cartan's estimate (with $ H=N^{\tau/3} $) and using a
  covering argument we get
  \begin{equation*}
	\mes \left\{x : |\log|f_N(x)|-NL_N|>S_{\lambda V}N^{1-\tau/2}  \right\}
    \le C(d)\exp(-N^{\tau/(3d)})
    <\exp(-N^{\sigma/2}),
  \end{equation*}
  (recall that $ \sigma\ll \tau $). The same estimate also holds
  for general $ N $ and $ N\le N',N''\le 4N $.
\end{proof}

\begin{prop}\label{prop:LLBasic}
  Let $ E\in \C $, and $ \sigma,\tau $ as in \cref{thm:anyLDT}. There
  exists $ \lambda_0(a,b,\rho,V)=\exp((T_{V})^C) $,
  $ C=C(a,b,\rho) $, such that the following statements hold for
  $ \lambda\ge \lambda_0 $ and $ |E|\le \lambda\norm{V}_\infty+4 $.
  \begin{enumerate}[(a),leftmargin=2em]
  \item We have
    \begin{gather*}
      L_N(E)-L(E)\le \frac{C_0(\log \lambda) \log N}{N},\quad N\ge 2,\\
      L(E)\ge \log \lambda - C_1(\log \lambda)^{\frac{1}{2}}> \frac{1}{2}\log\lambda,
    \end{gather*}
    with $ C_0=C_0(a,b,\rho) $ and $ C_1 $ an absolute constant.
  \item For any $ N\ge \log\lambda $ we have
    \begin{equation*}
      \mes \left\{ x\in \T^d:
        |\log|f_N(x,E)|-L_N(E)|>S_{\lambda V}N^{1-\tau/2} \right\}<\exp(-N^{\sigma/2}).
    \end{equation*}
  \end{enumerate}
\end{prop}
\begin{proof}[Proof of Proposition~\ref{prop:LLBasic}]
  (a) By \cref{lem:LLrate1}, for $ 1\ll \ell \le \exp((\log
  \lambda)^{\frac{1}{4d}})/4 $, $ \ell\le \ell',\ell''\le 4\ell $, we have
  \begin{gather*}
	|L_{\ell'}(E)-L_{\ell'}(E)|\le \frac{C(\log\lambda)^{\frac{1}{2}}}{\ell}\le
    \frac{(\log \lambda)\log\ell}{\ell},\\
    L_{\ell'}(E)\ge \log\lambda-C(\log\lambda)^{\frac{1}{2}}\ge \frac{1}{2}\log\lambda.
  \end{gather*} 
  Let $ \ell_0 $ as in \cref{lem:condLDTf}. We choose $ \lambda_0 $
  such that $ \ell_0\le \log\lambda_0 $.
  Using the above, \cref{lem:condLDTf}, and induction we get
  that for any $ N\ge \ell_0 $, $ N\le N',N''\le 4N $ we have
  \begin{gather*}
	|L_{N'}(E)-L_{N''}(E)|\le \frac{(\log\lambda)\log N}{N}.
  \end{gather*}
  In particular we have
  \begin{equation*}
	L_N(E)-L_{2^kN}(E)\le \sum_{j=0}^{k-1}\frac{(\log\lambda)\log (2^jN)}{2^j N}
    \le \frac{C(\log\lambda)\log N}{N},
  \end{equation*}
  with $ C $ an absolute constant. The first statement of part (a) follows
  by letting $ k\to \infty $ and by adjusting the constant $ C $ to
  also cover the case $ N<\ell_0 $. The second statement follows from
  the fact that for $ \ell=\lfloor
  \exp((\log\lambda)^{\frac{1}{4d}}) \rfloor $, we have
  \begin{equation*}
	L(E)\ge L_{\ell}(E)-\frac{C(\log \lambda)\log \ell}{\ell}
    \ge \log \lambda -C(\log\lambda)^{\frac{1}{2}}-\exp(-(\log\lambda)^{\frac{1}{5d}})
    \ge \log \lambda-C'(\log\lambda)^{\frac{1}{2}}.
  \end{equation*}

  (b) Take $ \log \lambda\le \ell \le (\log \lambda)^{100} $. Using 
  \cref{lem:LLrate1} and \cref{thm:anyLDT} we get
  \begin{equation*}
	\mes \left\{ x : |\log|f_\ell(x,E)|-\ell L_\ell(E)|>C_0S_{\lambda V}\ell^{1-\tau}
      +C(\log\lambda)^\frac{1}{2}  \right\}
    <\exp(-(\log \lambda)^\frac{1}{3d}).
  \end{equation*}
  Note that with this choice of $ \ell $ we have
  \begin{equation*}
	C_0S_{\lambda V}\ell^{1-\tau}
    +C(\log\lambda)^\frac{1}{2}<S_{\lambda V}\ell^{1-\tau/2},\qquad
    \exp(-(\log \lambda)^\frac{1}{3d})<\exp(-\ell^{\sigma/2})
  \end{equation*}
  (recall that $ \sigma\ll \tau \ll 1 $). Recalling that $
  \ell_0\le \log\lambda_0 $,
  the conclusion follows by
  \cref{lem:condLDTf} and induction.
\end{proof}

\begin{remark}\label{rem:Lbridge}
  (a) The previous proposition shows that for $ \lambda\ge
  \lambda_0\gg 1 $  and $ |E|\le \lambda\norm{V}_\infty+4 $,  \cref{thm:DirLDT} holds with $ N_0=(\log
  \lambda)^{C(a,b)} $, and \cref{prop:uniform} holds with $
  C_0=C(a,b,\rho)\log\lambda $. Therefore, for such $ \lambda $  and $
  E $ we can take $ B_0= (\log \lambda)^{C(a,b,\rho)} $. By inspection of the
  previous proofs one can see that for $ |E|>\lambda \norm{V}_\infty+4
  $ we can take $ B_0= (\log \lambda+\log|E|)^{C(a,b,\rho)}$, but we
  will not use this fact.

  \medskip\noindent (b) The positivity of the Lyapunov exponent for 
  $\lambda\ge \lambda_0\gg1$ is well-known (see \cite{GolSch01},
  \cite{Bou05a}, \cite{DuaKle16}). We only included the proof because
  it is an easy consequence of the lemmas we needed for the other statements.
\end{remark}

Next we establish a version of the covering form of (LDT) and of the result
on finite scale localization from \cref{prop:stabilization},
starting from the potential.  We will need these results in
\cref{sec:A-to-DE} to connect the assumptions on the potential to the
initial conditions required by our inductive schemes from
\cref{sec:bulk} and \cref{sec:edges}.

\begin{lemma}\label{lem:efextension2b}
  Let $x_0\in \tor^d$, $ [a,b]\subset \Z $, $ a<b $. There exists
  $ \lambda_0(V)=\exp((T_{V})^C) $, $ C=C(\rho) $, such that the
  following hold for $\lambda\ge\lambda_0$ and $ |E_0|\le \lambda \norm{V}_\infty+4 $.  Assume
  \begin{gather*}
    |V(x_0+n\omega)-\lambda^{-1}E_0|\ge \exp(-K),\quad \text{for any $n\in [a,b]$},
  \end{gather*}
  with some $ K\ge (\log\lambda)^{1/3} $.
  Then for any $|x - x_0|< \exp(-2K)$, $
  \lambda^{-1}|E-E_0|<\frac{1}{2}\exp(-K) $,
  \begin{gather*}
	\tag{a} \dist (\spec H_{[a,b]}(x),E_0)\ge \frac{1}{2} \lambda\exp(-K),\\
    \tag{b} \big | \cG_{[a,b]} (x,E;j,k)\big |
    \le \exp\left(-(|j -k|+1)\log\lambda+C(b-a)K\right),
  \end{gather*}
  where $ C $ is an absolute constant.                                   
\end{lemma}
\begin{proof}
  For any $|x - x_0|< \exp(-2K)$,
  $\lambda^{-1}|E - E_0|\le \frac{1}{2}\exp(-K)$,
  \begin{gather*}
    |V(x+n\omega)-\lambda^{-1}E|\ge \frac{1}{4} \exp(-K),\quad j\in [a,b]
  \end{gather*}
  ($ \lambda_0 $ depends on $ \rho $ because we used a Cauchy
  estimate). Then
  \begin{equation*}
	|\log|\lambda V(x+n\omega)-E|-\log \lambda|\lesssim K,\quad n\in [a,b]
  \end{equation*}
  (note that $
  |V(x+n\omega)-\lambda^{-1}E|\le \exp((\log \lambda)^{1/3})\le\exp(K)
  $, for large enough $\lambda $)
  and this implies
  \begin{gather*}
    \big|\log |f_\ell (x+(n-1)\omega,E)|-\ell \log\lambda\big|\lesssim K,
    \quad n\in[a,b-\ell],\ell=1,2,\\
    \big|\log \|M_\ell(x+(n-1)\omega,E)\| -\ell \log\lambda\big|\lesssim K,
    \quad n\in [a,b-\ell],\ell=1,2.
  \end{gather*}
  Applying the avalanche principle (as in the proof of
  \cref{lem:LLrate1}) we have
  \begin{multline*}
	    \log|f_{[a,b]}(x,E)|=\log\norm{M_2(x+(a-1)\omega,E)\begin{bmatrix}
        1 & 0\\
        0 & 0
      \end{bmatrix}}+
    \sum_{n=a}^{b-a-2} \log \| M_{2} (x+n\omega,E) \|\\+\log\norm{\begin{bmatrix}
        1 & 0\\
        0 & 0
      \end{bmatrix}M_2(x+(b-a-1)\omega,E)}
    - \sum_{n=a}^{b-a-1} \log \| M_1 (x+n\omega,E) \| + O(\lambda^{-\frac{1}{2}}).
  \end{multline*}
  It then follows that
  \begin{equation}\nn
    |\log |f_{[a,b]}(x,E)|-(b-a+1)\log\lambda|\lesssim (b-a+1)K,
  \end{equation}
  In particular,  $E\notin \spec
  H_{[a,b]}(x)$.  This implies (a). Analogous estimates hold on any
  subinterval of $ [a,b] $. Using these estimates and 
  Cramer's rule for the resolvent we get (for $ j\le k $) 
  \begin{multline*}
    \log \big | \cG_{[a,b]}(x,E;j, k) \big |  = \log\big | f_{[a,j-1]}
    (x,E)\big |+\log \big | f_{[k+1,b]} \bigl(x,
    E\bigr)\big|- \log\big | f_{[a,b]}(x,E)\big |\\
    \le [(j-a)+(b-k)](\log\lambda +CK)
    -(b-a+1)((j-a)(\log\lambda-CK))\\
    \le (j-k-1)\log\lambda+C'(b-a)K.
  \end{multline*}
  This implies (b).
\end{proof}

\begin{cor}\label{cor:covering-perturb}
  Let $x_0\in \tor^d$, $ S\subset \mathbb{R}$, $ [a,b]\subset \Z $, $ a<b $. There exists
  $ \lambda_0(V)=\exp((T_{V})^C) $, $ C=C(\rho) $, such that the
  following hold for $\lambda\ge\lambda_0$. If
  \begin{gather*}
    \dist(V(x_0+n\omega,\lambda^{-1}S)\ge \exp(-K),\quad \text{for any $n\in [a,b]$},
  \end{gather*}
  with some $ K\ge (\log\lambda)^{1/3} $,
  then for any $|x - x_0|< \exp(-2K)$,
  \begin{equation*}
    \dist (\spec H_{[a,b]}(x),S)\ge \frac{1}{2} \lambda\exp(-K).
  \end{equation*}
\end{cor}
\begin{proof}
  This follows by applying \cref{lem:efextension2b} (a) for each
  $ E_0\in S $ with $ |E_0|\le \lambda\norm{V}_\infty+4 $. Note that
  for $ |E_0|>\lambda\norm{V}_\infty+4 $, \cref{lem:efextension2b} (a)
  holds trivially.
\end{proof}

In the results of this section we could have used $
(\log\lambda)^\epsilon $, $ \epsilon\in(0,1) $, instead of $
(\log\lambda)^{1/2} $. So far, working in such generality wasn't
needed. However, we will need this setting for the applications of
the next lemma.

\begin{lemma}\label{lem:efextension2}
  Let $x_0\in \tor^d$, $ a<0<b $, $ \epsilon\in (0,1) $, and assume
  \begin{equation*}
    |V(x_0+n\omega)-V(x_0)|\ge \exp(-(\log\lambda)^{\epsilon}),\quad \text{for any } n\in[a,b]\setminus \{ 0 \}.
  \end{equation*}
  There exists
  $ \lambda_0(V)=\exp((T_{V})^C) $, $ C=C(\rho,\epsilon) $, such that the
  following hold for $\lambda\ge\lambda_0$.
  There exist $ E_k^{[a,b]},\psi_k^{[a,b]} $ such that for any
  $|x-x_0|<\exp(-3(\log\lambda)^{\epsilon})$ the following estimates hold:
  \begin{gather*}
	\tag{1}
    |\lambda^{-1}E_k^{[a,b]}(x)-V(x)|\le 2\lambda^{-1}\\
    \tag{2}
    |\psi_k^{[a,b]}(x,n)|<\exp(-(\log \lambda) |n|/2),\quad |n|>0,\\
    \tag{3}
    |\psi_k^{[a,b]}(x,0)-1|<\exp(-(\log\lambda)/2),\\
    \tag{4}
    \lambda^{-1}|E_j^{[a,b]}(x)-E_k^{[a,b]}(x)|\ge
    \frac{1}{8}\exp(-(\log\lambda)^{\epsilon}),\quad j\neq k.
  \end{gather*}
  Furthermore, if
  \begin{equation}\label{eq:Vn-V}
	V(x_0+n\omega)-V(x_0)\ge \exp(-(\log\lambda)^{\epsilon}),\quad \text{
      for any } n\in [a,b]\setminus \{ 0 \},
  \end{equation}
  then
  \begin{equation*}
	\tag{4'}     \lambda^{-1}(E_j^{[a,b]}(x)-E_k^{[a,b]}(x))\ge
    \frac{1}{8}\exp(-(\log\lambda)^{\epsilon}),\quad j\neq k.
  \end{equation*}
\end{lemma}
\begin{proof}
  The proof is very similar to the one of \cref{prop:stabilization}.
  We have
  \begin{equation*}
    \|(\lambda^{-1}H_{[a,b]}(x)-V(x))\delta_0\|\le \sqrt{2}\lambda^{-1},
  \end{equation*}
  where $\delta_0$ stands for the standard unit vector with mass
  concentrated at $0$. By \cref{lem:eigenvector-stability} there
  exists $ k=k(x) $ such that (1) holds. At the end we will argue that
  $ k(x)=k(x_0) $. Note that 
  \begin{equation*}
    \lambda^{-1}|E_k^{[a,b]}(x)-E_0|\ll \exp(- 2(\log\lambda)^{\epsilon}),\quad E_0=\lambda V(x_0).
  \end{equation*}
  Estimate (2) now follows from Poisson's formula and \cref{lem:efextension2b}
  (b) (applied, for $ n>0 $, on $ [1,2n]\cap[a,b] $).
  Since $ \psi_k^{[a,b]} $ is normalized, estimate (3) follows
  from (2) (obviously, we choose $ \psi_k^{[a,b]} $ such that $
  \psi_k^{[a,b]}(x,0)\ge 0 $). To prove $(4)$ assume
  to the contrary that there exist $j\neq k$ and $ x $ such that
  \[
    \lambda^{-1}|E_j^{[a,b]}(x)-E_k^{[a,b]}(x)|<\frac{1}{8}\exp(-2(\log\lambda)^{\epsilon}).
  \]
  Then
  \begin{equation*}
    \lambda^{-1}|E_j^{[a,b]}(x)-E_0|< \exp(- 2(\log\lambda)^{\epsilon}),\quad E_0=\lambda V(x_0).
  \end{equation*}
  and just as above we get
  \[
    |\psi_j^{[a,b]}(x,n)|<\exp(-(\log\lambda)|n|/2),\quad |n|>0,
    \qquad |\psi_j^{[a,b]}(x,0)-1|<\exp(-(\log\lambda)/2).
  \]
  Therefore $
  \norm{\psi_j^{[a,b]}(x,\cdot)-\psi_k^{[a,b]}(x,\cdot)}\ll 1 $,
  contradicting the fact that
  \begin{equation*}
    \norm{\psi_j^{[a,b]}(x,\cdot)-\psi_k^{[a,b]}(x,\cdot)}^2=2.
  \end{equation*}

  Now we argue that $ k(x)=k(x_0) $. Since
  \begin{equation*}
	|\lambda^{-1}E_{k(x_0)}^{[a,b]}(x_0)-V(x_0)|\le 2\lambda^{-1},
  \end{equation*}
  we have
  \begin{equation*}
	|\lambda^{-1}E_{k(x_0)}^{[a,b]}(x)-V(x)|\ll \exp(-2(\log\lambda)^{\epsilon}),
  \end{equation*}
  and the conclusion follows using (1) and (4).

  Finally, suppose that \cref{eq:Vn-V} holds. Clearly, estimates
  (1)-(4) still hold. Suppose to the contrary that there exist $
  j\neq k $ and $ x $ such that (4') fails. By (4) we must have
  \begin{equation*}
	\lambda^{-1}E_j^{[a,b]}(x)<\lambda^{-1}E_k^{[a,b]}(x)-\frac{1}{8}(\log \lambda)^{\epsilon}.
  \end{equation*}
  By (1),
  \begin{equation*}
	\lambda^{-1}E_j^{[a,b]}(x)<V(x)-\frac{1}{4}(\log \lambda)^{\epsilon}.
  \end{equation*}
  Note that due to \cref{eq:Vn-V},
  \begin{equation*}
	V(x+n\omega)-V(x)\ge \frac{1}{2}(\log\lambda)^{\epsilon},
  \end{equation*}
  for $ |x-x_0|<\exp(-3(\log\lambda)^{\epsilon}) $. It follows that
  \begin{equation*}
	|V(x+n\omega)-\lambda^{-1}E_j^{[a,b]}(x)|\ge \frac{1}{4}(\log\lambda)^{\epsilon},\quad n\in [a,b],
  \end{equation*}
  and by \cref{lem:efextension2b}, $ E_j^{[a,b]}(x)\notin \spec
  H_{[a,b]}(x) $. This contradiction shows that (4') holds.
\end{proof}

\begin{cor}\label{cor:close-ev-lambda} 
  Using the assumptions and notation of \cref{lem:efextension2} the
  following hold. For simplicity let $ E^{[a,b]},\psi^{[a,b]} $ be the
  eigenpair from \cref{lem:efextension2}. If $ N\ge 1 $, $
  [-N,N]\subset [a,b] $, then for any $ |x-x_0|<\exp(-3(\log\lambda)^{\epsilon}) $,
  \begin{equation*}
	|E^{[a,b]}(x)-E^{[-N,N]}(x)|\lesssim \exp(-(\log\lambda)N/2).
  \end{equation*}
\end{cor}
\begin{proof}
  Using (2) from \cref{lem:efextension2}, we have
  \begin{equation*}
	\norm{(H_{[-N,N]}(x)-E^{[a,b]}(x))\psi^{[a,b]}(x,\cdot)}\lesssim  \exp(-(\log\lambda)N/2).
  \end{equation*}
  The conclusion follows from \cref{lem:eigenvector-stability}, and
  (1) and (4) from the previous lemma. 
\end{proof}

\section{Cartan Type Estimates Along Level Sets near a Non-Degenerate Extremum Point}\label{sec:Cartan-Morse}

The goal of this section is to prove the next proposition that we will
use to handle the edges of the spectrum in \cref{sec:edges}. We
let $\mathfrak{H}(f)$ stand for the Hessian of a function $ f $. When
the function is clear from the context, we will simply write $ \fH
$. Recall that  $ \norm{\cdot} $ denotes the Euclidean norm, and $
|\cdot| $ denotes the sup-norm.

\begin{prop}\label{prop:levelsetshifts}
  Let $f(x)$ be a  real-analytic function defined on $ \{ x\in \R^n:
  |x|<r_0 \} $, $r_0<1$, which extends analytically to the polydisk $
  \cP:=\{ z\in \C^n: |z|<r_0 \} $. Assume that
  \begin{gather*}
    f(0)=0,\quad \nabla f(0)=0,\\
    \mathfrak{H}(0)\ge \nu_0I,\quad 0<\nu_0<1.
  \end{gather*}
  Let
  $M(k)=\max_{|\alpha|=k}\sup_{\cP}|\partial^\alpha  f|$.  Set
  \begin{gather*}
    \nu_1:=c(n)\nu_{0}(1+M(2)+M(3))^{-1},\quad    \rho=r_0 \nu_1^{10},
  \end{gather*}
  with $c(n)$ a sufficiently small constant.
  Let $0<\|x_0\|<\rho$, $E_0=f(x_0)$, $ r=\nu_1 \norm{x_0} $. Then there exists a
  real-analytic map $ x(y,E) $, $ (y,E)\in \R^{n-1}\times \R $, $
  |y|<r $, $ |E-E_0|<r^2 $, such that
  \begin{equation*}
	f(x(y,E))=E,\quad x(0,E_0)=x_0
  \end{equation*}
  and satisfying the following properties.

  \smallskip\noindent
  (I) The map $ x(y,E) $ extends analytically to $ \{ (w,E)\in
  \C^{n-1}\times \C: |w|<r,|E-E_0|<r^2\} $ and satisfies
  \begin{equation*}
	\norm{x(w,E)-x_0}<\frac{\norm{x_0}}{2}.
  \end{equation*} 

  \smallskip\noindent
  (II)   For any $|E-E_0|<r^2$, any vector $h\in \mathbb{R}^n$ with
  $0<\|h\|<\rho$, and any $H\gg 1 $, we have
  \begin{equation*}
    \mes \{y\in \R^{n-1}, |y|<r: \log |f(x(y,E)+h)-E|\le H_0H\}
    \le (\nu_1^{-2}r)^{n-1}\exp(-H^{\frac{1}{n-1}}),
  \end{equation*}
  with $     H_0=C(n)\log(\norm{h}\norm{x_0}) $.

  \smallskip\noindent
  (III) Let $h_0\in \mathbb{R}^n$ be an arbitrary unit vector.
  For any $|E-E_0|<r^2 $, and any $ H\gg 1 $, we have
  \begin{equation*}
    \mes \{y\in \R^{n-1},|y|<r: \log \big|\langle \nabla f(x(y,E)),h_0 \rangle\big|\le H_1H\}
    \le  (\nu_1^{-2}r)^{n-1}\exp(-H^{\frac{1}{n-1}}),
  \end{equation*}
  with $ H_1=C(n)\log(\nu_1\|x_0\|) $.
\end{prop}

Part (I) of the proposition is a version of the implicit function
theorem. For parts (II) and (III) we will apply Cartan's estimate to $
f $ along its level sets. To apply it we need a reference point with a ``nice'' lower bound estimate.
So, it is important to accurately book-keep the size of the
neighborhood where one can apply the implicit function theorem
for it limits the search for the point in question.
The same applies to all auxiliary estimates in the proof.
For that matter we need to work out a
version of the implicit function theorem, explicit enough for our
purposes (see \cref{lem:parametrization1}).

\begin{lemma}\label{lem:impl2}
  Let $ f(z,w) $ be an analytic function defined on the polydisk
  \begin{equation*}
    \mathcal{P}=\{(z,w)\in \C\times \C^n:|z|,|w|<\rho_0\}.
  \end{equation*}
  Let $ M_1=\sup |\partial_z f| $,  $M(2)= \max_{|\alpha|=2}\sup\big|\partial^\alpha f\big|$.
  Assume that $f(0,0)=0$, $\mu_{0}:=| \partial_z f(0,0)| >0$.  Let
  \begin{equation*}
    \rho_1\le  \min(\rho_0/2,c(n)\mu_0M(2)^{-1}),
    \quad r_i=c(n)\rho_1\min(1,\mu_0/|\partial_{w_i}f(0,0)|),
  \end{equation*}
  with $c(n)$ a sufficiently small constant. Then for any $ w $, $|w_i|<r_i$, the equation
  \begin{gather*}
    f(z,w)=0
  \end{gather*}
  has a unique solution $|z(w)|<\rho_1$ which is an analytic function
  of $ w $.
\end{lemma}
\begin{proof}
  Take arbitrary $ w $, $|w_i|<r_i$, and $ z $, $ |z|=\rho_1 $. Then
  by Taylor's formula and the definition of $ \rho_1,r_i $,
  \begin{equation}\label{eq:f-rho}
    \big|f(z,w)\big|\ge |\partial_z f(0,0)||z|-|\langle \nabla_w f(0,0),w  \rangle|-C(n)M(2)\norm{(z,w)}^2
    \ge \mu_0 \rho_1/2.
  \end{equation}
  In particular we also have
  \begin{equation*}
	|f(z,0)|\ge |\partial_z f(0,0)||z|-C(n)M(2)|z|^2>0,
  \end{equation*}
  for $ 0<|z|\le \rho_1$. 
  So, $ f(z,0) $ has a simple root at $z=0$ and no other roots in
  the disk $|z|<\rho_1$,
  hence 
  \begin{gather*}
    \frac{1}{2\pi i}\oint_{|z|=\rho_1}\frac{\partial_z f(z,0)}{f(z,0)}dz=1.
  \end{gather*}
  By continuity,
  \begin{gather*}
    \frac{1}{2\pi i}\oint_{|z|=\rho_1}\frac{\partial_z f(z,w)}{f(z,w)}dz=1,
  \end{gather*}
  for $|w_i|< r_i $.  This means $z\to f(z,w)$ has one
  simple root $z(w)$ in the disk $\{|z|<\rho_1\}$ and by
  the residue theorem
  \begin{equation*}
    z(w)=\frac{1}{2\pi i}
    \oint_{|z|=\rho_1}z\frac{\partial_z f(z,w)}{f(z,w)}dz.
  \end{equation*}
  Clearly, the function on the right-hand side is analytic in
  $w$  for $|w_i|<r_i$.
\end{proof}

For the proof of part (II) of \cref{prop:levelsetshifts} it will be
crucial that  the size in the direction of $ y $ of the polydisk where the implicit function is defined is
of magnitude $\simeq \|\nabla f\|$ and in particular is much bigger than
$\simeq \|\nabla f\|^2$  (assuming $ \norm{\nabla f}<1 $; see
\cref{lem:Morse10}). This is one reason why in \cref{lem:parametrization1}
we consider implicit functions in the direction of the gradient. The
second reason is the fact that this way one gets some quadratic control
over the implicit function (see \cref{eg:g-bound}). 

\begin{definition}\label{def:normal-coordinates}
  Given a function $ f $ differentiable at $ x_0\in\R^n $, with $
  \mu_{x_0}:=\norm{\nabla f(x_0)}> 0 $, we let $
  \mathfrak{n}_{x_0}=\mu_{x_0}^{-1}\nabla f(x_0) $.
  Let $\mathfrak{e}_{x_0,j}$, $1\le j\le n-1$ be an orthonormal basis
  in $\{\mathfrak{n}_{x_0}\}^\bot$.
  Given $(\xi,y)\in \mathbb{R}\times \mathbb{R}^{n-1}$ we denote
  \[
    \varphi(\xi,y;x_0):=x_0+\xi\mathfrak{n}_{x_0}+\sum_jy_j\mathfrak{e}_{x_0,j}.
  \]
\end{definition}


The set-up of the lemmas to follow is tailored around that of \cref{prop:levelsetshifts}.
\begin{lemma}\label{lem:parametrization1}
  Let $f(z)$ be an analytic function defined on
  $\cP=\{z\in \mathbb{C}^n:|z-x_0|<\rho_0\}$, $ x_0\in \R^n $.
  Let
  $M(k)=\max_{|\alpha|=k}\sup|\partial^\alpha  f|$.
  Assume $\mu_{x_0}:=\|\nabla f(x_0)\|>0$. Let $ E_0=f(x_0) $.
  Let
  \begin{equation*}
	\rho_1\le c(n)\min(\rho_0,\mu_{x_0}M(2)^{-1}),\quad
    r=c(n)\rho_1,\quad r'=c(n)\rho_1\min(1,\mu_{x_0}),
  \end{equation*}
  with $ c(n) $ a sufficiently small constant. Then for any $
  (w,E)\in \C^{n-1}\times \C  $, $ |w|<r $, $ |E-E_0|<r' $, the equation
  \begin{equation*}
	f(\varphi(\xi,w;x_0))=E
  \end{equation*}
  has a unique solution $ \xi=g(w,E) $ in $ |\xi|<\rho_1 $ which is an
  analytic function of $ w,E $. Furthermore, the following statements
  hold.
  \begin{enumerate}[(a),leftmargin=2em]
  \item For any $  (w,E)\in \C^{n-1}\times \C  $, $ |w|<r $, $
    |E-E_0|<r' $ we have
    \begin{equation}\label{eg:g-bound}
      |g(w,E)|\le 2\mu_{x_0}^{-1}(|E-E_0|+C(n) M(2)|w|^2).
    \end{equation}
  \item For any $ x_0'\in \R^n $, $ \norm{x_0'-x_0}<r $, such that $
    f(x_0')=E $, $ |E-E_0|<r' $, there exists
    $ y\in \R^n $, $ \norm{y}\le \norm{x_0'-x_0} $ such that $
    x_0'=\varphi(g(y,E),y;x_0) $.
  \end{enumerate}
\end{lemma}
\begin{proof}
  The existence and uniqueness of the solution $ \xi=g(w,E) $ follows
  from \cref{lem:impl2} applied to $ F(\xi,w,E)=f(\varphi(\xi,w;x_0))-E $ on $ \cP'= \{ (\xi,w,E) :
  |\xi|,|w|,|E-E_0|<c\rho_0 \} $, with $ c $ small enough so that $
  |\varphi(\xi,w;x_0)-x_0|<\rho_0/2 $, for $ |\xi|,|w|<c\rho_0 $. Note that
  \begin{equation*}
	F(0,0,E_0)=0,\quad \partial_{\xi}F(0,0,E_0)=\mu_{x_0},\quad
    \partial_{w_i}F(0,0,E_0)=0,\quad
    \quad \partial_E F(0,0,E_0)=-1,
  \end{equation*}
  We just need to prove the claims (a),(b).

  \smallskip\noindent
  (a) Note that $ \langle \nabla f(x_0),\varphi(\xi,w;x_0)-x_0
  \rangle=\mu_{x_0}\xi $. Using Taylor's formula we have
  \begin{equation*}
	f(\varphi(\xi,w;x_0))-f(x_0)=\mu_{x_0}\xi+R(\xi,w),
  \end{equation*}
  with
  \begin{equation*}
	|R(\xi,w)|\le C(n)M(2)(|\xi|^2+|w|^2).
  \end{equation*}
  By setting $ \xi=g(w,E) $ we get
  \begin{multline*}
	|g(w,E)|=\mu_{x_0}^{-1}|E-E_0-R(g(w,E),w)|
    \le \mu_{x_0}^{-1}(|E-E_0|+C(n)M(2)(|g(w,E)|^2+|w|^2))\\
    \le \mu_{x_0}^{-1}(|E-E_0|+C(n)M(2) (\rho_1|g(w,E)|+|w|^2))\\
    \le \frac{1}{2}|g(w,E)|+\mu_{x_0}^{-1}(|E-E_0|+C(n)M(2)|w|^2),
  \end{multline*}
  provided $ \rho_1 $ is small enough, and
  \cref{eg:g-bound} follows.

  \smallskip\noindent
  (b) Let $ (\xi,y)\in \R^{n-1}\times \R $ be such that $
  x_0'=\varphi(\xi,y;x_0) $. We have $ |\xi|,|y|\le \norm{x_0'-x_0} $.
  Since $ f(\varphi(\xi,y;x_0))=E $, $ |\xi|<r<\rho_1 $, and $ |y|<r $,
  uniqueness implies that $ \xi=g(y,E) $.
\end{proof}

\begin{remark}\label{rem:real-valued}
  In \cref{lem:impl2} and \cref{lem:parametrization1}, if the function
  $ f $ is real-valued on $ \R^n $, then the implicit functions are
  also real-valued on $ \R^n $. Indeed, by the usual implicit function
  theorem, the implicit functions will be real valued on some small
  real polydisk, and by analyticity they will be real-valued on their
  whole real domain.
\end{remark}

Part (I) of \cref{prop:levelsetshifts} will follow by letting $
x(y,E)=\varphi(g(y,E),y;x_0) $, with $ g $ as in the previous
lemma. For part (II) it will be enough to prove the result with $
E=E_0 $, so we focus on this case. To simplify notation we let $ g(y):=g(y,E_0) $.
Part (II) will follow from Cartan's estimate as soon as we find a
point $ |y|\ll r $ such that
\begin{equation*}
  |f(x(y,E_0)+h)-E_0|=|f(\varphi(g(y),y;x_0)+h)-f(x_0)|
  \ge \epsilon,
\end{equation*}
with a certain $ \epsilon=\epsilon(\norm{h},\norm{x_0}) $. If $
|f(x_0+h)-f(x_0)|\ge \epsilon $, then we can simply choose $ y=0
$. We single out a simple case when this happens.

\begin{lemma}\label{lem:quadratic-domination}
  Let $f(x)$ be a smooth real function defined on $ \{ x\in \R^n :
  |x-x_0|<\rho_0 \} $. Let $M(k)=\max_{|\alpha|=k}\sup|\partial^\alpha
  f|$.
  Assume $\mathfrak{H}(x_0)\ge \nu_{x_0} I>0$
   and set
  \begin{gather*}
    \nu_1:=c(n)\nu_{x_0}(1+M(3))^{-1}.
  \end{gather*}
  with $ c(n) $ a sufficiently small constant. If $
  \nu_1^{-1}\norm{\nabla f(x_0)}\le \norm{h}< \min(\nu_1,\rho_0) $,
  then
  \begin{equation*}
	|f(x_0+h)-f(x_0)|\ge \frac{1}{4}\nu_{x_0}\norm{h}^2.
  \end{equation*}
\end{lemma}
\begin{proof}
  Using Taylor's formula and the assumptions on $ h $,
  \begin{multline*}
	|f(x_0+h)-f(x_0)|\ge \frac{1}{2}\left| \langle \fH(x_0)h,h \rangle \right|
    - \left| \langle \nabla f(x_0),h \rangle \right|-C(n)M(3)\norm{h}^3\\
    \ge \frac{1}{2}\nu_{x_0}\norm{h}^2-\nu_1\norm{h}^2-C(n)M(3)\nu_1\norm{h}^2
    \ge \frac{1}{4}\nu_{x_0}\norm{h}^2.
  \end{multline*}
\end{proof}

\begin{figure}[h]
  \centering
  \includegraphics{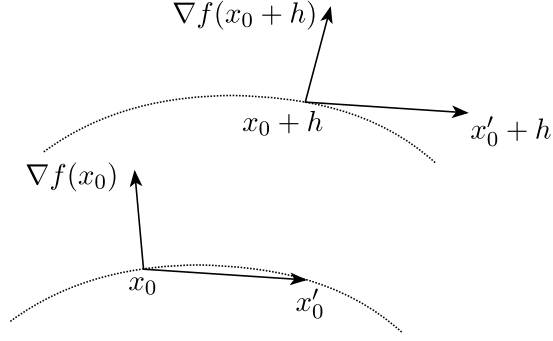}
    \caption{If $ \nabla f(x_0+h) $ is not collinear with $\nabla f(x_0)
    $, then the projection of $ x_0'-x_0 $ onto $ \nabla f(x_0+h) $ is
    relatively large. }\label{fig:non-collinear}
\end{figure}

Suppose that $ |f(x_0+h)-f(x_0)|<\epsilon $. Then we want to find $
x_0'=\varphi(g(y),y;x_0) $, $ f(x_0')=f(x_0) $, such that $
|f(x_0'+h)-f(x_0')|\ge \epsilon $.
To this end it is enough to find $ x_0' $ such that $
|f(x_0'+h)-f(x_0+h)|\ge 2\epsilon $. By Taylor's formula 
\begin{equation*}
  f(x_0'+h)-f(x_0+h)=\langle \nabla f(x_0+h),x_0'-x_0 \rangle+O(|x_0'-x_0|^2).
\end{equation*}
The linear term will dominate the quadratic term if the projection of
$ x_0'-x_0 $ onto $ \nabla f(x_0+h) $ is large relative to $ x_0'-x_0
$.
By \cref{eg:g-bound}, the projection of $ x_0'-x_0 $ onto $ \nabla
f(x_0) $ is relatively small, so the projection onto $ \{ \nabla
f(x_0) \}^\perp $ is relatively large. This means that if $ \nabla
f(x_0) $ and $ \nabla f(x_0+h) $ are not too close to being collinear,
the projection of $ x_0'-x_0 $ onto $ \nabla f(x_0+h) $ will be
relatively large (see Figure \ref{fig:non-collinear}), and
we should be able to find a lower bound on $ |f(x_0'+h)-f(x_0+h)| $ via
the linear term of the Taylor expansion. A quantitative version of
this observation is given in the next lemma.

\begin{lemma}\label{lem:deviation1} Using the notation and assumptions
  of Lemma~\ref{lem:parametrization1}
  the following hold.
  Let $ h\in \R^n $, $ |h|<\rho_0/2 $, $ x_1=x_0+h $, 
  $\mu_{x_1}:=\|\nabla f(x_1)\|$. Assume
  \begin{equation*}
	\langle\nabla f(x_1),\nabla f(x_0)\rangle^2\le
  (1-\delta_0^2)\|\nabla f(x_1)\|^2\|\nabla f(x_0)\|^2,\qquad 0<\delta_0\le 1.
  \end{equation*}
  Let  
  \begin{equation*}
    \rho\le  c(n)\min(r,\mu M(2)^{-1}\delta_0^2)\ll r,\quad \mu=\min(\mu_{x_0},\mu_{x_1})
  \end{equation*}
  where $c(n)$ is a sufficiently small constant and $ r $ as in
  \cref{lem:parametrization1}.  Then there exists $\norm{x_0'-x_0}\le
  2\rho $, $ x_0'=\varphi(g(y),y;x_0) $, $ \norm{y}\le \rho $,
  such that 
  \begin{equation*}
    |f(x_0'+h)-f(x_0+h)|\ge \frac{1}{2}\mu_{x_1}\delta_0^2 \rho.
  \end{equation*}
\end{lemma}
\begin{proof}
  The case $ \mu_{x_1}=0 $ is trivial, so we assume $ \mu_{x_1}>0 $.
  Given $\fn\in \R^n, y\in \R^{n-1} $ and using the notation of \cref{def:normal-coordinates} let
  \begin{equation}\label{eq:fp-fq}
    \fp(\fn;x_0)=\sum_j \langle \fe_{x_0,j},\fn \rangle \fe_{x_0,j},\qquad
	\mathfrak{q}(y;x_0)=\sum_j y_j \mathfrak{e}_{x_0,j}.
  \end{equation}
  Let $ \mathfrak{n}_{x_1}=\mu_{x_1}^{-1}\nabla f(x_1) $.
  We choose $ y\in \R^{n-1} $ such that $
  \mathfrak{q}(y;x_0)=\rho\mathfrak{p}(\mathfrak{n}_{x_1};x_0) $,
  with $ \rho $ as in the statement. Note that
  \begin{equation*}
	1\ge \norm{\fp(\fn_{x_1};x_0)}^2=\norm{\fn_{x_1}}^2
    -\langle \fn_{x_1},\fn_{x_0} \rangle^2\ge 1-(1-\delta_0^2)=\delta_0^2.
  \end{equation*}
  It follows that
  \begin{equation*}
	\norm{y}=\norm{\fq(y;x_0)}=\rho\norm{\fp(\fn_{x_1};x_0)}\le \rho
  \end{equation*}
  and
  \begin{equation*}
	\langle \nabla f(x_1),\fq(y;x_0) \rangle=\mu_{x_1} \langle \fn_{x_1}, \fq(y;x_0)\rangle
    =\mu_{x_1}\langle \fp(\fn_{x_1};x_0),\fq(y;x_0) \rangle=\mu_{x_1}\rho\norm{\fp(\fn_{x_1};x_0)}^2
    \ge \mu_{x_1}\delta_0^2\rho.
  \end{equation*}
  Let $ x_0'=\varphi(g(y),y;x_0) $ with $ y $ as above.
  Then
  \begin{equation*}
	\norm{x_0'-x_0}\le |g(y)|+\norm{y}\le \mu_{x_0}^{-1}C(n)M(2)\norm{y}^2+\norm{y}\le 2\norm{y}\le 2\rho,
  \end{equation*}
  provided $ \rho $ is small enough. Note that we used \cref{eg:g-bound}.
  By Taylor's formula
  \begin{multline}
	f(x_0'+h)-f(x_0+h)= \langle \nabla f(x_1),x_0'-x_0 \rangle+R(x_0'-x_0)\\
    = \langle \nabla f(x_1),g(y) \mathfrak{n}_{x_0} \rangle
    + \langle \nabla f(x_1),\mathfrak{q}(y;x_0) \rangle
    +R(x_0'-x_0),\label{eq:Taylor-h}
  \end{multline}
  with
  \begin{equation*}
	|R(x_0'-x_0)|\le C(n)M(2)\|x_0'-x_0\|^2
    \le 4C(n)M(2)\rho^2\le \frac{1}{4}\mu_{x_1}\delta_0^2 \rho.
  \end{equation*}
  We also have
  \begin{equation*}
	|\langle \nabla f(x_1),g(y)\fn_{x_0} \rangle|\le \mu_{x_1} \mu_{x_0}^{-1}C(n)M(2)\rho^2
    \le \frac{1}{4}\mu_{x_1}\delta_0^2 \rho.
  \end{equation*}
  The conclusion follows by combining the estimates we obtained for
  the terms on the left hand side of \cref{eq:Taylor-h}.  
\end{proof}

Now we have to deal with the situation when $
|f(x_0+h)-f(x_0)|<\epsilon $, and $ \nabla f(x_0) $ and $ \nabla
f(x_0+h) $ are close to being collinear. We show that for small enough
$ h $ this can only happen  if  $ h $  is very close to a particular ``bad''
direction. 

\begin{lemma}\label{lem:deviation1opposite3}
  Let $f(x)$ be a smooth real function defined on $ \{ x\in \R^n :
  |x-x_0|<\rho_0 \} $. Let $M(k)=\max_{|\alpha|=k}\sup|\partial^\alpha
  f|$.
  Assume $\mathfrak{H}(x_0)\ge \nu_{x_0} I$, $ 0<\nu_{x_0}<1 $,
  and set
  \begin{gather*}
    \nu_1:=c(n)\nu_{x_0}(1+M(2)+M(3))^{-1}
  \end{gather*}
  with $c(n)$ a sufficiently small constant.  Let $ 0<\|h\|<\min
  (\rho_0,\nu_1^6) $.
  Assume that the following conditions hold
  \begin{gather}
    |f(x_0+h)-f(x_0)|\le \|h\|^{3},\label{eq:close-shift}\\
    \|\nabla f(x_0+h)-\lambda\nabla f(x_0)\|\le \|h\|^2, \label{eq:close-shift-gradient}
  \end{gather}
  with some $ \lambda\in \R $.
  Then
  \begin{equation}\label{eq:bad-direction}
    \|h+2\mathfrak{H}(x_0)^{-1}\nabla f(x_0)\|\le \nu_1^{-8}\|\nabla f(x_0)\|^2.
  \end{equation}
\end{lemma}
\begin{proof}
  Note that \cref{eq:close-shift} together with
  \cref{lem:quadratic-domination} imply $ \norm{h}\le
  \nu_1^{-1}\norm{\nabla f(x_0)} $. In particular, this implies $
  \norm{\nabla f(x_0)}>0 $.
  
  Combining \cref{eq:close-shift-gradient} with Taylor's formula we get
  \begin{equation*}
    \|(\lambda-1)\nabla f(x_0)-\mathfrak{H}(x_0)h\|\le C(n)(1+M(3))\|h\|^2.
  \end{equation*}
  Therefore
  \begin{multline*}
    \|(\lambda-1)\mathfrak{H}(x_0)^{-1}\nabla f(x_0)-h\|
    \le \|\mathfrak{H}(x_0)^{-1}\|  C(n)(1+M(3))\|h\|^2\\
    \le \nu_{x_0}^{-1}C(n)(1+M(3))\|h\|^2\le \nu_1^{-1}\norm{h}^2.
  \end{multline*}
  Combining \cref{eq:close-shift} with Taylor's formula we get
  \begin{equation*}
	\Big|\langle\nabla f(x_0),h\rangle+\frac{1}{2} \langle \mathfrak{H}(x_0)h,h\rangle
    \Big|\le  \|h\|^3+C(n)M(3)\|h\|^3
    \le \nu_1^{-1} \norm{h}^3.
  \end{equation*}
  Let $ v=(\lambda-1)\mathfrak{H}(x_0)^{-1}\nabla f(x_0)-h $. Combining the previous two estimates yields
  \begin{multline*}
	\Big|(\lambda-1)\langle\nabla f(x_0),\mathfrak{H}(x_0)^{-1}\nabla f(x_0)\rangle
    +\frac{1}{2} (\lambda-1)^2\langle\nabla f(x_0),\mathfrak{H}(x_0)^{-1}\nabla f(x_0)\rangle
    \Big|\\
    \kern-7em \le \nu_1^{-1}\|h\|^3+\norm{\nabla f(x_0)}\norm{v}
    +\frac{1}{2}\norm{\fH(x_0)}(\norm{v}^2+2\norm{v}\norm{v+h})\\
    \le \nu_1^{-1}\norm{h}^3+\norm{\nabla f(x_0)}\nu^{-1}\norm{h}^2
    +C(n)M(2)(\nu_1^{-2}\norm{h}^4+\nu_1^{-1}\norm{h}^3)\\
    \le \nu_1^{-1}\norm{\nabla f(x_0)}\norm{h}^2+\nu_1^{-2}\norm{h}^3.
  \end{multline*}
  Since $ \langle\nabla f(x_0),\mathfrak{H}(x_0)^{-1}\nabla
  f(x_0)\rangle\ge \norm{\fH(x_0)}^{-1} \norm{\nabla f(x_0)}^2\ge
  \nu_1 \norm{\nabla f(x_0)}^2 $, it follows that
  \[
    |(\lambda-1)(\lambda+1)|\le\epsilon:=\nu_1^{-2}\norm{\nabla f(x_0)}^{-1}\norm{h}^2
    +\nu_1^{-3}\norm{\nabla f(x_0)}^{-2}\norm{h}^3 .
  \]
  Since $ \max(|\lambda-1|,|\lambda+1|)\ge 1 $, we have
  \[
    \min(|\lambda-1|,|\lambda+1|)\le \epsilon.
  \]
  If $|\lambda-1|\le \epsilon$, then
  \begin{multline*}
    \|h\|\le \|h-(\lambda-1)\mathfrak{H}(x_0)^{-1}\nabla f(x_0)\|
    +\|(\lambda-1)\mathfrak{H}(x_0)^{-1}\nabla f(x_0)\|\\
    \le \nu_1^{-1}\norm{h}^2+\nu_1^{-1}\epsilon\norm{\nabla f(x_0)}
    =\nu_1^{-1}\norm{h}^2+\nu_1^{-3}\norm{h}^2+\nu_1^{-4}\norm{\nabla f(x_0)}^{-1}\norm{h}^3
    \le  \nu_1^{-6}\norm{h}^2
  \end{multline*}
  (recall that $ \norm{h}\le \nu_1^{-1}\norm{\nabla f(x_0)} $). This
  is not compatible with the assumption that $ 0<\norm{h}<\nu_1^6
  $. So, we must have $|\lambda+1|\le \epsilon$ and therefore
  \begin{multline*}
    \|h+2\mathfrak{H}(x_0)^{-1}\nabla f(x_0)\|
    \le \|h-(\lambda-1)\mathfrak{H}(x_0)^{-1}\nabla f(x_0)\|
    +\|(\lambda+1)\mathfrak{H}(x_0)^{-1}\nabla f(x_0)\|\\
    \le \nu_1^{-1}\norm{h}^2+\nu_1^{-1}\epsilon\norm{\nabla f(x_0)}
    \le \nu_1^{-6}\norm{h}^2
    \le     \nu_1^{-8}\|\nabla f(x_0)\|^2.
  \end{multline*}
\end{proof}

Finally, we show that \cref{eq:bad-direction} cannot hold over the
entire piece of the $ f(x_0) $-level set parametrized in \cref{lem:parametrization1}.
\begin{lemma}\label{lem:Morse10}
  Let $f(x)$ be a smooth real function defined on $ \{ x\in \R^n :
  |x-x_0|<\rho_0 \} $, $ \rho_0<1 $. Let $M(k)=\max_{|\alpha|=k}\sup|\partial^\alpha
  f|$.
  Assume $\mathfrak{H}(x_0)\ge \nu_{x_0} I$, $ 0<\nu_{x_0}<1 $, and
  $ 0<\norm{\nabla f(x_0)}< \rho_0\nu_1^{9}/20 $ with 
  \begin{gather*}
    \nu_1:=c(n)\nu_{x_0}(1+M(2)+M(3))^{-1}
  \end{gather*}
  with $c(n)$ a sufficiently small constant. Then there exists $\norm{x_0'-x_0}\ll
  r$, with $ r $ as in \cref{lem:parametrization1}, $
  x_0'=\varphi(g(y),y;x_0) $, $ \norm{y}\ll r $,
  such that
  \begin{equation*}
    \norm{\fH(x_0')^{-1}\nabla f(x_0')-\fH(x_0)^{-1}\nabla f(x_0)}
      > \nu_1^{-8}(\norm{\nabla f(x_0)}^2+\norm{\nabla f(x_0')}^2).
  \end{equation*}
\end{lemma}
\begin{proof}
  Choose $ y\in \R^{n-1} $ such that $ \norm{y}=\nu_1\norm{\nabla
    f(x_0)} $ and let $ x_0'=\varphi(g(y),y;x_0) $. Using \cref{eg:g-bound} we have
  \begin{equation*}
	\norm{x_0'-x_0}\le |g(y)|+\norm{y}\le \mu_{x_0}^{-1}C(n)M(2)\norm{y}^2+\norm{y}
    \le 2\nu_1\norm{\nabla f(x_0)}\ll r,
  \end{equation*}
  provided $ \nu_1 $ is small enough.
  Then
  \begin{equation*}
	\norm{\fH(x_0')-\fH(x_0)}\le C(n)M(3)\norm{x_0'-x_0}\le \norm{\nabla f(x_0)}\le \frac{\nu_{x_0}}{2}
  \end{equation*}
  (recall that $ \norm{\nabla f(x_0)}<\rho_0\nu_1^{10} $, $ \rho_0<1 $)
  and therefore $ \fH(x_0')\ge \frac{\nu_{x_0}}{2}I $ and $
  \norm{\fH(x_0')^{-1}}\le 2\nu_{x_0}^{-1} $.
  We have
  \begin{multline*}
	\norm{\fH(x_0')^{-1}\nabla f(x_0')-\fH(x_0)^{-1}\nabla f(x_0)}\\
    \ge \norm{\fH(x_0')^{-1}(\nabla f(x_0)-\nabla f(x_0'))}
    -\norm{(\fH(x_0')^{-1}-\fH(x_0)^{-1})\nabla f(x_0)}.
  \end{multline*}
  On one hand using Taylor's formula applied to the gradient we get
  \begin{multline*}
	\norm{\fH(x_0')^{-1}(\nabla f(x_0)-\nabla f(x_0'))}
    \ge \norm{x_0'-x_0}-\norm{\fH(x_0')^{-1}}C(n)M(3)\norm{x_0'-x_0}^2\\
    \ge \norm{x_0'-x_0}-\nu_1^{-1}\norm{x_0'-x_0}^2\ge \frac{\norm{x_0'-x_0}}{2}
    \ge \frac{\nu_1\norm{\nabla f(x_0)}}{2}.
  \end{multline*}
  On the other hand
  \begin{multline*}
	\norm{(\fH(x_0')^{-1}-\fH(x_0)^{-1})\nabla f(x_0)}\\
    \le\norm{\fH(x_0')^{-1}}\norm{\fH(x_0)^{-1}}\norm{(\fH(x_0')-\fH(x_0)}\norm{\nabla f(x_0)}
    \le \nu_1^{-1}\norm{\nabla f(x_0)}^2.
  \end{multline*}
  Therefore
  \begin{multline*}
	\norm{\fH(x_0')^{-1}\nabla f(x_0')-\fH(x_0)^{-1}\nabla f(x_0)}\\
    \ge \frac{\nu_1\norm{\nabla f(x_0)}}{2}-\nu_1^{-1}\norm{\nabla f(x_0)}^2
    \ge \frac{\nu_1\norm{\nabla f(x_0)}}{4}
    > 5\nu_1^{-8} \norm{\nabla f(x_0)}^2.
  \end{multline*}
  Since
  \begin{equation*}
	\norm{\nabla f(x_0)-\nabla f(x_0')}\le C(n)M(2)\norm{x_0'-x_0}\le 2C(n)M(2)\nu_1\norm{\nabla f(x_0)}
    \le \norm{\nabla f(x_0)},
  \end{equation*}
  we get that
  \begin{equation*}
	\nu_1^{-8}(\norm{\nabla f(x_0')}^2+\norm{\nabla f(x_0)}^2)\le 5\nu_1^{-8}\norm{\nabla f(x_0)}^2,
  \end{equation*}
  and the conclusion follows.
\end{proof}

We will use the following simple consequence of Taylor's formula. We
leave the proof as a simple exercise.
\begin{lemma}\label{lem:Morse7}  Let $f(x)$ be a smooth real function defined on
  $\{x\in \mathbb{R}^n:|x|<r_0\}$. Assume that
  \begin{gather*}
    f(0)=0,\quad \nabla f(0)=0,\\
    \mathfrak{H}(0)\ge \nu_0I,\quad \nu_0>0.
  \end{gather*}
  Let
  $M(k)=\max_{|\alpha|=k}\sup_x|\partial^\alpha f|$.
  Then for $|x|<\min(r_0,c(n)\nu_0M(3)^{-1})$, with $c(n)$ a
  sufficiently small constant, we have
  \begin{gather*}
    \frac{\nu_0}{2}\|x\|^2\le f(x)\le (C(n)M(2)+1)\|x\|^2,\\
    \frac{\nu_0}{2}\|x\|\le \|\nabla f(x)\|\le (C(n)M(2)+1)\|x\|,\\
    \mathfrak{H}(x)\ge \frac{\nu_0}{2}I.
  \end{gather*}
\end{lemma}

Now we prove Proposition~\ref{prop:levelsetshifts}.
\begin{proof}[Proof of Proposition~\ref{prop:levelsetshifts}]
  Let $ 0<\norm{x_0}<\rho $, $ E_0=f(x_0) $. Using
  \cref{lem:Morse7} we have
  \begin{equation}\label{eq:grad-x0}
	\frac{\nu_0}{2}\norm{x_0}\le \norm{\nabla f(x_0)}\le (C(n)M(2)+1)\norm{x_0}\le \frac{r_0}{2}\nu_1^{9}
    \ll 1.
  \end{equation}
  Let $ \mu_{x_0}=\norm{\nabla f(x_0)} $, 
  \begin{equation*}
	\tilde\rho_1=\tilde c(n)\min(r_0,\mu_{x_0}M(2)^{-1}),\quad
    \tilde r=\tilde c(n)\tilde \rho_1,\quad
    \tilde r'=\tilde c(n)\rho_1\min(1,\mu_{x_0}),
  \end{equation*}
  with $ \tilde c(n) $ standing for the $ c(n) $ constant from \cref{lem:parametrization1}.
  By \cref{lem:parametrization1}, for any $
  (w,E)\in \C^{n-1}\times \C  $, $ |w|<\tilde r $, $ |E-E_0|<\tilde r' $, the equation
  \begin{equation*}
	f(\varphi(\xi,w;x_0))=E
  \end{equation*}
  has a unique solution $ \xi=g(w,E) $ in $ |\xi|<\tilde \rho_1 $ which is an
  analytic function of $ w,E $. Note that by the smallness of $ x_0 $ we have
  \begin{gather}
	\tilde\rho_1=\tilde c(n)\mu_{x_0}M(2)^{-1},\quad
    \tilde r=\tilde c(n)^2\mu_{x_0}M(2)^{-1},\label{eq:r-tilde}
    \quad \tilde r'=\tilde c(n)^2 \mu_{x_0}^2 M(2)^{-1},\\
    r\ll \tilde r,\quad  r^2\ll \tilde r'\label{eq:r-r-tilde}
  \end{gather}
  (we used the fact that $ M(2)\ge c(n)\nu_0 $).    
  By \cref{eg:g-bound},
  \begin{multline}\label{eq:x-bound}
	\norm{\varphi(g(w,E),w;x_0)-x_0}\le |g(w,E)|+\norm{w}\\
    \le 2\mu_{x_0}^{-1}(r^2+C(n)M(2)r^2)+r\sqrt{n-1}
    <\frac{1}{2}\nu_1^{-1}r=\frac{1}{2}\norm{x_0},
  \end{multline}
  for any $ |w|<r $, $ |E-E_0|<r^2 $.
  Now part (I) follows by setting $ x(w,E)=\varphi(g(w,E),w;x_0) $.

  We first prove (II) with $ E=E_0 $. Let $ 0<\norm{h}<\rho $. We claim that there exists $ y_0
  $, $ \|y_0\|\ll \tilde r $, such that
  \begin{equation*}
	|f(x(y_0,E_0)+h)-E_0|=|f(\varphi(g(y_0),y_0;x_0)+h)-f(x_0)|\ge \norm{h}^8\norm{x_0}.
  \end{equation*}
  From the claim (also  note that $
  |f(x(w,E)+h)-E|\ll 1 $), \cref{lem:high_cart}, and \cref{lem:Cartan-measure} it follows that for $ H\gg 1
  $ we have
  \begin{multline*}
	\mes \{ y\in \R^{n-1}, |y|<r : \log|f(x(y,E_0)+h)-E_0|\ge C(n) H\log(\norm{h}\norm{x_0}) \}\\
    \le C(n)\tilde r^{n-1}\exp(-H^{\frac{1}{n-1}})
    \le (\nu_1^{-2}r)^{n-1}\exp(-H^{\frac{1}{n-1}})
  \end{multline*}
  as stated in \cref{prop:levelsetshifts}
  (recall \cref{eq:grad-x0},
  \cref{eq:r-tilde}, \cref{eq:r-r-tilde}). Now we check the claim.
  Let $ x_1=x_0+h $, $ \mu_{x_1}=\norm{\nabla f(x_1)} $. If
  $ |f(x_1)-f(x_0)|>\norm{h}^8\norm{x_0} $, the claim holds with $ y_0=0
  $. Suppose
  \begin{equation*}
    |f(x_1)-f(x_0)|\le \norm{h}^8\norm{x_0}\qquad\text{and}\qquad
    \norm{\nabla f(x_1)-\lambda \nabla f(x_0)}>\norm{h}^2,\quad
    \lambda
    =\frac{\langle \nabla f(x_0),\nabla f(x_1) \rangle}{\langle \nabla f(x_0),\nabla f(x_0) \rangle}.
  \end{equation*}
  Then a direct computation yields
  \begin{equation*}
	\langle\nabla f(x_1),\nabla f(x_0)\rangle^2\le
    (1-\delta_0^2)\|\nabla f(x_1)\|^2\|\nabla f(x_0)\|^2,\quad
    \delta_0=\frac{\norm{h}^4}{\mu_{x_1}^2}.
  \end{equation*}
  Note that
  \begin{equation*}
	\norm{\nabla f(x_1)}\ge \norm{\nabla f(x_1)-\lambda \nabla f(x_0)}>\norm{h}^2.
  \end{equation*}
  We choose a small enough constant $ c(n) $ such that
  \cref{lem:deviation1} applies with
  \begin{equation*}
	\tilde \rho= c(n)\mu M(2)^{-1}\delta_0^2,\quad \mu=\min(\mu_{x_0},\mu_{x_1})
  \end{equation*}
  instead of $ \rho $, $ \rho_0=r_0/2 $, $ \tilde r $ instead of $ r $, and $
  \delta_0 $ as above. 
  Applying \cref{lem:deviation1} we get that there exists $ y $, $
  \|y\|\le \tilde \rho \ll \tilde r $, such that
  \begin{equation*}
	|f(\varphi(g(y),y;x_0)+h)-f(x_1)|\ge \frac{1}{2}\mu_{x_1}\delta_0^2 \tilde \rho
    \ge c(n)M(2)^{-1}\mu_{x_1}\mu \delta_0^4  
    \ge 2\norm{h}^{8} \norm{x_0}.
  \end{equation*} 
  We used \cref{eq:grad-x0} and the fact that
  \begin{equation*}
	\norm{\nabla f(x_1)}\le (C(n)M(2)+1)\norm{x_0+h}\le r_0\nu_1^9\ll 1.
  \end{equation*}
  Since $ |f(x_1)-f(x_0)|\le \norm{h}^8\norm{x_0} $, the claim follows with $ y_0=y $.

  We are left with the case when
  \begin{equation*}
	|f(x_1)-f(x_0)|\le \norm{h}^8\norm{x_0}\quad \text{and} \quad
    \norm{\nabla f(x_1)-\lambda \nabla f(x_0)}\le \norm{h}^2. 
  \end{equation*}
  Note that by \cref{lem:Morse7}, $ \fH(x_0')\ge \frac{\nu_0}{2}I $
  for any $ \norm{x_0'-x_0}<\tilde r $. Choosing
  sufficiently small constants  $ c(n) $ we can apply
  \cref{lem:deviation1opposite3} and \cref{lem:Morse10} with the same
  $ \nu_1 $ as in \cref{prop:levelsetshifts}. Furthermore, we can
  apply \cref{lem:deviation1opposite3} with any $
  \norm{x_0'-x_0}<\tilde r $ instead of $ x_0 $.  
  \cref{lem:deviation1opposite3} and \cref{lem:Morse10} imply that
  there exists
  \begin{equation*}
	\norm{x_0'-x_0}\ll \tilde r,\quad f(x_0')=f(x_0),
    \quad x_0'=\varphi(g(y'),y';x_0),\ \norm{y'}\ll \tilde r,
  \end{equation*}
  such that
  \begin{equation*}
	\norm{h+2\fH(x_0')^{-1}\nabla f(x_0')}>\nu_1^{-8}\norm{\nabla f(x_0')}^2.
  \end{equation*}
  \cref{lem:deviation1opposite3} (with $ x_0' $ instead of $ x_0 $)
  implies that
  \begin{equation*}
	|f(x_1')-f(x_0')|> \norm{h}^3\quad \text{or} \quad
    \norm{\nabla f(x_1')-\lambda' \nabla f(x_0')}> \norm{h}^2,
  \end{equation*}
  with
  \begin{equation*}
	x_1'=x_0'+h,\quad
    \lambda'
    =\frac{\langle \nabla f(x_0'),\nabla f(x_1') \rangle}{\langle \nabla f(x_0'),\nabla f(x_0') \rangle}
  \end{equation*}
  If $ |f(x_1')-f(x_0')|> \norm{h}^3 $, the claim holds with $ y_0=y'
  $. If $ \norm{\nabla f(x_1')-\lambda' \nabla f(x_0')}> \norm{h}^2 $,
  the reasoning above, based on \cref{lem:deviation1}, implies that
  there exists $ \norm{x_0''-x_0'}\le 2\tilde \rho'\ll \tilde r $,
  \begin{equation*}
	\tilde \rho'=c(n)\mu'M(2)^{-1}(\delta_0')^2,
    \quad \mu'=\min(\mu_{x_0'},\mu_{x_1'},\mu_{x_0}),
    \quad (\delta_0')^2=\frac{\norm{h}^4}{\mu_{x_1'}^2},
  \end{equation*}
  such that $ f(x_0'')=f(x_0')=f(x_0) $ and
  \begin{equation*}
	|f(x_0''+h)-f(x_0'+h)|\ge \frac{1}{2}\mu_{x_1'}(\delta_0')^2\tilde \rho'\ge 2\norm{h}^8\norm{x_0}.
  \end{equation*}
  Note that we added $ \mu_{x_0} $ to the definition of $ \mu' $ to
  ensure $ \tilde \rho'\ll \tilde r $, and we used the fact that $
  \norm{x_0'}\ge \norm{x_0}/2 $. 
  We now have that either
  \begin{equation*}
	|f(x_0'+h)-f(x_0')|> \norm{h}^8\norm{x_0}\quad\text{or}\quad|f(x_0''+h)-f(x_0'')|>\norm{h}^8\norm{x_0}.
  \end{equation*}
  Since $ \norm{x_0''-x_0}\ll \tilde r $, \cref{lem:parametrization1}
  implies that there exists $ y'' $, $ \norm{y''}\ll \tilde r $, such
  that $ x_0''=\varphi(g(y''),y'';x_0) $. Therefore the claim holds
  with either $ y_0=y' $ or $ y_0=y'' $.

  Next we consider part (II) with $ |E-E_0|<r^2 $. Let $
  x_0'=\varphi(g(0,E),0;x_0) $. Repeating the above argument with $
  x_0' $ instead of $ x_0 $ we get that there exists $ y_0' $,
  \begin{equation*}
	\|y_0'\|\ll \tilde c(n)^2\mu_{x_0'}M(2)^{-1}
  \end{equation*}
  (recall that $ \tilde r= \tilde c(n)^2\mu_{x_0}M(2)^{-1}$) such that
  \begin{equation*}
	|f(\varphi(g(y_0';x_0'),y_0';x_0')+h)-E|\ge \norm{h}^8\norm{x_0'}.
  \end{equation*}
  We used $ g(y;x_0') $ to denote the analogue of $ g(y) $ obtained by
  applying \cref{lem:parametrization1} with $ x_0' $ replacing $ x_0
  $. By \cref{eq:x-bound} we have $ \norm{x_0'}\ge \norm{x_0}/2 $. Let
  $ x_0''=\varphi(g(y_0';x_0'),y_0';x_0') $. Note that $
  f(x_0'')=f(x_0')=E $. We have
  \begin{equation*}
	\norm{x_0''-x_0}\le \norm{x_0'-x_0}+|g(y_0';x_0')|+\norm{y_0'}.
  \end{equation*}
  Using \cref{eg:g-bound} we get
  \begin{equation*}
	\norm{x_0'-x_0}=|g(0,E)|\le 2\mu_{x_0}^{-1}|E-E_0|
    \le 2\left( \nu_0\norm{x_0}/2 \right)^{-1}r^2\le r\ll \tilde r
  \end{equation*}
  and
  \begin{equation*}
	|g(y_0';x_0')|\le \mu_{x_0'}^{-1}C(n)M(2)\norm{y_0'}^2
    \le \mu_{x_0'}^{-1}C(n)M(2)\tilde c(n)^2\mu_{x_0'}M(2)^{-1}\norm{y_0'}
    \le \norm{y_0'},
  \end{equation*}
  provided $ \tilde c(n) $ is made small enough.
  Since
  \begin{equation*}
	|\mu_{x_0'}-\mu_{x_0}|
    \le C(n)M(2)\norm{x_0'-x_0}\le 
    C(n)M(2)\left( \nu_0\norm{x_0}/2 \right)^{-1}r^2
    \le \nu_0\norm{x_0}/2\le \mu_{x_0},
  \end{equation*}
  we have
  \begin{equation*}
	\norm{y_0'}\ll \tilde c(n)^2\mu_{x_0'}M(2)^{-1}\le 2\tilde r.
  \end{equation*}
  Therefore we have
  \begin{equation*}
	\norm{x_0''-x_0}\ll \tilde r.
  \end{equation*}
  By \cref{lem:parametrization1} there exists $ y_0 $, $ \norm{y_0}\ll
  \tilde r$,  such that $ x_0''=\varphi(g(y_0,E),y_0;x_0) $. Since
  \begin{equation*}
	|f(\varphi(g(y_0,E),y_0;x_0)+h)-E|\ge \norm{h}^8\norm{x_0}/2,
  \end{equation*}
  the conclusion follows as above from Cartan's estimate.

  Next we prove (III) with $ E=E_0 $. We will argue that there exists
  $ y_0 $, $ \norm{y_0}\ll \tilde r $, such that
  \begin{equation}\label{eq:directional-claim}
	\log|\langle f(x(y_0,E_0)),h_0 \rangle|\gtrsim \nu_1\norm{x_0}.
  \end{equation}
  Recall that $ x(y,E_0)=\varphi(g(y),y;x_0) $. If $ |\langle
  \nabla f(x_0),h_0 \rangle|\ge \norm{x_0}^2 $, we take $ y_0=0 $. We
  just need to deal with the case
  \begin{equation}\label{eq:directional-case}
	|\langle \nabla f(x_0),h_0 \rangle|<\norm{x_0}^2.
  \end{equation}
  Let $ x_0'=\varphi(g(y),y;x_0) $, with $ y $ to be specified later. By Taylor's formula
  \begin{multline*}
	|\langle \nabla f(x_0'),h_0 \rangle- \langle \nabla f(x_0),h_0 \rangle|
    \ge |\langle \fH(x_0)(x_0'-x_0),h_0) \rangle|-C(n)M(3)\norm{x_0'-x_0}^2\\
    = |\langle (x_0'-x_0),\fH(x_0)h_0) \rangle|-C(n)M(3)\norm{x_0'-x_0}^2.
  \end{multline*}
  Using the notation from \cref{eq:fp-fq} we write
  \begin{equation*}
	\fH(x_0)h_0=\alpha_0\fn_{x_0}+\fp(\fH(x_0)h_0;x_0)
  \end{equation*}
  and we choose $ y $ such that $ \fq(y;x_0)=\rho \fp(\fH(x_0)h_0;x_0)
  $, $ \rho=\nu_1^2\norm{x_0} $. Note that $ \norm{y}\le r\ll \tilde r
  $ and
  \begin{equation*}
	\langle (x_0'-x_0),\fH(x_0)h_0) \rangle
    =\alpha_0 g(y)+\langle \fq(y;x_0),\fp(\fH(x_0)h_0;x_0) \rangle
    =\alpha_0 g(y)+\rho \norm{\fp(\fH(x_0)h_0;x_0)}^2.
  \end{equation*}
  Using \cref{eg:g-bound} it follows that
  \begin{multline*}
	|\langle \nabla f(x_0'),h_0 \rangle- \langle \nabla f(x_0),h_0 \rangle|
    \ge \rho \norm{\fp(\fH(x_0)h_0;x_0)}^2-|\alpha_0 g(y)|-C(n)M(3)(|g(y)|^2+\norm{y}^2)\\
    \ge \frac{\rho}{2}\norm{\fp(\fH(x_0)h_0;x_0)}^2
  \end{multline*}
  (note that $ |\alpha_0|\le \norm{\fH(x_0)h_0}\le C(n)M(2) $).
  We claim that $ \norm{\fp(\fH(x_0)h_0;x_0)}\ge \norm{x_0} $. We
  argue by contradiction. Assume that
  \begin{equation*}
    \norm{\fH(x_0)h_0-\alpha_0\fn_{x_0}}=\norm{\fp(\fH(x_0)h_0;x_0)}<\norm{x_0}.
  \end{equation*}
  By Taylor's formula (recall that $ \nabla f(0)=0 $)
  \begin{equation*}
	\norm{\nabla f(x_0)-\fH(x_0)x_0}\le C(n)M(3)\norm{x_0}^2.
  \end{equation*}
  So, using \cref{eq:grad-x0} we have
  \begin{equation*}
	\norm{\fn_{x_0}-\mu_{x_0}^{-1}\fH(x_0)x_0}\le \mu_{x_0}^{-1}C(n)M(3)\norm{x_0}^2
    \le \nu_1^{-1}\norm{x_0},
  \end{equation*}
  and using \cref{eq:directional-case} we have
  \begin{equation*}
	|\langle \fH(x_0)h_0,x_0 \rangle|= |\langle \fH(x_0)x_0,h_0 \rangle|\le (C(n)M(3)+1)\norm{x_0}^2
    \le \nu_1^{-1}\norm{x_0}^2.
  \end{equation*}
  Now we have
  \begin{equation*}
	\norm{\fH(x_0)h_0-\alpha_0\mu_{x_0}^{-1}\fH(x_0)x_0}\le (1+\alpha_0\nu_1^{-1})\norm{x_0}
    \le \nu_1^{-2}\norm{x_0}.
  \end{equation*}
  and therefore
  \begin{multline}\label{eq:directional-estimate-1}
	|\langle \alpha_0\mu_{x_0}^{-1}\fH(x_0)x_0,x_0 \rangle|
    \le |\langle \fH(x_0)h_0-\alpha_0\mu_{x_0}^{-1}\fH(x_0)x_0,x_0 \rangle|
    +|\langle \fH(x_0)h_0,x_0 \rangle|\\
    \le \nu_1^{-2}\norm{x_0}^2+\nu_1^{-1}\norm{x_0}^2\le 2\nu_1^{-2}\norm{x_0}^2.
  \end{multline}
  On the other hand
  \begin{equation}\label{eq:directional-estimate-2}
	|\langle \alpha_0\mu_{x_0}^{-1}\fH(x_0)x_0,x_0 \rangle|
    \ge |\alpha_0|\mu_{x_0}^{-1}\frac{\nu_0}{2}\norm{x_0}^2\ge \nu_1^3 \norm{x_0}.
  \end{equation}
  We used \cref{lem:Morse7}, \cref{eq:grad-x0}, and the fact that
  \begin{equation*}
	|\alpha_0|^2=\norm{\fH(x_0)h_0}^2-\norm{\fp(\fH(x_0)h_0;x_0)}^2
    \ge (\nu_0/2)^{-2}-\norm{x_0}^2\ge \nu_0^{-2}
  \end{equation*}
  (recall that $ \norm{x_0}<\nu_1^{11}\ll \nu_0 $). The estimates
  \cref{eq:directional-estimate-1} and
  \cref{eq:directional-estimate-2} are incompatible due to the
  smallness of $ x_0 $. Therefore we have $
  \norm{\fp(\fH(x_0)h_0;x_0)}\ge \norm{x_0} $ and
  \begin{equation*}
	|\langle \nabla f(x_0'),h_0 \rangle- \langle \nabla f(x_0),h_0 \rangle|
    \gtrsim \rho\norm{x_0}^2=\nu_1^2\norm{x_0}^3.
  \end{equation*}
  This shows that \cref{eq:directional-claim} must hold either with $
  y_0=0 $ or $ y_0=y $.   From \cref{eq:directional-claim} (also  note that $
  \norm{\nabla f(x(w,E))}\ll 1 $),
  \cref{lem:high_cart}, and \cref{lem:Cartan-measure} it follows that for $ H\gg 1
  $ we have
  \begin{multline*}
	\mes \{ y\in \R^{n-1}, |y|<r : \log|\langle \nabla f(x(y,E_0)),h_0 \rangle|
    \ge C(n) H\log(\nu_1\norm{x_0}) \}\\
    \le C(n)\tilde r^{n-1}\exp(-H^{\frac{1}{n-1}})
    \le (\nu_1^{-2}r)^{n-1}\exp(-H^{\frac{1}{n-1}})
  \end{multline*}
  as stated in \cref{prop:levelsetshifts}. The case $ |E-E_0|<r^2 $
  follows from the case $ E=E_0 $ analogously to  the proof of (II).
\end{proof}

\section{Inductive Scheme for the Bulk of the Spectrum}\label{sec:bulk}

In this section we assume the same non-perturbative setting as in \cref{sec:basic-tools}.
We introduce five conditions such that once they hold at a large
enough initial scale they can be propagated to arbitrarily large
scales (see \cref{thm:D} below) and lead to the formation of an interval in
the spectrum, away from the edges (see \cref{thm:B} in \cref{sec:main-thm}).

For the statement of the conditions we need several
exponents. Let $ \sigma\ll \tau\ll 1 $ be as in (LDT). Set
$\delta=(\sigma')^{C_0}$, $\beta=(\sigma')^{C_1}$, $
\mu=(\sigma')^{C_2} $ with $ 0< \sigma'\le \sigma $, and
$ C_0,C_1,C_2>1 $,  satisfying the following relations:
\begin{equation*}
  C_1+1<C_2<C_0<2C_1.
\end{equation*}
Then we have
\begin{equation}\label{eq:relations}
  \beta^2\ll \delta\ll \mu \ll \beta\sigma\ll \beta\ll \sigma,
\end{equation}
with the constants implied by $ \ll $ being as large as we wish,
provided we take $ \sigma'\le c(C_0,C_1,C_2)\sigma $ small enough. 
The specific choice of the exponents $\delta,\beta,\mu$ is not
important. However, to carry out the induction with our set-up we will
need that \cref{eq:relations} holds.

Let  $ \gamma>0 $. Given an integer $ s\ge 0 $, let
\begin{equation*}
  E_s\in \R,\quad  N_s\in \N,\quad r_s:=\exp(-N_s^{\delta}).
\end{equation*}
The inductive conditions are as follows.

\medskip
\noindent\cond{A}
There exist integers $|N'_s-N_s|,|N''_s-N_s|<N_s^{1/2}$,
a map $ x_s:\Pi_s\to \R^d $,
\begin{equation*}
  \Pi_s=\cI_{s}\times (E_s-r_s,E_s+r_s),\quad \cI_{s}=\phi_s+(-r_s,r_s)^{d-1},
\end{equation*}
and $k_s $ such that
for any $(\phi,E)\in \Pi_{s}$ we have
  \begin{gather}\label{eq:mapx1}
   E_{k_s}^{[-N'_{s},N''_{s}]}(x_{s}(\phi,E))=E,\\
   \label{eq:mapx2} \big|E_j^{[-N'_{s},N''_s]}(x_{s}(\phi,E))-E\big|>
   \exp(-N^{\delta}_s),\quad j\neq k_s.
  \end{gather}
  To simplify notation we suppress $k_s$ and use
 $E^{[-N'_{s},N''_{s}]},\psi^{[-N'_s,N''_s]}$ instead.

\medskip 
\noindent\cond{B} The map $x_{s}(\phi,E)$ extends analytically on the domain
\begin{equation}\label{eq:mapx5}
  \mathcal{P}_{s}=\{(\phi,E)\in \mathbb{C}^d:\dist((\phi,E),\Pi_s)<r_s\}
\end{equation}
(the distance is with respect to the sup-norm)
and
\begin{equation}\label{eq:mapx6}
  x_s(\cP_s)\subset \T_{\rho/2}^d.
\end{equation}

\medskip
\noindent\cond{C} For each $(\phi,E)\in \Pi_{s}$, 
\begin{equation}\label{eq:mapx3}
  |\psi^{[-N'_s,N''_s]}(x_{s}(\phi,E),n)|\le \exp(-\gamma | n|/10), \quad |n|\ge N_s/4.
\end{equation}

\medskip
\noindent\cond{D}
Define
\begin{equation}\label{eq:fT_0}
  \mathfrak{T}_{s}=\{n\omega:0\le |n|\le 3N_s/2\}.
\end{equation}
Take an arbitrary $h\in \mathbb{R}^d$ with $ \dist (h,
\mathfrak{T}_{s})\ge \exp (-N_s^\mu) $. Then for any $ E\in
(E_s-r_s,E_s+r_s) $,
\begin{multline*}
  \mes \Big\{\phi\in \cI_{s}:\max_{|n'|,|n''|< N_s^{1/2}}
  \dist (\spec  H_{[-N_s+n',N_s+n'']}(x_{s}(\phi,E)+h),E)<\exp(-N_s^\beta/2)\Big\}\\
  <\exp (-N_s^{2\delta})
\end{multline*}

\medskip
\noindent\cond{E} Take an arbitrary unit vector $h_0\in \mathbb{R}^d$.
Then for any $E\in(E_s-r_s,E_s+r_s)$,
\[
  \mes \{\phi\in \cI_{s}:\log |\langle\nabla E^{[-N'_s,N''_s]}(x_{s}(\phi,E)),h_0\rangle|
  <-N_s^{\mu}/2\}<\exp (-N_s^{2\delta}).
\]

\begin{remark}\label{rem:A-E-comments}
  (a) From the proof of \cref{prop:inductive4} below it will become
  clear that in \refcond{A} it would be enough to have separation of
  eigenvalues by $ \exp(-N_s^\beta) $. However, it will also be clear
  that even if we have separation by $ \exp(-N_0^\beta) $, for $ s=0 $, we
  will still get separation by $ \exp(-N_s^\delta) $, for $ s\ge 1 $. 
  
  \smallskip\noindent
  (b) The fact that condition \refcond{B} also increases the domain of $ x_s $  in $
  \R^d $ is not accidental. This buffer around the original domain is convenient for
  Cauchy estimates and for avoiding problems with ``over-shooting'' the
  domain of $ x_s $ in the $ E $ variable.

  \smallskip\noindent (c) The particular choices of the
  $ \exp(-N_s^\beta/2) $ cutoff in \refcond{D} and of the
  $ -N_s^\mu/2 $ cutoff in \refcond{E} are made out of technical convenience.
  Specifically, the first choice allows us to have \cref{lem:sa} with a
  $ \exp(-N_s^\beta) $ cutoff, and the second choice spares us one
  application of Cartan's estimate in \cref{lem:inductive9}.

  \smallskip\noindent
  (d) For the measure estimate from \refcond{D} to be possible we need
  that the intervals $ h+[-N_s+n',N_s+n''] $ do not overlap the
  localization centre from \refcond{C}. This is the reason for the
  choice of $ \fT_s $.

  \smallskip\noindent
  (e) The reason for working with non-symmetric intervals $
  [-N_s',N_s''] $, as well as for the set being used in \refcond{D} is
  explained in \cref{rem:non-symmetric} below.
\end{remark}

To simplify notation, the dependence of the constants in this section
on the choice of the exponents $ \delta,\beta,\mu $ will be kept
implicit as part of the dependence on the parameters $ a,b $ of the
Diophantine condition.

\begin{thmd}\label{thm:D}
  Assume the notation of the inductive conditions. Let $ E_0\in \R $, and assume
  $ L(E)>\gamma>0 $ for $ E\in(E_0-2r_0,E_0+2r_0) $. Let
  $ N_0\ge 1 $, $ N_s=\lfloor N_{s-1}^A\rfloor $, $ A=\beta^{-1} $,
  $ s\ge 1 $. If $ N_0\ge (B_0+ S_V+ \gamma^{-1})^C $, $
  C=C(a,b,\rho) $, and conditions \refcond{A}-\refcond{E}
  hold with $ s=0 $, then for any $s\ge 1 $
  and $ E_s\in (E_{s-1}-r_{s-1},E_{s-1}+r_{s-1}) $ the conditions
  \refcond{A}-\refcond{E} also hold
  with $ \cI_{s}\Subset \cI_{s-1} $. Furthermore, for any $ (\phi,E)\in \Pi_{s} $,
  \begin{gather}
    \label{inductive50E}
    |x_{s}(\phi,E)-x_{s-1}(\phi,E)|<\exp(-\gamma N_{s-1}/30),\\
    \label{inductive51EE}
    \|\psi^{[-N'_s,N''_s]}(x_{s}(\phi,E),\cdot)-\psi^{[-N'_{s-1},N''_{s-1}]}(x_{s-1}(\phi,E),\cdot)\|<
    \exp(-\gamma N_{s-1}/40).
  \end{gather}
\end{thmd}

\begin{remark}\label{rem:choice-of-A}
  \cref{thm:D} also holds with any $ A\ge \beta^{-1} $, but the
  relations \cref{eq:relations} would need to be adjusted. The reason
  for needing $ A\ge \beta^{-1} $ will become clear at the end of the
  proof of \cref{prop:inductive8} below (see \cref{rem:non-symmetric}).
\end{remark}

We split the proof into several auxiliary statements. Ultimately the
theorem will follow by referring to these statements.  We will check
the theorem for the case $ s=1 $. The inductive conditions and the
auxiliary statements are designed so that the general inductive step
follows from this particular one by simply changing indices. In what
follows we fix $ E_0 $, $ N_0 $, such that the assumptions
of \cref{thm:D} are satisfied. We also fix
$ E_1\in (E_0-r_0,E_0+r_0) $ and let $ N_1,A $ be as in the statement.

For simplicity, in all of the following statements we assume tacitly
that $ N_0 $ is large enough. More precisely we assume $ N_0\ge (B_0+
S_V+ \gamma^{-1})^C $, with $ C=C(a,b,\rho) $ large enough. In
particular this allows us to invoke any of the results from
\cref{sec:basic-tools}. It will be clear from the proofs that any
further largeness constraints on $ N_0 $ can be accounted for by
increasing $ C $. Of course, it is then important that we only have
finitely many additional constraints. To this end we note that the
additional constraints are independent of $ s $.

Our first goal is to identify $ [-N_1',N_1''] $ and $
E_{k_1}^{[-N_1',N_1'']} $. In what follows we let $ \cB_{0,E,h} $ be
the set from the measure estimate in condition \refcond{D}, with $ s=0 $.

\begin{lemma}\label{lem:sa}
  Let $ h $ as in
  \refcond{D}, with $ s=0 $.  Set
  \begin{equation*}
    \cB_{0,E,h}'=\Big\{\phi\in \cI_{0}:\max_{|n'|,|n''|< N_0^{1/2}}
    \dist (\spec  H_{[-N_0+n',N_0+n'']}(x_{0}(\phi,E)+h),E)<\exp(-N_0^\beta)\Big\}.
  \end{equation*}
  Then for any
  $ E\in (E_0-r_0,E_0+r_0) $, the set $ \cB_{0,E,h}' $ is contained in
  a semialgebraic set of degree less than $ N_0^{20} $ and with measure less
  than $ \exp(-N_0^{2\delta}) $.
\end{lemma}
\begin{proof}
  Fix $ E\in (E_0-r_0,E_0+r_0) $. By truncating the Taylor series of $
  x_0(\cdot,E) $ we obtain a polynomial $ \tilde x_0(\cdot,E) $ of
  degree less than $ C(d)N_0^4 $ such that
  \begin{equation*}
	\sup_{\phi\in \cI_0}|x_0(\phi,E)-\tilde x_0(\phi,E)|\le \exp(-N_0^2)
  \end{equation*}
  To estimate the remainder of the Taylor series we used 
  condition \refcond{B} and Cauchy estimates (also recall
  \cref{rem:A-E-comments} (a)). Note that for any $ [a,b]\subset
  \Z$, $ \phi\in \cI_0 $,
  \begin{equation*}
	\norm{H_{[a,b]}(x_0(\phi,E))-H_{[a,b]}(\tilde x_0(\phi,E))}\le C_\rho \norm{V}_\infty
    |x_0(\phi,E)-\tilde x_0(\phi,E)|\le \exp(-N_0^2/2).
  \end{equation*}
  Let $ \tilde V, \tilde H $ be as in \cref{eq:V-tilde},
  \cref{eq:H-tilde} (with $ N_0 $ instead of $ N $). We have
  \begin{equation}\label{eq:sa-approximation}
	\norm{H_{[a,b]}(x_0(\phi,E))-\tilde H_{[a,b]}(\tilde x_0(\phi,E))}\le \exp(-N_0^2/4)
  \end{equation}
  for any $ [a,b]\subset \Z $.
  Let
  \begin{equation*}
    \tilde \cB_{0,E,h}=\Big\{\phi\in \cI_{0}:\max_{|n'|,|n''|< N_0^{\frac{1}{2}}}
     \normhs{(\tilde H_{[-N_0+n',N_0+n'']}(\tilde x_{0}(\phi,E)+h)-E)^{-1}}>\exp(-3N_0^\beta/4)\Big\},
  \end{equation*}
  where $ \normhs{\cdot} $ stands for the Hilbert-Schmidt norm. Then $
  \tilde \cB_{0,E,h} $ is semialgebraic of degree less than $
  N_0^{20} $ and using \cref{eq:sa-approximation} we have
  \begin{equation*}
	\cB_{0,E,h}'\subset \tilde \cB_{0,E,h}\subset \cB_{0,E,h},
  \end{equation*}
  thus concluding the proof.
\end{proof}

\begin{lemma}\label{lem:inductive1}
  For any $ E\in(E_0-r_0,E_0+r_0) $
  there exists a semialgebraic
  set $\cB_{0,E,N_1}$,
  \begin{equation*}
	\deg(\cB_{0,E,N_1})\lesssim N_1N_0^{20},\quad \mes (\cB_{0,E,N_1})<\exp (-N_0^{2\delta}/2),
  \end{equation*}
  such that for any $\phi\in \cI_0\setminus\cB_{0,E,N_1}$ and any
  $3N_0/2<|m|\le N_1$, there exist
  $|n'(\phi,m)|, |n''(\phi,m)|< N_0^{1/2}$ such that with
  $J_m=m+[-N_0+n'(\phi,m),N_0+n''(\phi,m)]$
  \begin{equation*}
    \dist (\spec  H_{J_m}(x_{0}(\phi,E)),E)
    \ge \exp(-N_0^\beta).
  \end{equation*}
\end{lemma}
\begin{proof}
  Take arbitrary $3N_0/2<|m|\le N_1$. Then $ 0<|m-n|< 3N_1 $ for any $
  n\in \fT_0 $ (recall \cref{eq:fT_0}) and due to the Diophantine condition we have
  \[
    \dist (m\omega, \mathfrak{T}_{0})> a(3N_1)^{-b}\ge a(CN_0^A)^b>\exp (-N_0^\mu).
  \]
  Hence, for any $3N_0/2<|m|\le N_1$
  condition \refcond{D} applies with $h=m\omega$.  We let
  $ \cB_{0,E,N_1}:=\bigcup_{m}\tilde \cB_{0,E,m\omega} $, where
  $ \tilde \cB_{0,E,m\omega}$ are the semialgebraic sets from the
  statement of \cref{lem:sa}.  Then $ \cB_{0,E,N_1} $ is semialgebraic
  of degree $ \lesssim N_1 N_0^{20} $ and we have
  \begin{gather*}
    \mes (\cB_{0,E,N_1}) \lesssim N_1\exp (-N_0^{2\delta})<
    \exp (-\frac{1}{2}N_0^{2\delta}).
  \end{gather*}
  Take  $\phi\in \cI_0\setminus \cB_{0,E,N_1}$. Since $ \phi\in
  \cI_0\setminus \tilde \cB_{0,E,m\omega} $, the conclusion follows
  from the definition of $ \cB_{0,E,m\omega} $ (recall \cref{eq:basic-identity}).
\end{proof}

The next lemma is not needed at the moment, but it motivates one of
the choices we make in the statement of \cref{prop:inductive4}
\begin{lemma}\label{lem:eigenvalue-analyticity}
  (a) The function $ E^{[-N_0',N_0'']} $ is analytic on $ \{ z\in \C^d
  : |z-x_0(\phi,E)|<\exp(-2N_0^\delta)\} $, for any $ (\phi,E)\in
  \Pi_0 $.

  \smallskip \noindent (b) The function $
  E^{[-N_0',N_0'']}(x_0(\phi,E)) $ is analytic on
  \begin{equation*}
	\cP_0'=\{ (\phi,E)\in \C^d: \dist((\phi,E),\Pi_0)<r_0^4 \}.
  \end{equation*}
\end{lemma}
\begin{proof}
  Statement (a) follows from the separation of eigenvalues in \refcond{A} and basic
  perturbation theory. Statement (b) follows from (a) by noticing that
  \begin{equation*}
    |x_{0}(\phi+\zeta,E+\eta)-x_{0}(\phi,E)|\le C_\rho\exp(N_0^{\delta})(|\zeta|+|\eta|)<\exp(-2N_0^\delta)
  \end{equation*}
  for any $(\zeta,\eta)\in \C^d $ with
  $|\zeta|,|\eta|<\exp(-4N_0^\delta)$ (we used \refcond{B} and Cauchy estimates).
\end{proof}

\begin{prop}\label{prop:inductive4}
  There exists $ \phi_1\in \T^d $, $ |\phi_1-\phi_0|\ll r_0^4 $,  and
  $|N_1'-N_1|,|N_1''-N_1|\lesssim N_0$ such that the following hold.
  
  \medskip
  \noindent(i)
  $ \cI_1'\subset \cI_0 \setminus \cB_{0,E_1,N_1} $,
  $ \cI_1'=\phi_1+(-r_1',r_1')^{d-1}$, $ r_1'=\exp(-3N_0^\beta) $,
  with $ \cB_{0,E_1,N_1} $ as in \cref{lem:inductive1}.

  \medskip
  \noindent(ii)    
  There exists  $ k_1 $ such that for any
  $\phi\in \mathcal{I}'_{1}$, $y\in\mathbb{R}^d$,
  $|y|<r_1'$, $E\in \mathbb{R}$,
  $|E-E_1|<r_1'$,
  \begin{gather}\label{inductive22}
    \big|E_{k_1}^{[-N_1',N_1'']}(x_{0}(\phi,E)+y)-E^{[-N'_{0},N''_0]}(x_{0}(\phi,E)+y)\big|<
    \exp(-\gamma N_0/20),\\
    \big|E_j^{[-N_1',N_1'']}(x_{0}(\phi,E)+y)-E_{k_1}^{[-N_1',N_1'']}(x_{0}(\phi,E)+y)\big|>
    \frac{1}{8}\exp(-N^\beta_0),\quad j\neq {k_1},\label{inductive25}\\    
    |\psi_{k_1}^{[-N_1',N_1'']}(x_{0}(\phi,E)+y,n)|<\exp(-\gamma |n|/10),\quad |n|\ge 3N_0/4,
    \label{inductive24}\\
    \label{inductive23}
    \|\psi_{k_1}^{[-N_1',N_1'']}(x_{0}(\phi,E)+y,\cdot)-\psi^{[-N'_0,N''_0]}(x_{0}(\phi,E)+y,\cdot)\|
    <\exp(-\gamma N_0/20).
  \end{gather} 
\end{prop}
\begin{proof}
  Using the information we have on $ \cB_{0,E_1,N_1} $ and
  \cref{lem:sa-covering}, it follows that there exists $ \phi_1 $,
  $|\phi_1-\phi_0|\ll r_0^4$ (in fact, we could replace $ r_0^4 $ by
  $ r_0^C $, with any fixed $ C\ge 1 $), such that
  $ \cI_1'\subset \cI_0\setminus \cB_{0,E_1,N_1} $ (recall that $
  \beta\gg \delta $).  Take the
  intervals $J_m=m+[-N_0+n'(\phi_1,m),N_0+n''(\phi_1,m)]$ from
  Lemma~\ref{lem:inductive1}. Define
  \begin{equation}\label{eq:N1-interval}
    [-N_1',N_1'']=[-3N_0/2,3N_0/2]\cup\bigcup_{3N_0/2<|m|\le N_1} J_m.	
  \end{equation}
  Due to \cref{lem:inductive1},
  \begin{gather*}
    \dist (\spec  H_{J_m}(x_{0}(\phi_1,E_1)),E_1)
    \ge \exp(-N_0^\beta).
  \end{gather*}
  Using condition \refcond{B} and Cauchy estimates we have that for $ \phi\in\cI_1' $,
  $ |y|<\exp(-3N_0^\beta) $, $ |E-E_1|<\exp(-3N_0^\beta) $,
  \begin{equation*}
	\left| x_0(\phi,E)+y-x_0(\phi_1,E_1) \right|\le \exp(CN_0^{\delta})(|\phi-\phi_1|+|E-E_1|)+|y|
    <\exp(-2N_0^\beta).
  \end{equation*}
  The conclusion follows by invoking
  \cref{prop:stabilization} (recall that $ \beta\ll \sigma $) with $x_0=x_{0}(\phi,E_1)$,
  $ E_0=E_1 $.
\end{proof}

For the rest of this section we adopt the notation of
\cref{prop:inductive4}. To simplify the notation, we suppress $k_1$
from the notation and use $E^{[-N_1',N_1'']},\psi^{[-N_1',N_1'']}$
instead. Next we want to prove the existence of the parametrization $ x_1 $.

\begin{lemma}\label{lem:inductive5}
  (a)  The function $ E^{[-N_1',N_1'']} $ is analytic on
  $ \{ z\in \C^d : |z-x_0(\phi,E)|<\exp(-2N_0^\beta)\} $, for any
  $ (\phi,E)\in \Pi_1' $.

  \smallskip\noindent (b) 
  The function $E^{[-N_1',N_1'']}(x_{0}(\phi,E))$ is analytic on
  \begin{gather*}
    \mathcal{P}'_{1}=\{(\phi,E)\in \mathbb{C}^d:\dist((\phi,E),\Pi_1')<r_1'\},
  \end{gather*}
  with $ \Pi_1'=\cI_1'\times (E_1-r_1',E_1+r_1') $.
  Furthermore, for any $ (\phi,E)\in \frac{1}{50}\mathcal{P}_1'$,
  \begin{gather}
    \big|E^{[-N_1',N_1'']}(x_{0}(\phi,E))-E\big|<
    \exp(-c_0\gamma N_0),\label{inductive29}\\
    \big|\partial_E E^{[-N_1',N_1'']}(x_{0}(\phi,E))-1\big|<
    \exp(-c_0\gamma N_0/2).\label{inductive49}
  \end{gather}
  with $c_0=c_0(d)$.
\end{lemma}
\begin{proof}
  The analyticity statements follow as in \cref{lem:eigenvalue-analyticity}. By \cref{prop:inductive4}, the
  estimate \eqref{inductive29} holds for real $(\phi,E)\in
  \frac{1}{2}\cP_1'\cap \R^d $
  with $ c_0=1/20 $ (recall \eqref{eq:mapx1}).  With the help of
  \cref{cor:high_cart} one concludes that the estimate is 
  also valid for complex $\phi,E$, with some $c_0(d)<1/20 $. 
  The estimate \eqref{inductive49} follows from Cauchy estimates combined with \eqref{inductive29}.
\end{proof}

\begin{prop}\label{prop:inductive6nhh}
  Let
  \begin{equation*}
	\cP_1''=\{ (\phi,E)\in \C^d : |\phi-\phi_1|,|E-E_1|<\exp(-C_0N_0^\beta) \},
  \end{equation*}
  with $ C_0=C_0(d)\gg 1 $.
  There exists a map  $x_1: \Pi_1''\to \R^d$, $
  \Pi_1'':=\cP_1''\cap \R^d $, that extends analytically on
  $\mathcal{P}_{1}''$, such that
  \begin{gather}
    \label{inductive44s}
    E^{[-N'_1,N''_1]}(x_{1}(\phi,E))=E,\quad (\phi,E)\in\mathcal{P}_{1}'',\\
    \label{inductive44t}
    x_1(\cP_1'')\subset \T_{\rho/2}^d.
  \end{gather}
  Furthermore, for any $ (\phi,E)\in \Pi_1'' $,
  \begin{equation}    \label{inductive50}
	|x_{1}(\phi,E)-x_0(\phi,E)|<\exp(-\gamma N_0/30).
  \end{equation}
  and for any $ (\phi,E)\in \cP_1'' $,
  \begin{equation*}
    |x_{1}(\phi,E)-x_0(\phi,E)|<\exp(-c_0\gamma N_0),\quad c_0=c_0(d).
  \end{equation*}
\end{prop}
\begin{proof}
  By \cref{prop:inductive4} one has 
  \begin{equation}\label{eq:inductive42}
    \big|E^{[-N'_1,N''_1]}(x_{0}(\phi,E))-E\big|<
    \exp(-\gamma N_0/20)
  \end{equation}
  for any $\phi\in \mathcal{I}'_{1}$ and any real
  $|E-E_1|<\exp(-3N_0^{\beta})$. Given real
  $|E-E_1|<\frac{1}{2}\exp(-3N_0^{\beta})$, set
  $E_\pm=E\pm 2\exp(-\gamma N_0/20)$.  Since
  $|E_\pm-E_1|< \exp(-3N_0^{\beta})$, using \cref{eq:inductive42} we
  have
  \begin{gather*}
  E^{[-N'_1,N''_1]}(x_{0}(\phi,E_-))<E<E^{[-N'_1,N''_1]}(x_{0}(\phi,E_+)).
  \end{gather*}
  It follows that
 \begin{gather}\label{inductive44}
  E^{[-N'_1,N''_1]}(x_{0}(\phi,\eta))=E
  \end{gather}
  has a solution $\eta\in (E_-,E_+)$. Let $ \eta_1 $ be the solution
  corresponding to $ \phi=\phi_1 $, $ E=E_1 $. Recall that due to \eqref{inductive49} in
  Lemma~\ref{lem:inductive5} one has
  \[
  \partial_\eta E^{[-N'_1,N''_1]}(x_{0}(\phi,\eta))\ge 1/2.
  \]
  Therefore, due to the implicit function theorem for analytic
  functions, see \cref{lem:impl2}, for
  \begin{equation*}
	|\phi-\phi_1|,|E-E_1|<\exp(-2CN_0^\beta),\ C=C(d)>3,
  \end{equation*}
  there exists a unique analytic solution $ \eta(\phi,E) $, $
  |\eta(\phi,E)-\eta_1|<\exp(-CN_0^\beta) $ of \cref{inductive44}.
  Then \eqref{inductive44s} and \eqref{inductive44t} hold by setting $ x_1(\phi,E)=x_0(\phi,\eta(\phi,E)) $.
  By
  uniqueness, for real $ \phi,E $,   $ \eta(\phi,E)\in (E_-,E_+) $,
  and therefore
  \[
    |\eta(\phi,E)-E|<2\exp(-\gamma N_0/20)
  \]
  and \eqref{inductive50} follows. The last estimate is a consequence
  of \cref{cor:high_cart}
  (note that we take $ C_0<2C $).
\end{proof}

\begin{cor}\label{cor:inductive5}
  Using the notation of \cref{prop:inductive6nhh}, for any $
  (\phi,E)\in \Pi_{1}'' $,
  \begin{gather}
    \label{inductive65I}
    \big|E-E_j^{[-N'_1,N''_1]}(x_{1}(\phi,E))\big|>
    {1\over 8}\exp(-N^{\beta}_{0})>\exp(-N_1^\delta),\quad j\ne k_1,\\
    \label{inductive64I}
    |\psi^{[-N'_1,N''_1]}(x_{1}(\phi,E),n)|<\exp(-\gamma |n|/10),\quad |n|\ge 3N_{0}/4,\\
    \label{inductive60E}
    \|\psi^{[-N'_1,N''_1]}(x_{1}(\phi,E),\cdot)-\psi^{[-N'_{0},N''_{0}]}(x_{1}(\phi,E),\cdot)\|<
    \exp(-\gamma N_{0}/20),\\
    \|\psi^{[-N'_1,N''_1]}(x_{1}(\phi,E),\cdot)-\psi^{[-N'_{0},N''_{0}]}(x_{0}(\phi,E),\cdot)\|<
    \exp(-\gamma N_{0}/40).
\end{gather}
\end{cor}
\begin{proof} All statements, except the last one follow from
  \cref{inductive50} and 
  \cref{prop:inductive4} with $ y=x_1(\phi,E)-x_0(\phi,E) $. In the
  first estimate we used $N_1\simeq N_0^{\beta^{-1}}$ and 
  $\beta^2\ll \delta$. The last estimate follows from 
  \begin{multline*}
	\norm{(H_{[-N_0',N_0'']}(x_0(\phi,E))-E^{[-N_1',N_1'']}(x_1(\phi,E)))\psi^{[-N_1',N_1'']}(x_1(\phi,E))}\\
    \le
    \norm{(H_{[-N_0',N_0'']}(x_0(\phi,E))-H_{[-N_0',N_0'']}(x_1(\phi,E)))\psi^{[-N_1',N_1'']}(x_1(\phi,E))}\\
    +
    \norm{(H_{[-N_0',N_0'']}(x_1(\phi,E))-E^{[-N_1',N_1'']}(x_1(\phi,E)))\psi^{[-N_1',N_1'']}(x_1(\phi,E))}\\
    \le C_\rho\norm{V}_\infty |x_0(\phi,E)-x_1(\phi,E)|+2\exp(-\gamma (N_0-N_0^{1/2}/10)
    <\exp(-\gamma N_0/35),
  \end{multline*}
  the separation of eigenvalues, and \cref{lem:eigenvector-stability}.
\end{proof}

Next we check condition \refcond{D} with $ s=1 $.
Let
\begin{equation*}
  \cI_0'= \{ \phi\in\R^{d-1}: |\phi-\phi_0|<r_0^4 \}
\end{equation*}
(recall \cref{lem:eigenvalue-analyticity}).

\begin{lemma}\label{lem:inductive9}
  Let $ h\in \R^d $, $ \exp (-N_1^\mu)\le\|h\|<\exp (-N_0^\mu) $, and
  $ E\in (E_1-r_1,E_1+r_1) $.
  Then for any $\nu>0$,
  \begin{equation*}
    \mes\{\phi\in \cI_0'/10:
    \log |E^{[-N'_0,N''_0]}(x_{0}(\phi,E)+h)-E\big|\le
    -N_1^{\mu+\nu}\}\\
    <\exp(-c(d) (N_0^\delta+N_1^{\nu/(d-1)})).
  \end{equation*}
\end{lemma}
\begin{proof}
  By Taylor's formula, 
  \begin{equation}\label{eq:Taylor-E}
    E^{[-N'_0,N''_0]}(x_{0}(\phi,E)+h)-E^{[-N'_0,N''_0]}(x_{0}(\phi,E))=
    \langle \nabla E^{[-N_0',N_0'']}(x_0(\phi,E),h \rangle+O(\exp(CN_0^\delta) \norm{h}^2).
  \end{equation}
  We used the fact that by  Cauchy
  estimates (recall \cref{lem:eigenvalue-analyticity}),
  \[ \left|\frac{d^2}{dh^2} E^{[-N'_0,N''_0]}(x_{0}(\phi,E)+h)\right|\le \exp(CN_0^\delta).
  \]
  Due to condition \refcond{E} we
  can find $|\hat \phi_0-\phi_0|\ll r_0^4$ such that
  \begin{equation*}
    |\langle\nabla E^{[-N'_0,N''_0]}(x_{0}(\hat \phi_0,E_1)),h_0\rangle|\ge \exp(-N_0^{\mu}/2),
    \qquad h_0:=\frac{h}{\norm{h}}.
  \end{equation*}
  Since
  \begin{multline*}
	|\nabla E^{[-N_0',N_0'']}(x_0(\hat \phi_0,E_1))-\nabla E^{[-N_0',N_0'']}(x_0(\hat \phi_0,E))|\\
    \le \exp(CN_0^\delta) |x_0(\hat \phi_0,E_1)-x_0(\hat \phi_0,E)|
    \le \exp(C'N_0^\delta)|E-E_1|,
  \end{multline*}
  we have
  \begin{equation*}
    |\langle\nabla E^{[-N'_0,N''_0]}(x_{0}(\hat \phi_0,E)),h_0\rangle|\gtrsim \exp(-N_0^{\mu}/2),
  \end{equation*}
  for any $ E\in (E_1-r_1,E_1+r_1) $ (note that $ N_1^\delta\gg N_0^\mu
  $; recall that $ \delta\gg \beta^2\gg \beta\mu $).

  Plugging the above in \cref{eq:Taylor-E},
  \begin{gather*}
    |E^{[-N'_0,N''_0]}(x_{0}(\hat\phi_0,E)+h)-E\big|\gtrsim \|h\|\exp(-N_0^{\mu}/2)\ge
    \exp (-2N_1^\mu)
  \end{gather*}
  (we used $ \exp(CN_0^\delta)\norm{h}\le \exp(CN_0^\delta-N_0^\mu)\ll \exp(-N_0^\mu/2) $).
  The conclusion follows by applying Cartan's estimate 
  to $E^{[-N'_0,N''_0]}(x_{0}(\phi,E)+h)-E$ on the polydisk
  $|\phi-\hat \phi_0|< r_0^4$, with $ H=c\exp(N_1^\nu) $, $ c\ll 1 $.
\end{proof}

\begin{prop}\label{prop:inductive8}
  Let $ h\in \R^d $ such that
  $\dist (h, \mathfrak{T}_{1})\ge \exp (-N_1^\mu)$ (recall
  \cref{eq:fT_0}) and
  \begin{equation*}
    \cB_{1,E,h}''=\Big\{\phi\in \cI_{1}'':\max_{|n'|,|n''|< N_1^{1/2}}
    \dist (\spec  H_{[-N_1+n',N_1+n'']}(x_{1}(\phi,E)+h),E)<\exp(-N_1^\beta/2)\Big\}.
  \end{equation*}
  Then for any $ E\in (E_1-r_1,E_1+r_1) $,
  $\mes (\cB''_{1,E,h}) <\exp(-N_1^{2\delta})$.
\end{prop}
\begin{proof}
  Let $|m_1|\le 3N_1/2$, $h_1\in \mathbb{R}^d$ such that
  \begin{equation*}
    \dist (h, \mathfrak{T}_{1})=\|h_1\|,\qquad
    h_1=h-m_1\omega \,(\mod \mathbb{Z}^d).
  \end{equation*}
  Note that for any $ m_1\in [-N_1,N_1] $ we have 
  \begin{equation}\label{eq:dist-m-T0}
    \dist (h+m\omega, \mathfrak{T}_{0})=\dist (h,-m\omega+\mathfrak{T}_{0})
    \ge \dist (h, \mathfrak{T}_{1})=\norm{h_1},
  \end{equation}
  since $-m+[-3N_0/2,3N_0/2]\subset [-3N_1/2,3N_1/2]$.
  At the same time, if $|m+m_1|>3N_0/2$, using the Diophantine
  condition we get
  \begin{equation}\label{eq:dist-m-T0-alt}
    \|h+m\omega-n\omega\|=\|h_1+(m+m_1-n)\omega\|
    \ge\|(m+m_1-n)\omega\|-\|h_1\|\ge a(CN_1)^{-b}-\norm{h_1},
  \end{equation}
  for any $ n\in\fT_0 $.

  We consider two cases: $ \norm{h_1}\ge \exp(-N_0^\mu) $ and $
  \norm{h_1}<\exp(-N_0^\mu) $. In either case, by the above, we have
  $ \dist(h+m\omega,\fT_0)\ge \exp(-N_0^\mu) $ for all $ m\in
  [-N_1,N_1] $ with $ |m+m_1|>3N_0/2 $.
  So, for such $ m $, 
  condition \refcond{D} implies that for each $\phi\in \mathcal{I}_{0}
  \setminus \cB_{0,E_1,h+m\omega}$ there exists $|n'(\phi,m,h)|,
  |n''(\phi,m,h)|< N_0^{1/2}$
  such that with $J_m(\phi)=m+[-N_0+n'(\phi,m,h),N_0+n''(\phi,m,h)]$,
  \begin{equation*}
    \dist (\spec  H_{J_m(\phi)}(x_{0}(\phi,E_1)+h),E_1)
    \ge \exp(-N_0^\beta/2)
  \end{equation*}
  and therefore
  \begin{equation*}
	\dist (\spec  H_{J_m(\phi)}(x_{0}(\phi,E)+h),E)
    \ge \exp(-N_0^\beta/4)\ge \exp(-N_0^{\sigma/2})
  \end{equation*}
  for any $ E\in(E_1-r_1,E_1+r_1) $ (note that $ N_1^\delta\gg N_0^\beta
  $; recall that $ \delta\gg \beta^2 $).
  In particular, since
  $\mes (\cB_{0,E_1,h+m\omega})<\exp(-N_0^{2\delta})$, there exists $ \phi_{0,m}\in \mathcal{I}_{0}
  \setminus \cB_{0,E_1,h+m\omega} $, $|\phi_{0,m}-\phi_0|\ll r_0^4$.
  Let $ J_m:=J_m(\phi_{0,m}) $.  Due to the
  spectral form of (LDT), 
  \[
    \log |f_{J_m}(x_{0}(\phi_{0,m},E)+h),E)|>|J_m|L_{|J_m|}(E)-|J_m|^{1-\tau/2}.
  \]
  Using the uniform upper estimate (see \cref{cor:logupper})
  we can apply Cartan's estimate  to get
  \begin{equation}\label{eq:Bm'}
    \mes\{\phi\in \cI_0'/10:
    \log |f_{J_m}(x_{0}(\phi,E)+h,E)|<|J_m|L(E)-|J_m|^{1-\tau/4}\}
    <\exp(-N^{\tau/8(d-1)}_0)
  \end{equation}
  (in fact, the estimate holds for $ \phi\in \cI_0/10 $).
  Denote by $\cB_{0,E,m}'$ the set in the above estimate and let
  \begin{equation*}
    \cB'_{0,E,N_1}= \bigcup_{-N_1\le m\le N_1, |m+m_1|>3N_0/2}\cB'_{0,E,m}.
  \end{equation*}
  Since $ \delta\ll \beta\sigma\ll \beta\tau $, we have
  \begin{gather*}
    \mes (\cB'_{0,E,N_1})
    \lesssim N_1\exp(-N^{\tau/8(d-1)}_0)<
    \exp (-N_0^{\tau/8(d-1)}/2)\ll \exp(-N_1^{2\delta}).
  \end{gather*}

  We now have to deal with $ |m+m_1|\le 3N_0/2 $. It will be enough to
  focus on $ m=-m_1 $. We assume $ m_1\in [-N_1,N_1] $ so that
  \cref{eq:dist-m-T0} holds. If $ \norm{h_1}\ge \exp(-N_0^\mu) $,
  then by \cref{eq:dist-m-T0}, $ \dist(h+m_1\omega,\fT_0)\ge
  \exp(-N_0^\mu) $ and by the above reasoning there exists an interval
  $ J_{-m_1} $ such that \cref{eq:Bm'} holds with
  $ m=-m_1 $. In this case we let $ \cB_{0,E,-m_1}' $ be the set from
  \cref{eq:Bm'}. Suppose that $ \norm{h_1}<\exp(-N_0^\mu) $.
  Let $J_{-m_1}:=-m_1+[-N_0',N''_0]$. We have
  \begin{gather*}
    \spec H_{J_{-m_1}}(x+h)
    =\spec H_{[-N'_0,N''_{0}]}(x+h_1).
  \end{gather*}
  Let
  \begin{equation*}
	\cB_{0,E,-m_1}':=\{\phi\in \cI_0'/10:
    \big|E^{[-N'_0,N''_0]}(x_{0}(\phi,E)+h_1)-E\big|\le
    \exp(-N_1^{\mu+\nu})\},
  \end{equation*}
  with $ \nu=3(d-1)\delta $.
  By Lemma~\ref{lem:inductive9},
  \begin{equation*}
    \mes( \cB_{0,E,-m_1}')<\exp(-c(N_0^\delta+N_1^{\nu/(d-1)}))\ll \exp(-N_1^{2\delta}).
  \end{equation*}
  Since
  \begin{gather*}
    E^{[-N'_0,N''_0]}(x_{0}(\phi,E))=E,\\
    \big|E_j^{[-N'_0,N''_0]}(x_{0}(\phi,E)+h_1)-E_j^{[-N'_0,N''_0]}(x_{0}(\phi,E))\big|
    \le C_\rho \|V\|_\infty \|h_1\| \ll \exp(-N_0^\delta),
  \end{gather*}
  the separation of eigenvalues in condition \refcond{A} implies
  \begin{equation*}
    \dist (\spec H_{J_{-m_1}}(x_0(\phi,E)+h),E)> \exp(-N_1^{\mu+\nu}),
  \end{equation*}
  for any $\phi\in \frac{1}{10}\cI_0'\setminus \cB_{0,E,-m_1}'$.  Note
  that
  $|J_{-m_1}|^{\sigma/2}\simeq N_0^{\sigma/2}\gg
  N_1^{\mu+\nu}\gg N_0^\delta $ since 
  $\delta,\mu\ll \beta\sigma$.  Therefore we can apply the spectral form of
  (LDT) to get
  \begin{equation*}
    \log|f_{J_{-m_1}}(x_0(\phi,E)+h)|> |J_{-m_1}|L(E)-|J_{-m_1}|^{1-\tau/2},
  \end{equation*}
  for $ \phi \in \frac{1}{10}\cI_0'\setminus \cB_{0,E,-m_1}' $. So, in
  either case we identified an interval $ J_{-m_1} $ and got a similar conclusion.

  Let
  \begin{equation}\label{eq:I}
	I:=\begin{cases}
      J_{-m_1}\cup (\bigcup_{-N_1\le m\le N_1,|m+m_1|>3N_0/2} J_m) &, m_1\in [-N_1,N_1]\\
      \bigcup_{-N_1+2N_0\le m\le N_1-2N_0} J_m &, m_1\notin [-N_1,N_1].
    \end{cases}
  \end{equation}
  Note that $ J_{-m_1} $ overlaps  with the union of the other
  intervals and $ |m+m_1|>3N_0/2 $ for all $ m $'s in the last union.
  By the above, we can use the covering form of (LDT) from
  \cref{lem:Greencoverap0} to get that
  \begin{equation*}
    \dist( \spec H_{I}(x_{0}(\phi,E))+h),E)   \ge\exp( -2 \max_m |J_m|^{1-\tau/4})\\
    >\exp(-4N_0^{1-\tau/4}),
  \end{equation*}
  for any
  $ \phi\in \frac{1}{10}\cI_0'\setminus (\cB_{0,E,N_1}'\cup
  \cB_{0,E,-m_1}') $. Due to
  \eqref{inductive50},
  \begin{equation}\label{eq:I-estimate}
    \dist( \spec H_{I}(x_{1}(\phi,E))+h),E)\gtrsim\exp(-4N^{1-\tau/4}_0)\gg \exp(-N_1^\beta/2),
  \end{equation}
  for $ \phi\in \cI_1''\setminus (\cB_{0,E,N_1}'\cup
  \cB_{0,E,-m_1}') $. Therefore $ \cB_{1,E,h}''\subset \cB_{0,E,N_1}'\cup
  \cB_{0,E,-m_1}' $ and the conclusion holds.
\end{proof}

\begin{remark}\label{rem:non-symmetric}
  (a) Taking the maximum in the definition of the set $ \cB_{0,E,h} $ from
  condition \refcond{D} is a convenient way of capturing the fact that
  while we do not know precisely the interval $ I $ for which
  \cref{eq:I-estimate} holds, we do know that it is ``close'' to $
  [-N_1,N_1] $.

  \smallskip \noindent (b) If in the definition of $ \cB_{0,E,h} $ we
  would use symmetric intervals, then we could also choose $ I $ to be
  symmetric. However, even so, $ [-N_1',N_1''] $ need not be symmetric
  because we don't have enough control over the sizes of the intervals
  $ J_m $ in \cref{eq:N1-interval} (for example we cannot say that $
  J_m $ and $ J_{-m} $ have the same size).

  \smallskip\noindent (c) The reason for wanting $ A\ge \beta^{-1} $,
  as noted in \cref{rem:choice-of-A}, is the estimate \cref{eq:I-estimate}.
\end{remark}

Now we just need to check condition \refcond{E} with
$ s=1 $.

\begin{lemma}\label{lem:inductive10}
  Let $h_0\in \R^d$ be a unit vector. Then
  \begin{equation*}
	\left| \nabla E^{[-N_1',N_1'']}(x_1(\phi,E))-\nabla E^{[-N_0',N_0'']}(x_0(\phi,E)) \right|
    < \exp(-c_0\gamma N_0),\quad c_0=c_0(d),
  \end{equation*}
  for any $ (\phi,E)\in \Pi_1'' $.
\end{lemma}
\begin{proof}
  Using \cref{inductive50}, we have
  \begin{multline*}
	\left| \nabla E^{[-N_0',N_0'']}(x_1(\phi,E))-\nabla E^{[-N_0',N_0'']}(x_0(\phi,E)) \right|\\
    \le \exp(CN_0^\delta) |x_1(\phi,E)-x_0(\phi,E)|<\exp(-\gamma N_0/35).
  \end{multline*}
  On the other hand, using \cref{inductive22}, \cref{inductive50},
  \cref{cor:high_cart}, and Cauchy estimates,
  we have
  \begin{equation*}
    \left| \nabla E^{[-N_1',N_1'']}(x_1(\phi,E))-\nabla E^{[-N_0',N_0'']}(x_1(\phi,E)) \right|
    \le \exp(-c(d)\gamma N_0),
  \end{equation*}
  and the conclusion follows.
\end{proof}

\begin{prop}\label{prop:inductiveDcond}
  Let $h_0\in \R^d $ be a unit vector.
  Then for any $ E\in (E_1-r_1,E_1+r_1) $, 
  \[
    \mes \{\phi\in \mathcal{I}_{1}'':\log
    |\langle \nabla E^{[-N_1',N_1'']}(x_1(\phi,E)),h_0 \rangle|
    <-N_1^{\mu}/2\}<\exp (-N_1^{2\delta}).
  \]
\end{prop}
\begin{proof}  
  Due to condition \refcond{E} 
  we can find $ \hat \phi_0 $, $|\hat \phi_0-\phi_0|\ll r_0^4$,  
  such that
  \begin{equation*}
    \log | \langle \nabla E^{[-N_0',N_0'']}(x_0(\hat \phi_0,E)),h_0 \rangle |\ge -N_0^\mu/2.
  \end{equation*}
  Applying Cartan's estimate  we get
  \begin{multline*}
    \mes\{\phi\in \cI_0'/10:
    \log | \langle \nabla E^{[-N_0',N_0'']}(x_0(\phi,E)),h_0 \rangle | < -N_0^{\mu+\nu}\}\\
    <\exp(-c(d)(N_0^\delta+N_0^{\nu/(d-1)}))< \exp(-N_1^{2\delta}),
  \end{multline*}
  where $ \nu=3(d-1)\beta^{-1}\delta $. Let $ \cB $ be the set on the
  left-hand side.  Note that $ \cI_1''\subset \cI_0'/10 $, since $ |\phi_1-\phi_0|\ll r_0^4 $.
  Since $ \mu+\nu\ll 1 $, \cref{lem:inductive10} implies
  \begin{equation*}
    \log | \langle \nabla E^{[-N_1',N_1'']}(x_1(\phi,E)),h_0 \rangle |\ge -2N_0^{\mu+\nu}
    \ge -N_1^{\mu}/2,
  \end{equation*}
  for any $ \phi\in \cI_1''\setminus \cB $ (recall that $ \delta\ll \mu $). This concludes the proof.
\end{proof}

We  briefly summarize how \cref{thm:D} follows from the previous statements.
\begin{proof}[Proof of \cref{thm:D}]
  The existence of $ \phi_1 $ was obtained in \cref{prop:inductive4}.
  Note that since $ \delta\gg \beta^2 $, we have
  $ \cP_1\Subset \cP_1'' $.  Conditions \refcond{A}-\refcond{C}, and
  the estimates \cref{inductive50E}, \cref{inductive51EE}, follow from
  \cref{prop:inductive6nhh} and \cref{cor:inductive5}. Condition
  \refcond{D} follows from \cref{prop:inductive8}. Condition
  \refcond{E} follows from \cref{prop:inductiveDcond}.
\end{proof}

\section{Inductive Scheme for the Edges of the Spectrum}\label{sec:edges}

As in the previous section we assume the  non-perturbative setting from \cref{sec:basic-tools}.
We introduce another set of conditions that will address the edges of
the spectrum. 

We assume the exponents $ \delta,\mu,\beta $ from the previous section
and we introduce a new exponent $ \fd $ such that $ \fd\ll \delta $.
Let  $ \gamma>0 $. Given an integer $ s\ge 0 $, let
\begin{equation*}
  \ux_s\in \T^d,\quad N_s\in \N,\quad r_s:=\exp(-N_s^{\fd}),
  \quad \uPi_s=\{ x\in \R^d: |x-\ux_s|<r_s \}. 
\end{equation*}
The inductive conditions for the lower edge are as follows.

\medskip\noindent
\condu{A} There exist integers $|N'_s-N_s|,|N''_s-N_s|<N_s^{1/2}$,
and $ \uE^{[-N'_s,N''_s]}=E_{k_s}^{[-N_s',N_s'']} $, such that
\begin{equation}\label{eq:edge-separation}
  E_{j}^{[-N_s',N_s'']}(x)-\uE^{[-N_s',N_s'']}(x)\ge \exp(-N_s^\fd),
\end{equation}
for any $ x\in \uPi_s $ and $ j\neq k_s $.

\medskip\noindent
\condu{B} For any $ x\in \uPi_s $,
\begin{equation*}
  |\upsi^{[-N_s',N_s'']}(x,n)|\le\exp(-\gamma|n|/10),\quad |n|\ge N_s/4.
\end{equation*}

\medskip\noindent \condu{C} The point $\underline{x}_s$, is a
non-degenerate minimum of the function
$\underline{E}^{[-N_{s}',N_{s}'']}$.  Specifically, with $ \nu_s=\exp(-N_s^\fd) $,
\begin{equation*}
  \nabla \underline{E}^{[-N_{s}',N_{s}'']}(\underline{x}_s)=0,\qquad
  \mathfrak{H}(\underline{E}^{[-N_{s}',N_{s}'']})(\underline{x}_s)\ge \nu_s I.
\end{equation*}

\medskip\noindent \condu{D} Let $ \uE_s=\uE^{[-N_s',N_s'']}(\ux_s)
$. Let $ \fT_s $ be as in \cref{eq:fT_0}. Take arbitrary
$h\in \mathbb{R}^d$ with
\[
\dist (h, \mathfrak{T}_{s})\ge \exp(-N_s^{2\fd}).
\]
There exist $|n'(h)|,|n''(h)|< N_s^{1/2}$ such that
\begin{equation*}
  \dist(\spec H_{[-N_s+n'(h),N_s+n''(h)]}(\ux_s+h),(-\infty,\uE_s])\ge \exp(-N_s^{4\fd}).
\end{equation*}

\medskip\noindent The conditions \condo{A}, \condo{B}, \condo{C}, \condo{D}, for the
upper edge are defined analogously, with obvious adjustments in notation.

\begin{thme}\label{thm:E}
  Assume the notation of the inductive conditions. Let
  $ \ux_0\in \T^d $, $ N_0\ge 1 $, assume that the conditions
  \refcondu{A}-\refcondu{D} hold with $ s = 0 $, and
  $ L(E)>\gamma>0 $ for $ E\in(\uE_0-2r_0,\uE_0+2r_0) $. Let
  $ N_s= N_{s-1}^5 $, $ s\ge 1 $. If
  $ N_0\ge (B_0+ S_V+ \gamma^{-1})^C $, $ C=C(a,b,\rho) $, then
  for any $s\ge 1 $ there exists $ \ux_s\in \T^d $
  such that the conditions
  \refcondu{A}-\refcondu{D} hold and we have
  \begin{equation}\label{eq:thmE-estimates}
    \begin{gathered}
      |\uE^{[-N_s',N_s'']}(x)-\uE^{[-N_{s-1}',N_{s-1}'']}(x)|<\exp(-\gamma N_{s-1}/20),\quad x\in \uPi_s,\\
      \norm{\upsi^{[-N_s',N_s'']}(x)-\upsi^{[-N_{s-1}',N_{s-1}'']}(x)}<\exp(-\gamma N_{s-1}/20),
      \quad x\in \uPi_s,\\
      |\ux_s-\ux_{s-1}|<\exp(-\gamma N_{s-1}/50),\quad 
      |\uE_s-\uE_{s-1}|<\exp(-\gamma N_{s-1}/60).
    \end{gathered}
  \end{equation}
  Furthermore, for any $ E_s\in \R $, $
  \exp(-N_s^{100\fd})\le |E_s-\uE_{s}|\le \exp(-N_{s}^{2\fd}) $, conditions \refcond{A}-\refcond{E}
  hold for $ \uE^{[-N_s',N_s'']} $. The analogous statements
  based on conditions \refcondo{A}-\refcondo{D} also hold.
\end{thme}

As for \cref{thm:D}, we only check \cref{thm:E} for $ s=1 $, the
general case following by simply replacing the indices. Furthermore,
we only consider the statement with the conditions for the lower edge,
the other case being completely analogous. Throughout the
section we tacitly assume that $ N_0\ge (B_0+ S_V+
\gamma^{-1})^C $, with $ C=C(a,b,\rho) $ large enough. As in the
previous section, the dependence on the exponents $
\fd,\delta,\beta,\mu $  is left implicit.  We split the
proof of the first part of \cref{thm:E} into several auxiliary
statements. In what follows we fix $ \ux_0 $, $ N_0 $, such that the
assumptions of \cref{thm:E} hold, and $ N_1=N_0^5 $.

\begin{prop}\label{prop:inductive4D}
  There exist integers
  $|N'_1-N_1|,|N''_1-N_1|\lesssim N_0$, $ k_1 $, such that the following
  hold with $ \uE^{[-N_1',N_1'']}=E_{k_1}^{[-N_1',N_1'']} $ and for any
  $ |x-\ux_0|<\exp(-2N_0^{4\fd}) $:
  \begin{gather}\label{eq:newE1D}
    \big|\uE^{[-N'_1,N''_1]}(x)-\underline{E}^{[-N'_{0},N''_0]}(x)\big|
    <\exp(-\gamma N_0/20),\\
    \label{eq:uE1-separation}
    E_j^{[-N_1',N_1'']}(x)-\uE^{[-N_1',N_1'']}(x)>\frac{1}{8}\exp(-N_0^{4\fd}),\quad j\neq k_1,\\
    |\upsi^{[-N'_1,N''_1]}(x,n)|<\exp(-\gamma |n|/10),\quad |n|>3N_0/4,\\
    \|\upsi^{[-N'_1,N''_1]}(x,\cdot)-\upsi^{[-N'_0,N''_0]}(x,\cdot)\|<\exp(-\gamma N_0/20).
  \end{gather}
\end{prop}
\begin{proof}
  Take arbitrary $3N_0/2<|m|\le N_1$. Using the Diophantine condition
  we have
  \begin{equation*}
	\dist(m\omega,\fT_0)\ge a (CN_0)^{-b}\ge\exp(-N_0^{2\fd}).
  \end{equation*}
  Then by condition \refcondu{D} with
  $h=m\omega$, there exist $|n'(m)|,|n''(m)|< N_0^{1/2}$ such that
  with $ J_m=m+[-N_0+n'(m),N_0+n''(m)] $,
  \begin{equation*}
    \dist(\spec H_{J_m}(\ux_0),(-\infty,\uE_0])\ge
    \exp(-N_0^{4\mathfrak{d}})
  \end{equation*}
  (recall \cref{eq:basic-identity}).
  Define
  \[
    [-N'_1,N''_1]=[-3N_0/2,3N_0/2]\cup \bigcup_{3N_0/2<|m|\le N_1} J_m.
  \]
  Using \cref{eq:edge-separation} and \refcondu{B} we can apply
  \cref{prop:stabilization}  (with $x_0=\underline{x}_{0}$, $ E_0=\uE_0
  $, $ \beta=4\fd $) and all the estimates follow.
\end{proof}

For the rest of this section $ \uE^{[-N_1',N_1'']} $ will stand for
the eigenvalue from the previous proposition.
Let
\begin{gather*}
  \cP_0'=\{ z\in \C^d: |z-\ux_0|<r_0' \},\quad r_0'=\exp(-2N_0^{\fd}),\\
  \cP_1'=\{ z\in \C^d: |z-\ux_0|<r_1' \},\quad r_1'=\exp(-3N_0^{4\fd}).
\end{gather*}

\begin{lemma}\label{lem:uE-analyticity}
  The functions $\underline{E}^{[-N'_{0},N''_{0}]}$,
  $ \uE^{[-N_1',N_1'']} $ are analytic on $\cP_0'$, $ \cP_1' $, respectively, and
  \begin{equation*}
    \max_{|\alpha|=k}\sup_{\cP_0'}|\partial^\alpha \underline{E}^{[-N'_{0},N''_{0}]}|
    \le \exp(C(k)N_0^{\fd}),\quad 
    \max_{|\alpha|=k}\sup_{\cP_1'}|\partial^\alpha \underline{E}^{[-N'_{1},N''_{1}]}|
    \le \exp(C(k)N_0^{4\mathfrak{d}}).
  \end{equation*}
  Furthermore,
  \begin{equation*}
    \sup_{\cP_1'}\norm{\nabla \uE^{[-N_1',N_1'']}-\nabla \uE^{[-N_0',N_0'']}},
    \ \sup_{\cP_1'}\norm{\fH(\uE^{[-N_1',N_1'']})-\fH(\uE^{[-N_0',N_0'']})}<\exp(-c_0\gamma N_0),
  \end{equation*}
  with $ c_0=c_0(d) $.
\end{lemma}
\begin{proof}
  The analyticity of the functions follows from the separation of
  eigenvalues (see \cref{eq:edge-separation} and \cref{eq:uE1-separation})
  combined with basic  perturbation theory. The derivative estimates are
  just Cauchy estimates.
  They hold
  on $ \cP_i' $ because the functions are in fact analytic on $
  100\cP_i' $, $ i=0,1 $. 

  Using \cref{eq:newE1D} and \cref{cor:high_cart} we have
  \begin{equation*}
	\sup_{2\cP_1'}|\uE^{[-N_1',N_1'']}-\uE^{[-N_0',N_0'']}|
    <\exp(-c\gamma N_0),\quad c=c(d),
  \end{equation*}
  and the last estimates holds by Cauchy estimates (we chose $
  r_1'=\exp(-3N_0^{4\fd}) $ instead of $ \exp(-2N_0^{4\fd}) $ to
  ensure we have the above estimate).
\end{proof}

\begin{prop}\label{prop:inductive5Dc}
  There exists $ \ux_1 $, $|\underline{x}_1-\underline{x}_0|< \exp(-\gamma N_0/50)$,
  such that
  \begin{gather*}
    \underline{E}^{[-N'_{1},N''_{1}]}(\underline{x}_1)\le
    \underline{E}^{[-N'_{1},N''_{1}]}(x),\quad
    \text{ for any }|x-\underline{x}_0|< r_1',\\
    \nabla \underline{E}^{[-N'_{1},N''_{1}]}(\underline{x}_1)=0,\quad
    \mathfrak{H}(\underline{E}^{[-N'_{1},N''_{1}]})(\underline{x}_1)\ge {\nu_0\over 4} I.
  \end{gather*}
\end{prop}
\begin{proof}
  By Taylor's formula (recall \cref{lem:Morse7} and \refcondu{C})
  \begin{equation*}
    \underline{E}^{[-N'_{0},N''_{0}]}(x)-\underline{E}_0\ge \frac{\nu_0}{2}|x-\underline{x}_0|^2,\quad
    \text{for }
    |x-\ux_0|<r_1'.
  \end{equation*}
  In particular,
  \begin{gather*}
    \underline{E}^{[-N'_{0},N''_{0}]}(x)\ge \underline{E}_0+ 3\exp(-\gamma N_0/20),\quad
    \text{for }\exp(-\gamma N_0/50)\le |x-\underline{x}_0|< r_1'.
  \end{gather*}
  Combining this with \eqref{eq:newE1D}
  we  get
  \begin{equation*}
    \underline{E}^{[-N'_{1},N''_{1}]}(x)\ge \underline{E}^{[-N'_{1},N''_{1}]}(\underline{x}_0)
    +\exp(-\gamma N_0/20),\quad
    \text{for }\exp(-\gamma N_0/50)\le  |x-\underline{x}_0|< r_1'.
  \end{equation*}
  This implies the existence of a point $ \ux_1 $, $ |\ux_1-\ux_0| $
  where $ \uE^{[-N_1',N_1'']} $ attains its minimum on
  $ |x-\ux_0|<r_1' $. The estimate on the Hessian follows from
  \cref{lem:uE-analyticity} and the fact that by Taylor's formula
  (again, recall \cref{lem:Morse7} and \refcondu{C}), we have
  $ \fH(\uE^{[-N_0',N_0'']})(\ux_1)\ge (\nu_0/2)I $.
\end{proof}

We fix an $ \ux_1 $ as in \cref{prop:inductive5Dc} (in fact, in can be
argued that such $ \ux_1 $ is unique).

\begin{lemma}\label{lem:uE0-uE1}
  We have $ |\uE_1-\uE_0|<\exp(-\gamma N_0/60) $.
\end{lemma}
\begin{proof}
  By the mean value theorem, \cref{lem:uE-analyticity}, and \cref{prop:inductive5Dc},
  \begin{equation*}
	|\uE^{[-N_0',N_0'']}(\ux_1)-\uE^{[-N_0',N_0'']}(\ux_0)|\le \exp(CN_0^\fd)|\ux_1-\ux_0|
    <\exp(-\gamma N_0/55).
  \end{equation*}
  Now the conclusion follows using \cref{eq:newE1D}.
\end{proof}

\begin{prop}\label{prop:inducopositD}
  The condition \refcondu{D} holds with $ s=1 $.
\end{prop}
\begin{proof}
  The proof is similar to that of \cref{prop:inductive8}. Let
  $|m_1|\le 3N_1/2$, $h_1\in \mathbb{R}^d$ such that
  \begin{equation*}
    \dist (h, \mathfrak{T}_{1})=\|h_1\|,\qquad
    h_1=h-m_1\omega \,(\mod \mathbb{Z}^d).
  \end{equation*}
  As in the proof of \cref{prop:inductive8} (recall
  \cref{eq:dist-m-T0},\cref{eq:dist-m-T0-alt}), we have
  \begin{equation}\label{eq:dist-h-T0}
    \begin{gathered}
      \dist(h+m\omega,\fT_0)\ge \norm{h_1},\quad\text{provided }|m|\le N_1,\\
      \dist(h+m\omega,\fT_0)\ge a(CN_1)^{-b}-\norm{h_1},\quad \text{provided }|m+m_1|>3N_0/2.
    \end{gathered}
  \end{equation}
  
  We consider two cases: $ \norm{h_1}\ge \exp(-N_0^{2\fd}) $ and $
  \exp(N_1^{2\fd})\le \norm{h_1}<\exp(-N_0^{\fd}) $. In either case, by the above, we have
  $ \dist(h+m\omega,\fT_0)\ge \exp(-N_0^{2\fd}) $ for all $ m\in
  [-N_1,N_1] $ with $ |m+m_1|>3N_0/2 $.
  So, for such $ m $, 
  condition \refcondu{D} (with $ h=m\omega $) implies that there exists an interval
  $ J_m=m+[-N_0+n'(h),N_0+n''(h)] $ such that
  \begin{equation}\label{eq:covering-Jm}
    \dist (\spec  H_{J_m}(\ux_0+h),(-\infty,\uE_0])
    \ge \exp(-N_0^{4\fd}).
  \end{equation}

  Our goal is to apply \cref{lem:Greencoverap1} (with $
  S=(-\infty,\uE_0] $). To this end we will 
  deal with $ |m+m_1|\le 3N_0/2 $ by 
  focusing on $ m=-m_1 $. We assume $ m_1\in [-N_1,N_1] $. If $ \norm{h_1}\ge \exp(-N_0^{2\fd}) $,
  then $ \dist(h+m_1\omega,\fT_0)\ge
  \exp(-N_0^{2\fd}) $ and by condition \refcondu{D} there exists an interval
  $ J_{-m_1} $ such that \cref{eq:covering-Jm} holds with
  $ m=-m_1 $. Suppose that $ \exp(-N_1^{2\fd})\le \norm{h_1}<\exp(-N_0^{2\fd}) $.
  Let $J_{-m_1}:=-m_1+[-N_0',N''_0]$. We have
  \begin{equation}\label{eq:J-m1}
    \spec H_{J_{-m_1}}(\ux_0+h)
    =\spec H_{[-N'_0,N''_{0}]}(\ux_0+h_1).
  \end{equation}
  By Taylor's formula (recall \cref{lem:Morse7} and \refcondu{C}),
  \begin{equation*}
	\uE^{[-N_0',N_0'']}(\ux_0+h_1)\ge \uE_0+\frac{\nu_0}{2}\norm{h_1}^2
    \ge \uE_0+\exp(-3N_1^{2\fd}).
  \end{equation*}
  Using \refcondu{A} it follows that
  \begin{equation*}
    \dist (\spec  H_{J_{-m_1}}(\ux_0+h),(-\infty,\uE_0])
    \ge \exp(-3N_1^{2\fd})>\exp(-N_0^{11\fd}).	
  \end{equation*}

  We now have what we need to invoke the covering form of (LDT). Let $
  I $ as in \cref{eq:I}. By the above, we can use \cref{lem:Greencoverap1} (with
  $ K=N_0^{11\fd} $; recall that $ \fd\ll \delta\ll \sigma$) to get that
  \begin{equation*}
    \dist( \spec H_{I}(\ux_{0}+h),(-\infty,\uE_0])   \ge\exp( -N_0^{12\fd})\gg \exp(-N_1^{4\fd}).
  \end{equation*}
  Using \cref{prop:inductive5Dc} and \cref{lem:uE0-uE1} we have
  \begin{equation*}
    \dist( \spec H_{I}(\ux_{1}+h),(-\infty,\uE_1])\ge\exp(-N_1^{4\fd})
  \end{equation*}
  and the conclusion follows.
\end{proof}

We now proceed to the proof of \cref{thm:E}.

\begin{proof}[{Proof of \cref{thm:E}}]
  The existence of $ \ux_1 $ and $ \uE^{[-N_1',N_1'']} $ is given by
  \cref{prop:inductive4D} and \cref{prop:inductive5Dc}. Note that due
  to \cref{prop:inductive5Dc}, $ \uPi_1\subset \{ |x-\ux_0|<r_1' \}  $
  (recall that $ r_1'=\exp(-3N_0^{4\fd}) $, $ N_1=N_0^{5} $). Now, for $ s=1 $,
  conditions \refcondu{A} and \refcondu{B} hold by
  \cref{prop:inductive4D}, condition \refcondu{C} holds by
  \cref{prop:inductive5Dc}, and condition \refcondu{D} holds by
  \cref{prop:inducopositD}. The estimates \cref{eq:thmE-estimates}
  (with $ s=1 $) hold by   \cref{prop:inductive4D},
  \cref{prop:inductive5Dc}, and \cref{lem:uE0-uE1}.

  Fix $ x $, $ \exp(-N_1^{200\fd})\le |x-\ux_0|\le \exp(-N_1^{\fd})
  $. We will check that the conditions \refcond{A}-\refcond{E}, with $ s=1 $,
  hold for $ \uE^{[-N_1',N_1'']} $ with $ E_1=\uE^{[-N_0',N_0'']}(x)
  $. The conclusion then holds by noticing that
  \begin{multline*}
	\{ \uE^{[-N_0',N_0'']}(x): \exp(-N_1^{200\fd})\le |x-\ux_0|\le \exp(-N_1^{\fd})\}\\
    \supset [\uE_0+\exp(-N_1^{100\fd})/2,\uE_0+2\exp(-N_1^{2\fd})]
    \supset [\uE_1+\exp(-N_1^{100\fd}),\uE_1+\exp(-N_1^{2\fd})]
  \end{multline*}
  (recall \cref{lem:Morse7} and \cref{lem:uE0-uE1}).

  We apply \cref{prop:levelsetshifts} to $ \uE^{[-N_0',N_0'']} $ on
  $ \cP_0' $. Using the notation of \cref{prop:levelsetshifts},
  condition \refcondu{C}, and \cref{lem:uE-analyticity}, we have
  \begin{equation*}
	\nu_1\simeq \exp(-CN_0^\fd),
    \quad \rho=r_0'\nu_1^{10}\simeq \exp(-C'N_0^{\fd}),\quad  r=\nu_1\norm{x-\ux_0}. 
  \end{equation*}
  Since $ 0<\norm{x-\ux_0}<\rho $, \cref{prop:levelsetshifts} 
  applies with $ x $ in the role of $ x_0 $ and we get the following:

  \smallskip
  \noindent (1) There exists a  map $ x_0
  :\Pi_0\to \R^d $,
  \begin{equation*}
	\Pi_0= \cI_0\times (E_1-r^2,E_1+r^2), \quad \cI_0=(-r,r)^{d-1},\quad  E_1=\uE^{[-N_0',N_0'']}(x), 
  \end{equation*}
  such that
  \begin{equation*}
	\uE^{[-N_0',N_0'']}(x_0(\phi,E))=E,
  \end{equation*}
  $ x_0(\phi,E) $ extends analytically to
  \begin{equation*}
	\cP_0=\{ (\phi,E)\in \C^d: |\phi|<r,\ |E-E_1|<r^2 \}, 
  \end{equation*}
  and
  \begin{equation}\label{eq:x0-ux0}
	\norm{x_0(\phi,E)-\ux_0}<\norm{x-\ux_0}/2\lesssim \exp(-N_0^{5\fd})
  \end{equation}
  In particular, from the last estimate it follows that $
  x_0(\cP_0)\subset \T_{\rho/2}^d $.

  \smallskip
  \noindent (2) For any $ |E-E_1|<r^2 $, any vector $ h\in \R^d $
  with $ 0<\norm{h}<\rho $, and any $ H\gg 1 $, we have
  \begin{equation}\label{eq:uE0-shift}
	\mes \{ \phi\in \cI_0 : \log|\uE^{[-N_0',N_0'']}(x_0(\phi,E))-E|\le H_0H \}< \exp(-H^{1/(d-1)}),
  \end{equation}
  with $ H_0=C(d)\log(\norm{h}\norm{x-\ux_0}) $ (note that
  $
  \nu_1^{-2}r=\nu_1^{-1}\norm{x-\ux_0}<\nu_1^{-1}\rho=r_0'\nu_1^{9}<1
  $).

  \smallskip
  \noindent (3) Let $ h_0 $ be an arbitrary unit vector. For any $
  |E-E_1|<r^2 $, and any $ H\gg 1 $, we have
  \begin{equation}\label{eq:uE0-grad}
	\mes \{ \phi\in \cI_0: \log|\langle \nabla \uE^{[-N_0',N_0'']}(x_0(\phi,E)),h_0 \rangle|\le H_1H \}
    < \exp(-H^{1/(d-1)}),
  \end{equation}
  with $ H_1=C(d)\log(\nu_1\norm{x-\ux_0}) $.

  By \cref{eq:newE1D} and \cref{eq:x0-ux0} we have
  \begin{equation*}
	|\uE^{[-N_1',N_1'']}(x_0(\phi,E))-\uE^{[-N_0',N_0'']}(x_0(\phi,E))|
    =|\uE^{[-N_1',N_1'']}(x_0(\phi,E))-E|<\exp(-\gamma N_0/20),
  \end{equation*}
  for $ (\phi,E)\in \Pi_0 $. Then, just as in
  \cref{prop:inductive6nhh}, we can find a map $ x_1: \Pi_1''\to \R^d
  $,
  \begin{equation*}
	\Pi_1''=\cP_1''\cap \R^d,\quad \cP_1''=\{ (\phi,E)\in \C^d: |\phi|<r^{C_0},|E-E_1|<r^{C_0} \},
    \ C_0=C_0(d)\gg 1,
  \end{equation*}
  that extends analytically to $ \cP_1'' $, $ x_1(\cP_1'')\subset
  \T_{\rho/2}^d $, and such that
  \begin{equation}\label{eq:u-x0-x1}
	|x_1(\phi,E)-x_0(\phi,E)|<\exp(-\gamma N_0/30),\quad (\phi,E)\in \Pi_1''.
  \end{equation}
  Since  $ r\gtrsim
  \exp(-N_1^{300\fd}) $, we have that $ \cP_1 $ as defined in
  condition \refcond{B} (with $ \phi_1=0 $), satisfies $ \cP_1\subset
  \cP_1'' $ (recall that $ \fd\ll \delta $). Note that $
  |x_1(\phi,E)-\ux_0|\ll \exp(-2N_0^{4\fd}) $. Now, by
  \cref{prop:inductive4D}, conditions
  \refcond{A}-\refcond{C} hold with the above choice of
  parametrization $ x_1 $.

  We proceed to check condition \refcond{D}. The argument is based on
  applying the covering form of (LDT), similarly to
  \cref{prop:inducopositD}.
  We assume everything from
  the proof of \cref{prop:inducopositD}, up to and including
  \cref{eq:covering-Jm}, except that we take the lower bound for $ \dist(h,\fT_1)
  $ to be $  \exp(-N_1^\mu) $. Fix $ |E-E_1|<r^{C_0} $. By
  \cref{eq:covering-Jm} and  \cref{eq:x0-ux0}, 
  \begin{multline}\label{eq:covering-Jm-x0}
	\dist(\spec J_m(x_0(\phi,E)+h),(-\infty,\uE^{[-N_0',N_0'']}(x_0(\phi,E))])\\
    =
    \dist(\spec J_m(x_0(\phi,E)+h),(-\infty,E])\gtrsim \exp(-N_0^{4\fd}),
  \end{multline}
  provided $ |m+m_1|>3N_0/2 $.

  Now we focus on $ m=-m_1 $. We assume $ m_1\in [-N_1,N_1] $.  If $ \norm{h_1}\ge \exp(-N_0^{2\fd}) $,
  then $ \dist(h+m_1\omega,\fT_0)\ge
  \exp(-N_0^{2\fd}) $ and as above, there exists an interval
  $ J_{-m_1} $ such that \cref{eq:covering-Jm-x0} holds with
  $ m=-m_1 $. Suppose that $ \exp(-N_1^{\mu})\le \norm{h_1}<\exp(-N_0^{2\fd}) $.
  Let $J_{-m_1}:=-m_1+[-N_0',N''_0]$ and recall \cref{eq:J-m1}.  From
  \cref{eq:uE0-shift} with $ H=N_1^{2(d-1)\delta} $ (note that $
  \norm{h_1}\le \exp(-N_0^{2\fd})< \rho $), it follows that
  \begin{equation}\label{eq:uE-cB1}
	\mes \{ \phi\in \cI_0: |\uE^{[-N_0',N_0'']}(x_0(\phi,E)+h_1)-E|
    \le \exp(-N_1^{2\mu}) \}<\exp(-N_1^{2\delta})
  \end{equation}
  (we used $ \fd\ll\delta\ll\mu $, $ H_0\gtrsim
  -(N_1^{200\fd}+N_1^\mu)\gtrsim -N_1^{\mu} $). Using \refcondu{A} it
  follows that
  \begin{equation*}
	\dist(\spec H_{J_{-m_1}}(x_0(\phi,E)+h),E)>\exp(-N_1^{2\mu}),
  \end{equation*}
  for any $ \phi\in \cI_0\setminus \cB_1' $, where $ \cB_1' $ is the
  set from \cref{eq:uE-cB1}.
  
  We now have what we need to invoke the covering form of (LDT). We
  let the interval $ I $ be as in the proof of \cref{prop:inducopositD}.
  By the above, we can use \cref{lem:Greencoverap1} (with
  $ K=N_1^{2\mu}=N_0^{10\mu}\ll N_0^{\sigma/2} $; recall that $ \mu\ll \sigma$) to get that
  \begin{equation*}
    \dist( \spec H_{I}(x_0(\phi,E)+h),E)\ge\exp( -2N_1^{2\mu})=\exp(-2N_0^{10\mu}),
  \end{equation*}
  for any $ \phi\in \cI_0\setminus \cB_1' $. Let $ \cI_1''=\Proj_\phi
  \Pi_1'' $. Then, using \cref{eq:u-x0-x1}, we get
  \begin{equation*}
    \dist( \spec H_{I}(x_1(\phi,E)+h),E)\ge\exp( -3N_1^{2\mu})>\exp(-N_1^{\beta}/2),
  \end{equation*}
  for any $ \phi\in \cI_1''\setminus \cB_1' $ (recall that $ \beta\gg
  \mu $). This implies that condition \refcond{D} holds.

  Finally, we check condition \refcond{E}. Fix $ |E-E_1|<r^{C_0} $ and
  $ h_0\in\R^d $ a unit vector. By
  \cref{eq:uE0-grad} with $ H=N_1^{2(d-1)\delta} $, 
  \begin{equation*}
	\mes \{ \phi\in \cI_0: \log|\langle \nabla \uE^{[-N_0',N_0'']}(x_0(\phi,E)),h_0 \rangle|<-N_1^{\mu}/4 \}
    <\exp(-N_1^{2\delta})
  \end{equation*}
  (we used $ HH_1\gtrsim -N_1^{200\fd}N_1^{2(d-1)\delta}\gg N_1^{\mu}
  $; recall that $ \mu\gg \delta\gg \fd $). Now condition \refcond{E}
  follows by using \cref{eq:u-x0-x1} and \cref{lem:uE-analyticity}.
\end{proof}

\section{From Conditions on Potential to Inductive Conditions}\label{sec:A-to-DE}

We start by assuming that $ V $ attains its absolute extrema at
exactly one non-degenerate critical point and show that for large
enough coupling we can satisfy the initial inductive conditions from
\cref{sec:edges}. This means that we are working with operators of the
form \cref{eq:H-lambda}. Having the assumption be about both absolute
extrema is just a matter of convenience, it will be clear that they
can be handled separately.

Let $ \ux $, $ \ox $,  be the points where the absolute minimum and maximum of $ V $ are
attained. Since $ \ux $, $ \ox $ are assumed to be non-degenerate critical
points they will be isolated from the other critical points. We give a quantitative version of this
observation. We use $\mathfrak{E}$ to denote the set of  critical points of $ V $. 

\begin{lemma}\label{lem:Morse7a}
  Given $ x_0,x_1\in \fE $, such that $ x_0 $ is non-degenerate, we have
  \begin{equation*}
    \norm{x_0-x_1}\ge c_\rho\norm{\fH(x_0)^{-1}}^{-1}
    (1+ \norm{V}_\infty)^{-1}.
  \end{equation*}
\end{lemma}
\begin{proof}
  By Taylor's formula and Cauchy estimates,
  \begin{multline*}
	\norm{\nabla V(x)}
    =\norm{\nabla V(x)-\nabla V(x_0)}\ge \norm{\fH(x_0)(x-x_0)}-C_\rho \norm{V}_\infty \norm{x-x_0}^2\\
    \ge \frac{1}{2}\norm{\fH(x_0)^{-1}}^{-1}\norm{x-x_0},
  \end{multline*}
  provided $ \norm{x-x_0}\le c_\rho
  \norm{\fH(x_0)^{-1}}^{-1}(1+ \norm{V}_\infty)^{-1} $. The conclusion follows. 
\end{proof}

Note that $ \fE $ is compact and since $ \ux,\ox $ are isolated, $
\fE\setminus \{ \ux,\ox \} $ is also compact. Therefore there exists 
$ \iota=\iota(V)>0 $, such that
\begin{equation}\label{eq:iota}
  V(\underline{x})+\iota\le V(x)\le V(\overline{x})-\iota,
  \quad x\in \fE\setminus \{ \ux,\ox \}.
\end{equation}
Let
\begin{equation*}
  \nu:=\min(\norm{\fH(\ux)^{-1}}^{-1},\norm{\fH(\ox)^{-1}}^{-1})
\end{equation*}
Note that since $ \ux $, $ \ox $ are non-degenerate extrema, we have
\begin{equation*}
  \fH(\ux)\ge \nu I,\qquad \fH(\ox)\le -\nu I.
\end{equation*}

\begin{lemma}\label{lem:uV-shifts}
  Let $ r=c\nu(1+\norm{V}_\infty)^{-1} $, with $ c=c(\rho) $
  sufficiently small. Then
  \begin{equation}\label{eq:uV-shifts}
    \begin{gathered}
      \frac{\nu}{2}\norm{x-\ux}^2\le V(x)-V(\ux)\le C_\rho(1+\norm{V}_\infty)\norm{x-\ux}^2,
      \quad \norm{x-\ux}\le r,\\
      \frac{\nu}{2}\norm{x-\ux}\le \norm{\nabla V(x)}\le C_\rho(1+\norm{V}_\infty)\norm{x-\ux},
      \quad \norm{x-\ux}\le r, \\
      \min(\iota, \nu r^2/2)\le V(x)-V(\ux), \quad \norm{x-\ux}\ge r.
    \end{gathered}
  \end{equation}
  Analogous estimates hold for $ \ox $.
\end{lemma}
\begin{proof}
  The estimates with $ \norm{x-\ux}\le r $ follow from \cref{lem:Morse7} (we use Cauchy
  estimates to control $ M(3) $). From \cref{lem:Morse7a} we have
  that, by choosing $ r $ small enough,
  \begin{equation*}
    \fE\setminus \{ \ux
    \}\subset \T^d \setminus \{ x: \norm{x-\ux}\le r \}.
  \end{equation*}
  Then
  \begin{equation*}
	\min_{\norm{x-\ux}\ge r} (V(x)-V(\ux))=\min \left( \min_{x\in \fE\setminus \{ \ux \}}(V(x)-V(\ux)),
      \min_{\norm{x-\ux}=r}(V(x)-V(\ux))\right)
  \end{equation*}
  and the conclusion follows.
\end{proof}

For the purpose of the next result we update $ T_V $ (recall \cref{eq:TV}), to
be
\begin{equation}\label{eq:TV-a}
  T_{V}=2+\max(0,\log\norm{V}_\infty)+\max(0,\log \uiota ^{-1})
  +\max(0,\log\iota^{-1})+\max(0,\log\nu^{-1}).
\end{equation}
Clearly all the previous results using $ T_V $ also hold with this
possibly larger $ T_V $. The proofs of the next proposition and later
of \cref{prop:A-to-D} are very similar to the proofs of \cref{thm:E}
and \cref{thm:D} respectively, with some of the tools from
\cref{sec:basic-tools} replaced by their analogues from
\cref{sec:perturbative-refinements}. Due to the similarity we omit
some details. However, for clarity, we do give complete proofs, as the
key differences are spread out. Recall the exponent $ \fd $ from the
inductive conditions \refcondu{A}-\refcondu{D}.

\begin{prop}\label{prop:A-to-E}
  Assume the notation of conditions \refcondu{A}-\refcondu{D} from
  \cref{sec:edges}. Let $ \epsilon>0 $.
  There exists $ \lambda_0=\exp((T_V)^C) $, $ C=C(a,b,\rho,\epsilon) $, such
  that the following hold for $ \lambda\ge \lambda_0 $.  For any
  $ (\log\lambda)^{C(a,b,\epsilon)} \le N_0\le \exp((\log\lambda)^{\epsilon/2})$
  there exists $ \ux_0\in \T^d $, $ |\ux_0-\ux|\ll \lambda^{-1/3} $, such that the conditions
  \refcondu{A}-\refcondu{D} hold with $ s=0 $,
  $ \gamma=(\log\lambda)/2 $, $ [-N_0',N_0'']=[-N_0,N_0]$, and $
  |\lambda^{-1}\uE_0-V(\ux)|\ll \lambda^{-1/4} $. Furthermore, for any $ E_0\in \R $,
  $ \exp(-N_0^{100\fd})\le |E_0-\uE_0|\le \lambda\exp(-(\log\lambda)^{4\epsilon}) $, conditions
  \refcond{A}-\refcond{E}, with $ s=0 $, hold  for $ \uE^{[-N_0,N_0]}
  $.  Analogous statements hold relative to conditions
  \refcondo{A}-\refcondo{D}.
\end{prop}
\begin{proof}
  To check \refcondu{D} we will need to obtain conditions
  \refcondu{A}-\refcondu{C} not just for $ [-N_0',N_0'']=[-N_0,N_0] $,
  but also for other intervals.   
  By \cref{lem:uV-shifts}, for any $ 0<|n|\le 2N_0 $ we either have
  \begin{equation*}
	V(\ux+n\omega)-V(\ux)\ge \frac{\nu}{2}\norm{n\omega}^2\ge \frac{\nu}{2} a (2N_0)^{-b},
  \end{equation*}
  or
  \begin{equation*}
	V(\ux+n\omega)-V(\ux) \ge \min(\iota,\nu r^2/2).
  \end{equation*}
  Then for large enough $ \lambda $ (this
  is why we added $ \max(0,\log\iota^{-1})+\max(0,\log\nu^{-1}) $ to $ T_V $) and $
  N_0 $ not too large, we have
  \begin{equation*}
	V(\ux+n\omega)-V(\ux)\ge \exp(-(\log\lambda)^\epsilon),\quad 0<|n|\le N_0.
  \end{equation*}
  Let $ a<0<b $, $ [a,b]\subset [-2N_0,2N_0] $. 
  Then by \cref{lem:efextension2}, there exists $
  \uE^{[a,b]}=E_{k}^{[a,b]} $ such that for any $
  |x-\ux|<\exp(-3(\log\lambda)^\epsilon) $,
  \begin{equation}\label{eq:uE-loc}
    \begin{gathered}
      |\lambda^{-1}\uE^{[a,b]}(x)-V(x)|\le 2\lambda^{-1}\\
      |\upsi^{[a,b]}(x,n)|<\exp(-(\log \lambda) |n|/2),\quad |n|>0,\\
      \lambda^{-1}(E_j^{[a,b]}(x)-\uE^{[a,b]}(x))\ge
      \frac{1}{8}\exp(-(\log\lambda)^{\epsilon}),\quad j\neq k.
    \end{gathered}	
  \end{equation}
  As in \cref{lem:uE-analyticity}, $ \uE^{[a,b]} $ is analytic on
  \begin{equation*}
	\cP'=\{ z\in \C^d: |z-\ux|<r' \},\quad r'=\exp(-4(\log\lambda)^\epsilon)
  \end{equation*}
  and
  \begin{equation*}
	\sup_{\cP'} \norm{\fH(\lambda^{-1}\uE^{[a,b]})-\fH(V)}\le \lambda^{-c(d)}.
  \end{equation*}
  As in \cref{prop:inductive5Dc}, we can find $ \tilde \ux=\tilde \ux([a,b]) $, $
  |\tilde \ux-\ux|\ll \lambda^{-1/3} $, such that
  \begin{gather*}
	\uE^{[a,b]}(\tilde \ux)\le \uE^{[a,b]}(x),\quad \text{ for any }|x-\ux|<r',\\
    \nabla \uE^{[a,b]}(\tilde \ux)=0,\quad \fH(\lambda^{-1}\uE^{[a,b]})(\tilde \ux)\ge \frac{\nu}{4}I.
  \end{gather*}
  Also, as in \cref{lem:uE0-uE1}, we have $ |\lambda^{-1}\tilde
  \uE-V(\ux)|\ll \lambda^{-1/4} $, where $ \tilde \uE=\uE^{[a,b]}(\tilde \ux) $.
  We need to work around the
  weakness of the estimate $  |\tilde \ux-\ux|\ll \lambda^{-1/3} $. From
  now on assume $ [a,b]\supset [-\hat N,\hat N] $, $ \hat
  N=\ceil{N_0^{1/4}} $. By \cref{cor:close-ev-lambda}, we have
  \begin{equation*}
	|\uE^{[a,b]}(x)-\uE^{[-\hat N,\hat N]}(x)|\lesssim \exp(-(\log\lambda)\hat N/2),
  \end{equation*}
  for any $ |x-\ux|<\exp(-3(\log\lambda)^\epsilon) $.
  Let $ \hat \ux =\tilde \ux([-\hat N,\hat N])$.   As in
  \cref{prop:inductive5Dc}, we can find, with a slight abuse of notation, $ \tilde \ux=\tilde \ux([a,b])  $,
  \begin{equation}\label{eq:tux-hux}
    |\tilde \ux-\hat \ux|<\exp(-(\log\lambda) \hat N/5),
  \end{equation}
  such that
  \begin{equation}\label{eq:umin-ab}
    \begin{gathered}
      \uE^{[a,b]}(\tilde \ux)\le \uE^{[a,b]}(x),
      \quad \text{ for any }|x-\hat \ux|<\exp(-C(\log\lambda)^\epsilon),\\
      \nabla \uE^{[a,b]}(\tilde \ux)=0,\quad \fH(\lambda^{-1}\uE^{[a,b]})(\tilde \ux)\ge \frac{\nu}{8}I.
    \end{gathered}
  \end{equation}
  Furthermore, as in \cref{lem:uE0-uE1},
  \begin{equation}\label{eq:tuE-huE}
	|\tilde \uE-\hat{\uE} |<\exp(-(\log\lambda)\hat N/6),
  \end{equation}
  with $ \hat \uE=\uE^{[-\hat N,\hat N]}(\hat \ux) $.
  Note that
  \begin{equation}\label{eq:tux-ux}
	|\tilde \ux-\ux|\ll \lambda^{-1/3},\quad |\lambda^{-1}\tilde \uE-V(\ux)|\ll \lambda^{-1/4}.
  \end{equation}
  Let $ \ux_0=\tilde \ux([-N_0,N_0]) $. Then the first statement,
  except for condition \refcondu{D}, holds by all  the above
  and by having $ N_0^{\fd}\gg (\log\lambda)^\epsilon $. As in
  \cref{sec:edges} we incorporate the dependence on $ \fd $ in the
  dependence on the Diophantine parameters.

  Next we check condition $ \refcondu{D} $. First we consider the case
  $ \dist(h,\fT_0)\ge \exp(-(\log\lambda)^{2\epsilon}) $. Since $
  \norm{h+n\omega}\ge \exp(-(\log\lambda)^{2\epsilon}) $, we have, by \cref{lem:uV-shifts},
  \begin{equation*}
	V(\ux+h +n\omega)-V(\ux)\ge \exp(-3(\log\lambda)^{2\epsilon}),\quad |n|\le N_0
  \end{equation*}
  (provided $ \lambda $ is large enough). By
  \cref{cor:covering-perturb} we get
  \begin{equation*}
	\dist(\spec H_{[-N_0,N_0]}(\ux+h),(-\infty,\lambda V(\ux)])\gtrsim\lambda\exp(-3(\log\lambda)^{2\epsilon})
  \end{equation*}
  and by \cref{eq:tux-ux},
  \begin{equation}\label{eq:uD-large-h}
	\dist(\spec H_{[-N_0,N_0]}(\ux_0+h),(-\infty,\uE_0])\gtrsim\lambda\exp(-3(\log\lambda)^{2\epsilon})
    \gg \exp(-N_0^{4\fd}).
  \end{equation}
  Next we consider the case $ \exp(-N_0^{2\fd})\le
  \dist(h,\fT_0)<\exp(-(\log\lambda)^{2\epsilon}) $. Let $ n_1 $, $
  |n_1|\le 3N_0/2 $, such that $ \dist(h,\fT_0)=\norm{h-n_1\omega}
  $. We consider two sub-cases depending on the position of $ n_1
  $. If $ n_1\notin [-N_0+N_0^{1/3},N_0-N_0^{1/3}] $, then for $ n\in[-N_0+N_0^{1/3},N_0-N_0^{1/3}] $
  \begin{equation*}
	\norm{h+n\omega}\ge \norm{(n-n_1)\omega}-\norm{h-n_1\omega}\ge a N_0^{-b}-\exp(-(\log\lambda)^{2\epsilon})
    \gg \exp(-(\log\lambda)^{2\epsilon})
  \end{equation*}
  (recall that $ N_0\le \exp((\log\lambda)^{\epsilon/2}) $) and as above
  we get
  \begin{equation*}
    \dist(\spec H_{[-N_0+N_0^{1/3},N_0-N_0^{1/3}]}(\ux_0+h),(-\infty,\uE_0])
    \gtrsim\lambda\exp(-3(\log\lambda)^{2\epsilon})\gg \exp(-N_0^{4\fd}).
  \end{equation*}
  Suppose $ n_1\in [-N_0+N_0^{1/3},N_0-N_0^{1/3}] $. Let $
  h_1=h-n_1\omega $ (so, $ \norm{h_1}=\dist(h,\fT_0) $), $
  [a_1,b_1]=n_1+[-N_0,N_0] $, $ \tilde \ux_1=\tilde \ux([a_1,b_1]) $, $ \tilde
  \uE_1=\uE^{[a_1,b_1]}(\tilde \ux_1) $. Note that $
  [a_1,b_1]\supset [-\hat N,\hat N] $. By Taylor's formula (recall
  \cref{lem:Morse7}, \cref{eq:umin-ab}),
  \begin{equation*}
	\uE^{[a_1,b_1]}(\tilde \ux_1+h_1)-\uE^{[a_1,b_1]}(\tilde \ux_1)
    \ge \frac{\nu}{2}\norm{h_1}^2\ge \exp(-3N_0^{2\fd}).
  \end{equation*}
  Then, by \cref{eq:uE-loc} (recall \cref{eq:tux-ux}),
  \begin{equation*}
	\dist(\spec H_{[a_1,b_1]}(\tilde \ux_1+h_1),(-\infty,\tilde \uE_1])\ge \exp(-3N_0^{2\fd}).
  \end{equation*}
  Since $ \spec H_{[a_1,b_1]}(\tilde \ux_1+h_1)=\spec
  H_{[-N_0,N_0]}(\tilde \ux_1+h)
  $ and by \cref{eq:tux-hux}, \cref{eq:tuE-huE},
  \begin{equation*}
	|\tilde \ux_1-\ux_0|\lesssim \exp(-(\log\lambda)N_0^{1/4}/5),
    \quad |\tilde \uE_1-\uE_0|\lesssim \exp(-(\log\lambda)N_0^{1/4}/6),
  \end{equation*}
  it follows that
  \begin{equation*}
	\dist(\spec H_{[-N_0,N_0]}(\ux_0+h),(-\infty,\uE_0])\gtrsim \exp(-3N_0^{2\fd})\gg \exp(-N_0^{4\fd})
  \end{equation*}
  Thus, condition \refcondu{D} holds.

  Next we check the last statement. Let $ N_1=N_0^5 $. Since all
  the statements of the proof hold for a range of $ N_0 $, they will
  also hold for $ N_1 $, by adjusting the range. In particular, let
  $ \ux_1=\tilde \ux([-N_1,N_1]) $. Note
  that by \cref{eq:tux-hux}, \cref{eq:tuE-huE},
  \begin{equation}\label{eq:uxE0-uxE1}
    |\ux_1-\ux_0|\lesssim \exp(-(\log\lambda)N_0^{1/4}/5),
    \quad |\uE_1-\uE_0|\lesssim \exp(-(\log\lambda)N_0^{1/4}/6).
  \end{equation}
  Fix
  $ x $,
  $
  \exp(-N_1^{100\fd})\le |x-\ux_0|\le \exp(-(\log\lambda)^{2\epsilon})$.  We
  will check that conditions \refcond{A}-\refcond{E}, with
  $ s=1 $, hold for $ \uE^{[-N_1,N_1]} $ with
  $ E_1=\uE^{[-N_0,N_0]}(x) $. Then
  the conclusion holds since
  \begin{multline*}
	\{ \uE^{[-N_0,N_0]}(x) :\exp(-N_1^{200\fd})\le |x-\ux_0|\le \exp(-(\log\lambda)^{2\epsilon}) \}\\
    \supset [\uE_0+\lambda\exp(-N_1^{150\fd}),\uE_0+\lambda\exp(-(\log\lambda)^{3\epsilon})]\\
    \supset [\uE_1+\exp(-N_1^{100\fd}),\uE_1+\lambda\exp(-(\log\lambda)^{4\epsilon})]
  \end{multline*}
  (we applied \cref{lem:Morse7} to $ \lambda^{-1}\uE^{[-N_0,N_0]} $
  and we used \cref{eq:uxE0-uxE1}).
  Note that since this statement will hold for a
  range of $ N_1 $, it will also hold for the stated range of $ N_0 $
  by relabelling. 

  We apply \cref{prop:levelsetshifts} to
  $ \lambda^{-1}\uE^{[-N_0,N_0]} $ on
  \begin{equation*}
	\cP_0'=\{ z\in \C^d :|z-\ux_0|<\exp(-4(\log\lambda)^\epsilon) \}.
  \end{equation*}
  Using the notation of \cref{prop:levelsetshifts}, we have
  \begin{equation*}
	\nu_1\simeq \exp(-C(\log\lambda)^{\epsilon}),
    \quad \rho=\exp(-4(\log\lambda)^\epsilon)\nu_1^{10}\simeq \exp(-C'(\log\lambda)^\epsilon)
    ,\quad  r=\nu_1\norm{x-\ux_0}. 
  \end{equation*}
  We chose to apply \cref{prop:levelsetshifts} to $
  \lambda^{-1}E^{[-N_0,N_0]} $ because of the $ 0<\nu_0<1 $
  restriction in the statement of the proposition. Of course, we could
  artificially choose any $ \nu_0\in (0,1) $ for $ E^{[-N_0,N_0]} $,
  but this would result in a much smaller $ \nu_1\simeq
  \lambda^{-1}\exp(C(\log\lambda)^\epsilon) $, which is too small for
  our purposes.
  Since $ 0<\norm{x-\ux_0}<\rho $, \cref{prop:levelsetshifts} 
  applies   with $ x $ in the role of
  $ x_0 $ and we get the following:

  \smallskip
  \noindent (1) There exists a  map $ x_0
  :\Pi_0\to \R^d $,
  \begin{equation*}
	\Pi_0= \cI_0\times (E_1-\lambda r^2,E_1+\lambda r^2), \quad \cI_0=(-r,r)^{d-1},
    \quad  E_1=\uE^{[-N_0,N_0]}(x), 
  \end{equation*}
  such that
  \begin{equation*}
	\uE^{[-N_0,N_0]}(x_0(\phi,E))=E,
  \end{equation*}
  $ x_0(\phi,E) $ extends analytically to
  \begin{equation*}
	\cP_0=\{ (\phi,\eta)\in \C^d: |\phi|<r,\ |E-E_1|<\lambda r^2 \}, 
  \end{equation*}
  and
  \begin{equation}\label{eq:x0-ux0-lambda}
	\norm{x_0(\phi,E)-\ux_0}<\norm{x-\ux_0}/2\lesssim \exp(-(\log\lambda)^{2\epsilon}).
  \end{equation}
  From the last estimate it follows that $
  x_0(\cP_0)\subset \T_{\rho/2}^d $. Of course,
  \cref{prop:levelsetshifts} actually gives a function $ \tilde
  x_0(\phi,\eta) $, such that $ \lambda^{-1}\uE^{[-N_0,N_0]}(\tilde
  x_0(\phi,\eta))=\eta $, and we get the above statement by setting $
  x_0(\phi,E)=\tilde x_0(\phi,\lambda^{-1}E) $.

  \smallskip
  \noindent (2) For any $ |E-E_1|<\lambda r^2 $, any vector $ h\in \R^d $
  with $ 0<\norm{h}<\rho $, and any $ H\gg 1 $, we have
  \begin{equation}\label{eq:uE0-shift-lambda}
	\mes \{ \phi\in \cI_0 : \log|\uE^{[-N_0,N_0]}(x_0(\phi,E))-E|\le H_0H \}< \exp(-H^{1/(d-1)}),
  \end{equation}
  with $ H_0=C(d)\log(\norm{h}\norm{x-\ux_0}) $.

  \smallskip
  \noindent (3) Let $ h_0 $ be an arbitrary unit vector. For any $
  |E-E_1|<\lambda r^2 $, and any $ H\gg 1 $, we have
  \begin{equation}\label{eq:uE0-grad-lambda}
	\mes \{ \phi\in \cI_0: \log|\langle \nabla \uE^{[-N_0,N_0]}(x_0(\phi,E)),h_0 \rangle|\le H_1H \}
    < \exp(-H^{1/(d-1)}),
  \end{equation}
  with $ H_1=C(d)\log(\nu_1\norm{x-\ux_0}) $.
  By \cref{cor:close-ev-lambda},
  \begin{equation}\label{eq:uE0-uE1}
	|\uE^{[-N_1,N_1]}(x)-\uE^{[-N_0,N_0]}(x)|\lesssim \exp(-(\log\lambda) N_0/2),
    \quad |x-\ux|<\exp(-3(\log\lambda)^\epsilon),
  \end{equation}
  and therefore
  \begin{equation*}
	|\uE^{[-N_1,N_1]}(x_0(\phi,E))-\uE^{[-N_0,N_0]}(x_0(\phi,E))|
    =|\uE^{[-N_1,N_1]}(x_0(\phi,E))-E|\lesssim \exp(-(\log\lambda) N_0/2),
  \end{equation*}
  for $ (\phi,E)\in \Pi_0 $. Then, just as in
  \cref{prop:inductive6nhh}, we can find a map $ x_1: \Pi_1''\to \R^d
  $,
  \begin{equation*}
	\Pi_1''=\cP_1''\cap \R^d,\quad \cP_1''=\{ (\phi,E)\in \C^d: |\phi|<r^{C_0},|E-E_1|<r^{C_0} \},
    \ C_0=C_0(d)\gg 1,
  \end{equation*}
  that extends analytically to $ \cP_1'' $, $ x_1(\cP_1'')\subset
  \T_{\rho/2}^d $, and such that
  \begin{equation}\label{eq:u-x0-x1-lambda}
	|x_1(\phi,E)-x_0(\phi,E)|<\exp(-(\log\lambda) N_0/3),\quad (\phi,E)\in \Pi_1''.
  \end{equation}
  In fact the domain in $ E $ is much larger, but we have no use for
  this improvement.
  Since $ r\gtrsim
  \exp(-N_1^{200\fd}) $ , we have that $ \cP_1 $ as defined in
  condition \refcond{B} (with $ \phi_1=0 $), satisfies $ \cP_1\subset
  \cP_1'' $ (recall that $ \fd\ll \delta $). Note that $
  |x_1(\phi,E)-\ux_0|\ll \exp(-3(\log\lambda)^\epsilon) $. Now, conditions
  \refcond{A}-\refcond{C} hold with the above choice of
  parametrization $ x_1 $ (recall that we have \cref{eq:uE-loc} with $ [a,b]=[-N_1,N_1] $).

  We proceed to check condition \refcond{D}. Let
  $|m_1|\le 3N_1/2$, $h_1\in \mathbb{R}^d$ such that
  \begin{equation*}
    \dist (h, \mathfrak{T}_{1})=\|h_1\|,\qquad
    h_1=h-m_1\omega \,(\mod \mathbb{Z}^d).
  \end{equation*}
  Recall that we have \cref{eq:dist-h-T0}.  We consider two cases:
  $ \norm{h_1}\ge \exp(-(\log\lambda)^{2\epsilon}) $ and
  $ \exp(N_1^{\mu})\le \norm{h_1}<\exp(-(\log\lambda)^{2\epsilon})
  $. In either case, by \cref{eq:dist-h-T0}, we have
  $ \dist(h+m\omega,\fT_0)\ge \exp(-(\log\lambda)^{2\epsilon}) $ for all
  $ m\in [-N_1,N_1] $ with $ |m+m_1|>3N_0/2 $.  For such $ m $,
  \cref{eq:uD-large-h} implies 
  \begin{equation}\label{eq:covering-Jm-lambda}
    \dist (\spec  H_{J_m}(\ux_0+h),(-\infty,\uE_0])
    \gtrsim \lambda \exp(-3(\log\lambda)^{2\epsilon}),
  \end{equation}
  with $ J_m=m+[-N_0,N_0] $.
  Fix $ |E-E_1|<r^{C_0} $. By
  \cref{eq:covering-Jm-lambda} and \cref{eq:x0-ux0-lambda},
  \begin{multline}\label{eq:covering-Jm-x0-lambda}
	\dist(\spec H_{J_m}(x_0(\phi,E)+h),(-\infty,\uE^{[-N_0',N_0'']}(x_0(\phi,E))])\\
    =
    \dist(\spec H_{J_m}(x_0(\phi,E)+h),(-\infty,E])\gtrsim \lambda \exp(-3(\log\lambda)^{2\epsilon}),
  \end{multline}
  provided $ |m+m_1|>3N_0/2 $.

  Now we focus on $ m=-m_1 $. We assume $ m_1\in [-N_1,N_1] $.  Let
  $J_{-m_1}:=-m_1+[-N_0,N_0]$. If
  $ \norm{h_1}\ge \exp(-(\log\lambda)^{2\epsilon}) $, then
  $ \dist(h+m_1\omega,\fT_0)\ge \exp(-(\log\lambda)^{2\epsilon}) $ and as above,
  \cref{eq:covering-Jm-x0-lambda} holds with $ m=-m_1 $. Suppose that
  $ \exp(-N_1^{\mu})\le \norm{h_1}<\exp(-(\log\lambda)^{2\epsilon}) $. From
  \cref{eq:uE0-shift-lambda} with $ H=N_1^{2(d-1)\delta} $, it follows that
  \begin{equation}\label{eq:uE-cB1-lambda}
	\mes \{ \phi\in \cI_0: |\uE^{[-N_0,N_0]}(x_0(\phi,E)+h_1)-E|
    \le \exp(-N_1^{2\mu}) \}<\exp(-N_1^{2\delta})
  \end{equation}
  (we used $ \fd\ll\delta\ll\mu $, $ H_0\gtrsim
  -(N_1^{200\fd}+N_1^\mu)\gtrsim -N_1^{\mu} $). Using \cref{eq:uE-loc} it
  follows that
  \begin{equation*}
	\dist(\spec H_{J_{-m_1}}(x_0(\phi,E)+h),E)>\exp(-N_1^{2\mu}),
  \end{equation*}
  for any $ \phi\in \cI_0\setminus \cB_1' $, where $ \cB_1' $ is the
  set from \cref{eq:uE-cB1-lambda}.
  
  Let $
  I $ be an interval as in \cref{eq:I}.
  By the above, we can use \cref{lem:Greencoverap1} (with
  $ K=N_1^{2\mu}=N_0^{10\mu}\ll N_0^{\sigma/2} $; recall that $ \mu\ll \sigma$) to get that
  \begin{equation*}
    \dist( \spec H_{I}(x_0(\phi,E)+h),E)\ge\exp( -2N_1^{2\mu})=\exp(-2N_0^{10\mu}),
  \end{equation*}
  for any $ \phi\in \cI_0\setminus \cB_1' $. Let $ \cI_1''=\Proj_\phi
  \Pi_1'' $. Then, using \cref{eq:u-x0-x1-lambda}, we get
  \begin{equation*}
    \dist( \spec H_{I}(x_1(\phi,E)+h),E)\ge\exp( -3N_1^{2\mu})>\exp(-N_1^{\beta}/2),
  \end{equation*}
  for any $ \phi\in \cI_1''\setminus \cB_1' $ (recall that $ \beta\gg
  \mu $). This implies that condition \refcond{D} holds.

  Finally, we check condition \refcond{E}. Fix $ |E-E_1|<r^{C_0} $ and
  $ h_0\in\R^d $ a unit vector. By
  \cref{eq:uE0-grad-lambda} with $ H=N_1^{2(d-1)\delta} $, 
  \begin{equation*}
	\mes \{ \phi\in \cI_0: \log|\langle \nabla \uE^{[-N_0,N_0]}(x_0(\phi,E)),h_0 \rangle|<-N_1^{\mu}/4 \}
    <\exp(-N_1^{2\delta})
  \end{equation*}
  (we used $ HH_1\gtrsim -N_1^{200\fd}N_1^{2(d-1)\delta}\gg N_1^{\mu}
  $; recall that $ \mu\gg \delta\gg \fd $). Now condition \refcond{E}
  follows by using \cref{eq:u-x0-x1-lambda} and Cauchy estimates.  
\end{proof}

For the rest of the section we assume that $ V\in \fG $, recall
\cref{defi:genericU}, and show that, for large enough coupling, we can
we can satisfy the initial inductive conditions from
\cref{sec:bulk}. In fact, it will be clear that we only use properties
(iii) and (iv) from the definition of $ \fG $. The first two
properties will only be needed in the proof of \cref{thm:A} (b). We
fix the constants $ \fc_0,\fc_1,\fC_0 $ from \cref{defi:genericU}.

\begin{prop}\label{prop:MSAIMPLICIT21}  
  Let $ x_0\in \T^d $, $ \eta_0=V(x_0) $ and assume $ \mu_0:=\norm{\nabla V(x_0)}>0 $. Let
  \begin{equation*}
    r=\min(\rho/4,c\mu_0^2(1+\norm{V}_\infty)^{-2}),
  \end{equation*}
  with $ c=c(\rho) $ small
  enough. There exists a map $ x:\Pi\to \R^d $,
  \begin{equation*}
    \Pi=\cI\times (\eta_0-r,\eta_0+r),\quad \cI=x_0+(-r,r)^{d-1},
  \end{equation*}
  such that the following hold.

  \medskip\noindent (a) The map extends analytically on the domain
  \begin{equation*}
	\cP=\{ (\phi,\eta)\in \C^d : \dist((\phi,\eta),\Pi)<r \},
  \end{equation*}
  and
  \begin{equation*}
    x(\cP)\subset \T_{\rho/2}^d,\qquad
	V(x(\phi,\eta))=\eta,\quad  (\phi,\eta)\in \cP.
  \end{equation*}

  \medskip\noindent (b) For any $K\gg \fC_0+C_\rho\max(0,\log\norm{V}_\infty) $,
  $\|h\|\ge e^{-\mathfrak{c}_0K}$,
  and $ \eta\in(\eta_0-r,\eta_0+r) $,
  \begin{gather*}
    \mes\{\phi\in \cI:
    |V(x(\phi,\eta)+h)-\eta|<\exp(-K)\}<\exp(-K^{\fc_1}/10).
  \end{gather*}

  \medskip\noindent (c) Take an arbitrary unit vector $h_0\in
  \mathbb{R}^d$. For any $ K\ge \fC_0 $, $ \eta\in (\eta_0-r,\eta_0+r) $,
  \[
    \mes \{\phi\in \cI:\log |\langle \nabla V(x(\phi,\eta)),h_0 \rangle|<-K\}
    <\exp (-K^{\fc_1}).
  \]
\end{prop}
\begin{proof}
  There exists $ i $ such that $ \partial_{x_i} V(x_0)\ge \mu_0/d
  $. To simplify the notation, we assume that $ i=1 $.
  Let $ \rho_1\le c_\rho \mu_0(1+\norm{V}_\infty)^{-1} $ with $ c_\rho
  $ sufficiently small. 
  Applying 
  \cref{lem:impl2} (also recall \cref{rem:real-valued}) to $ V(x)-\eta $ near $ (x_0,\eta_0) $, we get that
  there exists an analytic function $ x_1(x_2,\dots,x_d,\eta) $ on
  \begin{equation*}
	|x_2-x_{0,2}|,\dots,|x_d-x_{0,d}|,|\eta-\eta_0|<\rho_1^2
  \end{equation*}
  such that
  \begin{gather*}
	|x_1(x_2,\dots,x_d,\eta)-x_{0,1}|<\rho_1,\\	V(x_1(x_2,\dots,x_d,\eta),x_2,\dots,x_d)=\eta.
  \end{gather*}
  The existence of the map and part (a) follow by setting
  \begin{equation*}
	x(\phi,\eta)=(x_1(\phi,\eta),\phi),\quad \phi=(x_2,\dots,x_d).
  \end{equation*}
  Our choice of $ r<\rho_1^2$ is made to ensure that $ x(\cP)\subset \T_{\rho/2}^d$.

  Fix $ \norm{h}\ge \exp(-\fc_0 K) $, $ \eta\in (\eta_0-r,\eta_0+r)$.
  Let
  \begin{equation}\label{MSA3MOD12}
    F(\phi)=V(x(\phi,\eta)+h)-\eta.
  \end{equation}
  Let $g(x):=g_{V,h,1,2}(x)$ be as in \cref{defi:genericU}.
  We have
  \begin{multline}\label{eq:Fx2}
    \partial_{x_2} F(\phi)=\partial_{x_1} V(x(\phi,\eta)+h)\partial_{x_2} x_1(\phi,\eta)
    +\partial_{x_2} V(x(\phi,\eta)+h)\\
    =-\partial_{x_1} V(x(\phi,\eta)+h)\frac{\partial_{x_2}V(x(\phi,\eta))}{\partial_{x_1}V(x(\phi,\eta))}
    +\partial_{x_2} V(x(\phi,\eta)+h)
    =\frac{g(x(\phi,\eta))}{\partial_{x_1}V(x(\phi,\eta))}.
  \end{multline}
  Let $ K\ge \fC_0$. By \cref{defi:genericU} (iii) we have that
  \begin{equation*}
    \mes \{x_{\hat 1}: \min_{x_1}(|V(x+h)-V(x)|+
    |g(x)|)
    <\exp(-K)\}\le \exp(-K^{\fc_1}).
  \end{equation*}
  In particular, it follows that
  \begin{equation}\label{eq:cBprime}
    \mes \{\phi\in \cI: |V(x(\phi,\eta)+h)-\eta|+
    |g(x(\phi,\eta))|
    <\exp(-K)\}\le \exp(-K^{\fc_1}).
  \end{equation}
  Let
  \begin{gather*}
	\cB=\{ \phi\in \cI : |V(x(\phi,\eta)+h)-\eta|<\exp(-5K) \},\\
    \cB''=\{ \phi\in \cI : |V(x(\phi,\eta)+h)-\eta|<\exp(-5K),\ |g(x(\phi,\eta))|\ge \exp(-K)/2 \},
  \end{gather*}
  and $ \cB' $ the set from \cref{eq:cBprime}. Then
  \begin{equation*}
	\cB\subset \cB'\cup \cB''.
  \end{equation*}
  We want to estimate $ \mes(\cB'') $.
  Let $ z=(x_3,\dots,x_d) $ and
  \begin{equation*}
	\cB_z''= \{ x_2 : \phi=(x_2,z)\in \cB'' \}.
  \end{equation*}
  Fix $ z=(x_3,\dots,x_d) $ with $ |x_i-x_{0,i}|<r $, $ i=3,\dots,d
  $. By truncating the Taylor series (for both $ V $ and $ x(\phi,\eta) $) we can find polynomials
  $P(x_2),Q(x_2)$ (depending on $ z $) of degree
  $\le C\max(1,\log\norm{V}_\infty)K^4 $, such that for any $ |x_2-x_{0,2}|<r$,
  \begin{equation*}
    |F(x_2,z)-P(x_2)|,\ |\partial_{x_2} F(x_2,z)-P'(x_2)|,\ |g(x(x_2,z,\eta))-Q(x_2)| \le \exp(-5K).
  \end{equation*}
  Then
  \begin{gather*}
    \cB''_{z} \subset \cB'''_{z}:=\{x_2\in (x_{0,2}-r,x_{0,2}+r): |P(x_2)|\le 2\exp(-5K),\quad
    |Q(x_2)|
    \ge {1\over 4}\exp(-K)\}.
  \end{gather*}
  Using \cref{eq:Fx2} and Cauchy estimates, we have that for any $
  x_2\in \cB_z''' $,
  \begin{gather*}
    |\partial_{x_2} F(x_2,z)|
    \gtrsim \rho\norm{V}_\infty^{-1}|g(x(x_2,z,\eta))|
    \gtrsim \rho\norm{V}_\infty^{-1}(|Q(x_2)|-e^{-5K})\gtrsim\rho\norm{V}_\infty^{-1}  \exp(-K),\\
    |P'(x_2)|\gtrsim (\rho\norm{V}_\infty^{-1} \exp(-K)-\exp(-5K))>\exp(-2K),
  \end{gather*}
  provided  $ K$ is large enough.  It
  follows that each connected component of $ \cB_z''' $ has length
  $ \lesssim \exp(2K)\exp(-5K) $. Since $\cB'''_{z}$ consists of the
  union of $\lesssim (\deg P+\deg Q) $ intervals, it follows that
  \begin{equation*}
	\mes(\cB_z''')\le C\max(1,\log\norm{V}_\infty) K^4\exp(-3K)<\exp(-2K). 
  \end{equation*}
  Then we have $ \mes(\cB'')<\exp(-K) $ (recall that $ \rho\le 1 $, so
  $ r\le 1/4 $), $ \mes(\cB)<\exp(-K^{\fc_1}/2) $,
  and statement (b) follows.

  Given $ K\ge \fC_0$, by \cref{defi:genericU} (iv) we have
  \begin{equation*}
    \mes \{x_{\hat 1}: \min_{x_1}(|V(x)-\eta|+
    |\langle \nabla V(x),h_0 \rangle|)
    <\exp(-K)\}\le \exp(-K^{\fc_1}).
  \end{equation*}
  In particular, it follows that
  \begin{equation*}
    \mes \{\phi\in \cI: |V(x(\phi,\eta))-\eta|+
    |\langle \nabla V(x(\phi,\eta),h_0 \rangle|)
    <\exp(-K)\}\le \exp(-K^{\fc_1}).
  \end{equation*}
  Since $ V(x(\phi,\eta))=\eta $, statement (c) follows.
\end{proof}

For the purpose of the next  result we update $ T_V $ again to be to be
\begin{equation*}
  T_{V}=2+\max(0,\log\norm{V}_\infty)+\max(0,\log \uiota ^{-1})
  +\max(0,\log\iota^{-1})+\max(0,\log\nu^{-1})+\fC_0+\fc_0^{-1}.
\end{equation*}
We don't include $ \fc_1^{-1} $ because it doesn't depend on $ V
$.

\begin{prop}\label{prop:A-to-D}
  There exists $ \lambda_0=\exp((T_V)^C) $, $ C=C(a,b,\rho) $ such that the following hold for
  $ \lambda\ge \lambda_0 $.  Let $ x_0\in \T^d $, $ \eta_0=V(x_0) $,
  and assume
  $ \norm{\nabla V(x_0)}\ge \exp(-(\log\lambda)^{\fc_1/3}) $. Then
  for any $ (\log\lambda)^{C(a,b)} \le N_0\le \exp((\log\lambda)^{\fc_1/3})$, the
  conditions \refcond{A}-\refcond{E} hold with $ s=0 $, $ \gamma=(\log\lambda)/2 $, $
  E_0=\lambda\eta_0 $,
  and some $ \phi_0\in \R^d $.
\end{prop}
\begin{proof}
  The proof is similar to that of \cref{thm:D}. As in \cref{thm:D} we leave the dependence on the
  exponents $ \delta,\beta,\mu $ implicit, as part of the dependence
  on the Diophantine condition parameters $ a,b $. 

  Due to the lower bound on $ \norm{\nabla V(x_0)} $, we can apply
  \cref{prop:MSAIMPLICIT21} with $ r=\exp(-3(\log\lambda)^{\fc_1/3}) $.
  Furthermore, since $ \lambda $ is large enough, we can apply \cref{prop:MSAIMPLICIT21} (b),(c) with $
  K\gtrsim (\log\lambda)^{1/2} $ (this is why we added $ \fC_0 $ to $ T_V
  $). In what follows we let $ \cI $, $ x(\phi,\eta) $, be as in \cref{prop:MSAIMPLICIT21}.
  Let
  \begin{equation*}
	\cB_{K,\eta,h}= \{ \phi\in \cI: |V(x(\phi,\eta))-\eta|<\exp(-K) \},
    \quad \cB_{\eta,h}=\cB_{(\log\lambda)^{1/2},\eta,h}.
  \end{equation*}
  By \cref{prop:MSAIMPLICIT21}, for any $
  \eta\in(\eta_0-r,\eta_0+r) $, $ \norm{h}\ge
  \exp(-\fc_0(\log\lambda)^{1/2}) $,
  \begin{equation*}
	\mes(\cB_{\eta,h})<\exp(-(\log\lambda)^{\fc_1/2}).
  \end{equation*}
  As in \cref{lem:sa} we can find a semialgebraic set $ \tilde
  \cB_{\eta,h} $ containing $ \cB_{\eta,h} $, of degree $ \le
  (\log\lambda)^3 $, and with measure $ \le
  \exp(-(\log\lambda)^{\fc_1/2}/2) $. Let
  \begin{equation*}
	\cB_{\eta_0,N_0}=\bigcup_{0<|n|\le 2N_0} \tilde \cB_{\eta_0,n\omega}. 
  \end{equation*}
  Since $ N_0\le \exp((\log\lambda)^{\fc_1/3}) $ we have $
  \norm{n\omega}\ge \exp(-\fc_0(\log\lambda)^{1/2}) $, $ 0<|n|\le 2N_0 $ (provided $
  \lambda $ is large enough; this why we added $ \fc_0^{-1} $ to $ T_V
  $), and $ \mes(\cB_{\eta_0,N_0})<\exp(-(\log\lambda)^{\fc_1/2}/4) $.
  Since $ \cB_{\eta_0,N_0} $ is also semialgebraic of degree less than
  $ \exp(2(\log\lambda)^{\fc_1/3})$, it follows, using
  \cref{lem:sa-covering}, that there exists $ \phi_0 $, $
  |\phi_0-x_0|\ll r $, such that
  \begin{equation*}
	\cI_0'\Subset \cI\setminus \cB_{\eta_0,N_0},
    \qquad \cI_0'=\phi_0+(-r_0',r_0')^{d-1},\quad r_0'=\exp(-(\log\lambda)^{\fc_1/3}).
  \end{equation*}
  Let $ a<0<b $, $ [a,b]\subset [-2N_0,2N_0] $. We consider such
  general intervals for reasons similar to the ones in \cref{prop:A-to-E}. As in
  \cref{prop:inductive4}, but using \cref{lem:efextension2} (with $
  x_0=x(\phi,\eta_0) $, $ \phi\in \cI_0' $) instead of
  \cref{prop:stabilization}, we get that there exists $ k $ such that
  for any $ \phi\in \cI_0' $, $ y\in \R^d $, $ |y|<\exp(-4(\log\lambda)^{1/2}) $,
  $ |\eta-\eta_0|<\exp(-4(\log\lambda)^{1/2}) $,
  \begin{equation}\label{eq:loc-estimates}
    \begin{gathered}
      |\lambda^{-1}E_{k}^{[a,b]}(x(\phi,\eta)+y)-V(x(\phi,\eta)+y)|\le 2\lambda^{-1},\\
      \lambda^{-1}|E_j^{[a,b]}(x(\phi,\eta)+y)-E_{k}^{[a,b]}(x(\phi,\eta)+y)|
      >\frac{1}{8}\exp(-(\log\lambda)^{1/2}),\quad j\neq k,\\
      |\psi_{k}^{[a,b]}(x(\phi,\eta)+y,n)|<\exp(-(\log\lambda)|n|/2),\quad |n|>0.      
    \end{gathered}
  \end{equation}
  To simplify notation we will drop the index $ k $ and write $
  E^{[a,b]},\psi^{[a,b]} $. Let
  \begin{equation*}
	\cP_0''=\{ (\phi,E)\in \C^d : |\phi-\phi_0|,|E-E_0|<r_0'' \},\quad
    r_0''=\exp(-C_0(\log\lambda)^{\fc_1/3}).
  \end{equation*}
  $ C_0=C_0(d)\gg 1 $. Let $ \Pi_0''=\cP_0''\cap \R^d $, $
  \cI_0''=\Proj_{\phi}\Pi_0'' $. As in \cref{prop:inductive6nhh}, we
  can find an analytic map $ \tilde x(\phi,\eta) $ such that
  \begin{equation*}
	\lambda^{-1}E^{[a,b]}(\tilde x(\phi,\eta))=\eta,
  \end{equation*}
  for any $ (\phi,\lambda\eta)\in \cP_0'' $ and
  \begin{equation}\label{eq:delta-level}
	|\tilde x(\phi,\eta)-x(\phi,\eta)|\le \lambda^{-1/2},
  \end{equation}
  for $ (\phi,\lambda\eta)\in \Pi_0'' $ (in fact, in the definition of
  $ \cP_0'' $ we could take $ |E-E_0|<\lambda\exp(-C_0(\log\lambda)^{1/2}) $). We note that at this point,
  we have what we need for conditions \refcond{A}-\refcond{C} to hold.
  However, to check condition \refcond{D} we need to set things up
  more carefully. The problem we need to work around is the weakness
  of \cref{eq:delta-level}. From now on we assume that
  $ [a,b]\supset [-\uN,\uN] $, $ \uN=\ceil{N_0^{1/4}} $. Let $ \ux $
  be the parametrization obtained as above, so that
  \begin{equation*}
	\lambda^{-1}E^{[-\uN,\uN]}(\ux(\phi,\eta))=\eta.
  \end{equation*}
  By \cref{cor:close-ev-lambda} we have
  \begin{equation*}
	|E^{[a,b]}(x(\phi,\eta)+y)-E^{[-\uN,\uN]}(x(\phi,\eta)+y)|\lesssim \exp(-(\log\lambda) \uN/2).
  \end{equation*}
  for any $ |y|<r_0' $, $ (\phi,\lambda \eta)\in \Pi_0' $. Using
  \cref{eq:delta-level} (with $ \tilde x=\ux $) it follows that
  \begin{equation*}
	|E^{[a,b]}(\ux(\phi,\eta))-\lambda \eta|
    =|E^{[a,b]}(\ux(\phi,\eta))-E^{[-\uN,\uN]}(\ux(\phi,\eta))|\lesssim \exp(-(\log\lambda) \uN/2).
  \end{equation*}
  Again, as in \cref{prop:inductive6nhh}, we get that there exists a
  map $ \tilde x(\phi,\eta) $ such that
  \begin{equation*}
	E^{[a,b]}(\tilde x(\phi,\eta))=\lambda\eta, \qquad (\phi,\lambda\eta)\in \cP_0'',
  \end{equation*}
  and for $ (\phi,\lambda\eta)\in \Pi_0'' $,
  \begin{equation}\label{eq:better-delta-level}
	|\tilde x(\phi,\eta)-\ux(\phi,\eta)|\le \exp(-(\log\lambda)\uN/4).
  \end{equation}
  To justify keeping the same domain $ \cP_0'' $ as before we can
  increase the constant $ C_0 $ from its definition. Note that we
  still have
  \begin{equation*}
	|\tilde x(\phi,\eta)-x(\phi,\eta)|\lesssim \lambda^{-1/2},
  \end{equation*}
  and therefore (using \cref{eq:loc-estimates}) conditions
  \refcond{A}-\refcond{C} hold with
  $ x_0(\phi,E)=\tilde x(\phi,\lambda^{-1}E) $,
  $ [-N_0',N_0'']=[-N_0,N_0] $. Of course, we are assuming $ N_0 $ is large
  enough so that $ r_0=\exp(-N_0^\delta)\ll r_0'' $.

  Next we check condition \refcond{E}, as in \cref{prop:inductiveDcond}. Let $ h_0\in \R^d $ a unit
  vector, $ \eta\in (\eta_0-r_0'',\eta_0+r_0'') $. By \cref{prop:MSAIMPLICIT21}
  (c),
  \begin{equation*}
	\mes\{\phi\in \cI: |\langle \nabla V(x(\phi,\eta)),h_0 \rangle|<\exp(-(\log\lambda)^{1/2})\}
    <\exp(-(\log\lambda)^{\fc_1/2}).
  \end{equation*}
  Since $ \exp(-(\log\lambda)^{\fc_1/2})\ll \mes(\cI_0'') $, it
  follows that there exists $ \hat \phi $, $ |\hat \phi-\phi_0|\ll
  r_0'' $, such that
  \begin{equation*}
	|\langle \nabla V(x(\hat \phi,\eta)),h_0 \rangle|\ge \exp(-(\log\lambda)^{1/2})
  \end{equation*}
  and therefore
  \begin{equation}\label{eq:grad-lb}
	|\langle \nabla E^{[a,b]}(\tilde x(\hat \phi,\eta)),h_0 \rangle|
    \gtrsim \lambda \exp(-(\log\lambda)^{1/2})
  \end{equation}
  (we used the first estimate in \cref{eq:loc-estimates},
  \cref{eq:delta-level}, \cref{cor:high_cart}, and Cauchy estimates).
  Then Cartan's estimate  yields that given $ H\gg 1 $,
  \begin{equation*}
	\mes \{ \phi\in \cI_0''/10: 	|\langle \nabla E^{[a,b]}(\tilde x(\phi,\eta)),h_0 \rangle|
    <\log\lambda -CH(\log\lambda)^{1/2}\}<C(d)(r_0'')^{d-1}\exp(-H^{1/(d-1)}).
  \end{equation*}
  In particular, condition \refcond{E} follows by setting $
  H=N_0^{2(d-1)\delta} $, with $ [a,b]=[-N_0,N_0] $ (recall that $
  \mu\gg \delta $; we choose $ N_0 $ such that $ N_0^\mu\gg
  \log\lambda $).

  Finally, we check condition \refcond{D}. Fix
  $ \eta\in(\eta_0-r_0'',\eta_0+r_0'') $. For the rest of the proof $
  \tilde x $ stands for the parametrization associated with $ [a,b]=[-N_0,N_0] $. Note that for condition
  \refcond{D} to hold it is enough that, given $ h  $, $
  \dist(h,\fT_0)\ge \exp(-N_0^\mu) $, we can find $
  |n'|,|n''|<N_0^{1/2} $ such that
  \begin{equation*}
	\mes \{ \phi\in \cI_0: \dist(\spec H_{[-N_0+n',N_0+n'']}(\tilde x(\phi,\eta)),\lambda\eta)
    <\exp(-N_0^\beta/2) \}<\exp(-N_0^{2\delta}).
  \end{equation*}

  We first consider the case
  $ \dist(h,\fT_0)\ge \exp(-\fc_0(\log\lambda)^{3/4}) $. Let
  \begin{equation*}
	\cB_{\eta,h}'=\cB_{(\log\lambda)^{3/4},\eta,h},\qquad
    \cB_{N_0,\eta,h}'=\bigcup_{|n|\le N_0} \cB_{\eta,h+n\omega}'.
  \end{equation*}
  Since $ \norm{h+n\omega}\ge \exp(-\fc_0(\log\lambda)^{3/4}) $, using
  \cref{prop:MSAIMPLICIT21}, we have
  \begin{equation*}
    \mes(\cB_{N_0,\eta,h}')<\exp(-(\log\lambda)^{3\fc_1/4}/2).
  \end{equation*}
  In particular, there exists $ \hat \phi\in \cI_0''\setminus \cB_{N_0,\eta,h}' $, $ |\hat \phi-\phi_0|\ll
  r_0'' $, such that
  \begin{equation*}
	|V(x(\hat \phi,\eta)+h+n\omega)-\eta|\ge \exp(-(\log\lambda)^{3/4}),\quad |n|\le N_0,
  \end{equation*}
  and therefore
  \begin{equation*}
	|V(\tilde x(\hat \phi,\eta)+h+n\omega)-\eta|\gtrsim \exp(-(\log\lambda)^{3/4}),\quad |n|\le N_0.
  \end{equation*}
  Using Cartan's estimate
  \begin{equation*}
	\mes \{ \phi\in \cI_0''/10: \log |V(\tilde x(\phi,\eta)+h+n\omega)-\eta|
    <-C_d (\log\lambda)^{3/4} N_0^{3(d-1)\delta} \} <\exp(-2N_0^{2\delta}),\quad |n|\le N_0.
  \end{equation*}
  Using \cref{lem:efextension2b} we get
  \begin{equation*}
	\mes \{ \phi\in \cI_0''/10: \dist(\spec H_{[-N_0,N_0]}(\tilde x(\phi,\eta)+h),\lambda\eta)
    <\exp(-C(\log\lambda)^{3/4} N_0^{3(d-1)\delta}) \}<\exp(-N_0^{2\delta}),
  \end{equation*}
  and condition \refcond{D} holds, since $ \beta\gg \delta $.

  Next we consider the case $ \exp(-N_0^\mu)\le
  \dist(h,\fT_0)<\exp(-\fc_0(\log\lambda)^{3/4}) $. Let $ n_1 $, $
  |n_1|\le 3N_0/2 $, such that
  \begin{equation*}
	\dist(h,\fT_0)=\norm{h-n_1\omega}.
  \end{equation*}
  We consider two sub-cases. First, suppose
  $ n_1\notin [-N_0+N_0^{1/3},N_0-N_0^{1/3}] $. Note that for
  $ n\in [-N_0+N_0^{1/3},N_0-N_0^{1/3}] $,
  \begin{equation*}
	\norm{h+n\omega}\ge \norm{(n-n_1)\omega}-\norm{h-n_1\omega}
    \ge a(CN_0)^{-b}-\exp(-\fc_0 (\log\lambda)^{3/4})
    \ge \exp(-\fc_0 (\log\lambda)^{3/4}).
  \end{equation*}
  Then, as above, we get
  \begin{multline*}
	\mes \{ \phi\in \cI_0''/10
    : \dist(\spec H_{[-N_0+N_0^{1/3},N_0-N_0^{1/3}]}(\tilde x(\phi,\eta)+h),\lambda\eta)
    <\exp(-C(\log\lambda)^{3/4} N_0^{3(d-1)\delta}) \}\\<\exp(-N_0^{2\delta}),
  \end{multline*}
  and condition \refcond{D} holds. Next, we consider $ n_1\in [-N_0+N_0^{1/3},N_0-N_0^{1/3}] $.
  Let
  \begin{equation*}
	h_1=h-n_1\omega,\quad [a_1,b_1]=n_1+[-N_0,N_0]
  \end{equation*}
  and $ \tilde x_1(\phi,\eta) $ the parametrization associated with $
  [a_1,b_1] $. Note that  $ [a_1,b_1]\supset [-\uN,\uN]  $. Since
  \begin{equation*}
	|\tilde x_1(\phi,\eta)+h_1-x(\phi,\eta)|\lesssim \exp(-\fc_0(\log\lambda)^{3/4}),
  \end{equation*}
  using \cref{eq:loc-estimates} we have
  \begin{equation*}
	|\lambda^{-1}E^{[a_1,b_1]}(\tilde x_1(\phi,\eta)+h_1)-\eta|<\exp(-\fc_0(\log\lambda)^{3/4}/2),
  \end{equation*}
  for any $ \phi\in \cI_0'' $. Due to the separation of eigenvalues in
  \cref{eq:loc-estimates}, we now have
  \begin{equation*}
	\dist(\spec H_{[a_1,b_1]}(\tilde x_1(\phi,\eta)+h_1),\lambda\eta)
    =|E^{[a_1,b_1]}(\tilde x_1(\phi,\eta)+h_1)-\lambda\eta|.
  \end{equation*}
  Let $ \hat \phi $ be as in \cref{eq:grad-lb}, with $ [a,b]=[a_1,b_1]
  $, $ h_0=\norm{h_1}^{-1}h_1 $. Then by Taylor's formula
  \begin{multline*}
	|E^{[a_1,b_1]}(\tilde x_1(\hat \phi,\eta)+h_1)-\lambda\eta|
    \ge |\langle \nabla E^{[a_1,b_1]}(\tilde x_1(\hat \phi,\eta)),h_1 \rangle|\norm{h_1}
    -C_\rho\lambda\norm{V}_\infty \norm{h_1}^2\\
    \gtrsim \lambda\exp(-(\log\lambda)^{1/2})\norm{h_1}\ge \exp(-2N_0^\mu).
  \end{multline*}
  Using Cartan's estimate it follows that
  \begin{equation*}
	\mes \{ \phi\in \cI_0''/10: \dist(\spec H_{[a_1,b_1]}(\tilde x_1(\phi,\eta)+h_1,\lambda\eta))
    <\exp(-C(N_0^\mu+N_0^{3(d-1)\delta})) \}<\exp(-N_0^{2\delta}).
  \end{equation*}
  Now the conclusion follows from the fact that $ \spec
  H_{[-N_0,N_0]}(\tilde x(\phi,\eta)+h)=\spec H_{[a_1,b_1]}(\tilde
  x(\phi,\eta)+h_1) $, and that by \cref{eq:better-delta-level},
  \begin{equation*}
	|\tilde x_1(\phi,\eta)-\tilde x(\phi,\eta)|\lesssim \exp(-(\log\lambda)N_0^{1/4})\ll \exp(-N_0^\beta/2)
  \end{equation*}
  (also recall that $ \delta\ll \mu\ll \beta\ll 1 $).
\end{proof}

\section{Proofs of the Main Theorems}\label{sec:main-thm}

The first two results are non-perturbative and are stated for
operators as in \cref{eq:schr100}. For their statements recall the
constants $ S_V $ and $ B_0 $ introduced in \cref{eq:SV},\cref{eq:B0},
and the exponents $ \delta,\fd $ used for the inductive conditions in
\cref{sec:bulk} and \cref{sec:edges}. We will use the notation $ \cS:=\spec H(x) $.

\begin{thmb}\label{thm:B}
  Assume the notation of the inductive conditions
  \refcond{A}-\refcond{E} from \cref{sec:bulk}.  Let $ E_0\in \R $,
  $ N_0\ge 1 $, and assume $ L(E)> \gamma>0 $ for
  $ E\in(E_0-2r_0,E_0+2r_0) $, $ r_0=\exp(-N_0^\delta) $. If
  $ N_0\ge (B_0+ S_V+ \gamma^{-1})^C $, $ C=C(a,b,\rho) $, and the
  conditions \refcond{A}-\refcond{E} hold with $ s=0 $ for the given $ E_0 $, then
  $ [E_0-r_0,E_0+r_0]\subset \cS $.
\end{thmb}
\begin{proof}
  Take an arbitrary $E\in(E_0-r_0,E_0+r_0)$ and 
  apply \cref{thm:D} with $ E_s=E $, $ s\ge 1 $. 
  Since $\mathcal{I}_{s}\Subset \mathcal{I}_{s-1}$,  there exists
  $\hat\phi\in \bigcap_s \mathcal{I}_{s}$.
  Due to \eqref{inductive50E} there exists $ x(E) $ such that
  \[
    |x(E)-x_{s}(\hat \phi,E)|<2\exp(-\gamma N_s/30),\quad s\ge 0.
  \]
  Due to \eqref{inductive51EE} there exists $ \psi(E,\cdot) $, $ \norm{\psi(E,\cdot)}=1 $, such that
  \[
    \|\psi(E,\cdot)-\psi^{[-N_s',N_s'']}(x_s(\hat \phi,E),\cdot)\|<2\exp(-\gamma N_s/40),\quad s\ge 0.
  \]
  Note that
  \begin{equation*}
    \|(H(x_{s}(\hat\phi,E))-E)\psi^{[-N_s',N_s'']}(x_{s}(\hat\phi,E),\cdot)\|\lesssim
    \exp(-\gamma N_s/20)
  \end{equation*}
  (by condition \refcond{C})
  and
  \[
    \|H(x(E))-H(x_{s}(\hat\phi,E))\|\le C_\rho \norm{V}_\infty|x(E)-x_s(\hat \phi,E)|
    <\exp(-\gamma N_s/40).
  \]
  It follows that
  \begin{equation*}
	\norm{(H(x(E))-E)\psi(E,\cdot)}\lesssim \exp(-\gamma N_s/40),\quad s\ge 0,
  \end{equation*}
  and therefore $ H(x(E))\psi(E,\cdot)=E\psi(E,\cdot) $. In
  particular, $ E\in \cS $ and the conclusion holds (recall that $ \cS
  $ is closed).
\end{proof}

\begin{thmc}\label{thm:C}
  Assume the notation of the inductive conditions
  \refcondu{A}-\refcondu{D} from \cref{sec:edges}. Let $ \ux_0\in \T^d $, $ N_0\ge 1 $, such that the
  conditions \refcondu{A}-\refcondu{D} hold, and assume $ L(E)>\gamma>0 $
  for $ E\in(\uE_0-2r_0,\uE_0+2r_0) $, $ r_0=\exp(-N_0^{\fd}) $.
  If $ N_0\ge (B_0+ S_V+ \gamma^{-1})^C $, $ C=C(a,b,\rho) $, then
  there exists $ \uE\in \R $, such that $ |\uE-\uE_0|<\exp(-\gamma
  N_0/100) $, $ \cS\cap(-\infty,\uE)=\emptyset $,  and $
  [\uE,\uE_0+\exp(-N_0^{20\fd})]\subset \cS $. Analogous statements hold relative to
  conditions \refcondo{A}-\refcondo{D}.
\end{thmc}
\begin{proof} 
  We choose $ N_0 $ large enough for \cref{thm:E} to hold. Using
  \cref{eq:thmE-estimates}, we have 
  that there exist
  \begin{equation*}
	\underline{x}=\lim_{s\to \infty} \underline{x}_s,\quad
    \underline{E}=\lim_{s\to \infty} \uE_s,
  \end{equation*}
  and we have
  \begin{equation}\label{eq:u-us}
	|\ux-\ux_s|,\, |\uE-\uE_s|<\exp(-\gamma N_s/100),\quad s\ge 1.
  \end{equation}
  First we verify that $ (-\infty,\underline{E})\cap\cS=\emptyset $.
  Take an arbitrary $E<\underline{E}$ and let $ \rho=\uE-E>0 $. By \cref{eq:u-us}, for any $ s\ge 1 $ we have
  \begin{equation*}
	\uE_s-E>\rho-\exp(-\gamma N_s/100)
  \end{equation*}
  and therefore
  \begin{equation*}
	\dist(\spec H_{[-N_s',N_s'']}(\ux_s),E)>\rho-\exp(-\gamma N_s/100)
  \end{equation*}
  (recall condition \refcondu{A}). Using \cref{eq:u-us} again, 
  \begin{equation*}
	\dist(\spec H_{[-N_s',N_s'']}(\ux),E)>\rho-\exp(-\gamma N_s/200)\ge \rho/2>0,
  \end{equation*}
  for $ s\ge s_0 $, with $ s_0 $ such that $ \exp(-\gamma
  N_{s_0}/200)\le \rho/2
  $. Then by \cref{lem:elemspec1} we have $ \dist(E,\cS)\ge \rho/2>0 $, hence $
  E\notin \cS $, as desired. 

  By \cref{thm:E}, the conditions \refcond{A}-\refcond{E} are
  satisfied for any $ E_s $, $ \exp(-N_s^{100\fd})\le |E_s-\uE_s|\le
  \exp (-N_s^{2\fd}) $, $ s\ge 1 $. Then by \cref{thm:B},
  \begin{equation*}
	[\uE_s+\exp(-N_s^{100\fd}),\uE_s+\exp(-N_s^{2\fd})]\subset \cS.
  \end{equation*}
  These intervals overlap for consecutive $ s $ (recall that
  $ N_{s+1}=N_s^5 $ and $ |\uE_{s+1}-\uE_{s}|<\exp(-\gamma N_s/60) $)
  and we have
  \begin{equation*}
	\cS\supset \bigcup_{s\ge 0}[\uE_s+\exp(-N_s^{100\fd}),\uE_s+\exp(-N_s^{2\fd})]
    \supset (\uE,\uE_1+\exp(-N_1^{2\fd})]\supset (\uE,\uE_0+\exp(-N_0^{20\fd})]
  \end{equation*}
  The conclusion follows since $ \cS $ is closed.
\end{proof}

We are finally ready to prove \cref{thm:A}. We fix the constants $
\fc_1,\fc_0,\fC_0 $ from \cref{defi:genericU}.
\begin{proof}[Proof of \cref{thm:A}]
  (a) Let $ T_V $ as in \cref{eq:TV-a}.  Take $ C_0=C_0(a,b,\rho,d) $
  large enough, such that for $ \lambda\ge \exp((T_V)^{C_0}) $,
  \cref{prop:A-to-E} with $ \epsilon=\fc_1/20 $, \cref{thm:B}, and  
  \cref{thm:C}  hold for $ N_0=\floor{\exp((\log\log\lambda)^2)} $
  (recall \cref{prop:LLBasic} and \cref{rem:Lbridge}; of course, we take $ \gamma=\log\lambda/2 $). The choice of $
  \epsilon $ is made with part (b) in mind.

  Let $ \uE_0 $, $ |\lambda^{-1}\uE_0- V(\ux)|\ll \lambda^{-1/4} $ be
  as in \cref{prop:A-to-E} and $ \uE $,
  $ |\uE-\uE_0|<\exp(-(\log\lambda) N_0/2) $, be as in
  \cref{thm:C}. Combining \cref{prop:A-to-E} with \cref{thm:B} we have
  \begin{equation*}
    [\uE_0+\exp(-N_0^{100\fd}),\uE_0+\lambda\exp(-(\log\lambda)^{\fc_1/5})]\subset \cS_\lambda.
  \end{equation*}
  At the same time, combining \cref{prop:A-to-E} with \cref{thm:C} we
  have
  \begin{equation*}
	[\uE,\uE_0+\exp(-N_0^{20\fd})]\subset \cS_\lambda,\qquad 	(-\infty,\uE)\cap \cS_\lambda=\emptyset.
  \end{equation*}
  Then
  \begin{equation}\label{eq:uE-lambda}
	[\uE,\uE_0+\lambda\exp(-(\log\lambda)^{\fc_1/5})]\subset \cS_\lambda
  \end{equation}
  This yields part (a). Of course, the proof the statement relative to
  the absolute maximum is completely analogous. Also, in the statement
  of part (a) we could replace $ \exp(-(\log\lambda)^{1/2}) $ by $
  \exp(-(\log\lambda)^\epsilon) $, for any $ \epsilon\in(0,1) $, by
  adjusting the constant $ C_0 $ from above.

  \medskip\noindent (b)
  Recall that $ \fE $ denotes the set of critical points of $ V
  $. Note that since all the critical points are assumed to be
  non-degenerate, by \cref{lem:Morse7a}, $ \fE $ is discrete and hence
  finite. Let
  \begin{equation*}
	\nu=\min_{x\in \fE} \norm{\fH(x)^{-1}}^{-1}
  \end{equation*}
  Using \cref{lem:Morse7a} and \cref{lem:uV-shifts} we choose $
  c=c(\rho) $ small enough so that with $
  r=c\nu(1+\norm{V}_\infty)^{-1} $ we have that $ \T^d\setminus \bigcup_{x\in \fE} B(x,r)  $
  is connected and \cref{eq:uV-shifts} holds.
  Let
  \begin{equation*}
	\fg=\fg(V):=\min \{ \norm{\nabla V(x)} : x\in  \T^d\setminus \bigcup_{x\in \fE} B(x,r)  \}>0,
  \end{equation*}
  and increase $ T_V $ to be
  \begin{multline}\label{eq:TV-final}
    T_{V}=2+\max(0,\log\norm{V}_\infty)+\max(0,\log \uiota ^{-1})\\
    +\max(0,\log\iota^{-1})+\max(0,\log\nu^{-1})+\fC_0+\fc_0^{-1}+\max(0,\log \fg^{-1}).
  \end{multline}
  Take $ C_0=C_0(a,b,\rho,d) $ large enough, such that for $ \lambda\ge
  \exp((T_V)^{C_0}) $ in addition to the assumptions for part (a) we also have
  \begin{equation}\label{eq:lambda-grad-restriction}
	\exp(-(\log\lambda)^{\fc_1/3})\le \min(\nu r/2,\fg),
  \end{equation}
  and \cref{prop:A-to-D} holds with $ N_0=\floor{\exp((\log\log\lambda)^2)} $.

  Let $ r_\lambda $ such that $
  \nu r_\lambda/2=\exp(-(\log\lambda)^{\fc_1/3}) $. By
  \cref{eq:lambda-grad-restriction}, $ r_\lambda\le r $ and therefore
  $ \cG_\lambda:=\T^d\setminus \bigcup_{x\in \fE} B(x,r_\lambda) $ is
  connected. By \cref{eq:lambda-grad-restriction} and
  \cref{eq:uV-shifts},
  \begin{equation*}
	\norm{\nabla V(x)}\ge \exp(-(\log\lambda)^{\fc_1/3}),
    \quad x\in \cG_\lambda.
  \end{equation*}
  Combining \cref{prop:A-to-D} and \cref{thm:B} we have
  \begin{equation*}
	\{ \lambda V(x) : x\in \cG_\lambda\}\subset \cS_\lambda.
  \end{equation*}
  Take $ \ux',\ox'\in \cG_\lambda $, $
  \norm{\ux'-\ux}=\norm{\ox'-\ox}=r_\lambda$. Since $ \cG_\lambda $ is
  connected we have
  \begin{equation*}
	[\lambda V(\ux'),\lambda V(\ox')]\subset \{ \lambda V(x) : x\in \cG_\lambda\}\subset \cS_\lambda.
  \end{equation*}
  Let $ \uE_0,\uE $ as in part (a).
  By \cref{eq:uV-shifts} and by increasing $ C_0 $ if needed,
  \begin{equation*}
	\exp(-3(\log\lambda)^{\fc_1/3})\le V(\ux')-V(\ux)\le \exp(-(\log\lambda)^{\fc_1/3})
  \end{equation*}
  and therefore
  \begin{equation*}
	\lambda\exp(-3(\log\lambda)^{\fc_1/3})\lesssim |\lambda V(\ux')-\uE_0|
    \lesssim \lambda \exp(-(\log\lambda)^{\fc_1/3}).
  \end{equation*}
  From the above and \cref{eq:uE-lambda} it follows that $ [\uE,\lambda V(\ox')]\subset
  \cS_\lambda $. Let $ \oE $ be as in \cref{thm:C} with respect to the
  conditions \refcondot{A}-\refcondot{D}. Analogously, we get
  $ [\lambda V(\ux'),\oE]\subset \cS_\lambda $ and therefore $
  [\uE,\oE]\subset \cS_\lambda $.  Since 
  \begin{equation*}
	(-\infty,\uE)\cap \cS_{\lambda}=(\oE,\infty)\cap \cS_\lambda=\emptyset,
  \end{equation*}
  we conclude that $ \cS_\lambda=[\uE,\oE] $.
\end{proof}
\begin{remark}
  The constant $ \uiota $ in the definition of $ T_V $ from the proof
  of \cref{thm:A} (b) is redundant and can be dropped at the cost of slightly
  increasing $ C_0 $ in the lower bound for $ \lambda $. More precisely, it
  can be seen, by using Taylor's formula, that $ \uiota $ can be bound below in terms of $ \nu $,
  $ \fg $, $ \norm{V}_\infty $, and $ \rho $.
\end{remark}

\section{An Example}\label{sec:example}

For the purpose of this section it is convenient to redefine $
\T:=\R/(2\pi\Z) $. Let
\begin{equation*}
  V(x,y)=\cos(x)+s\cos(y).
\end{equation*}
We will check that $ V $ satisfies the conditions of
\cref{defi:genericU} for $ s\notin\{-1,0,1\} $.

First, a direct computation shows that conditions (i),(ii) of
\cref{defi:genericU} are satisfied for $ s\neq 0 $ and they fail for $
s=0 $.

Next we show that condition (iii) holds for $ s\notin \{ -1,0,1 \} $.  Take
\begin{equation}\label{eq:H-lb}
  H\gg 1+\max(\log |s|,\log |s|^{-1},\log |1-s^2|^{-1}),
\end{equation}
$ h\in \T^2 $, $ h=(\alpha,\beta) $, $ \norm{h}\ge \exp(-H) $. The
lower bound on $ H $ will be used tacitly in most of the estimates to
follow. Recall that when  $ \norm{\cdot} $ is applied to the shifts $
h,\alpha,\beta $, it stands for the usual norm on the torus.

\begin{lemma}\label{lem:simple-case}
  If $ \norm{\alpha}<\exp(-2H) $ or $ \norm{\beta}<\exp(-2H)$, then
  \begin{equation*}
	|V(x+\alpha,y+\beta)-V(x,y)|\gtrsim \exp(-2H),
  \end{equation*}
  for any $ (x,y)\in \T^2 $.
\end{lemma}
\begin{proof}
  Assume $ \norm{\alpha}<\exp(-2H) $. Since $ \norm{h}\ge \exp(-H) $,
  we must have $ \norm{\beta}\gtrsim \exp(-H) $ and therefore
  \begin{multline*}
	|V(x+\alpha,y+\beta)-V(x,y)|\ge s|\cos(y+\beta)-\cos y|-|\cos(x+\alpha)-\cos x|\\
    \gtrsim s \exp(-H)-C\exp(-2H)\gtrsim  s\exp(-H)\ge  \exp(-2H).
  \end{multline*}
  Similarly, if $ \norm{\beta}<\exp(-2H) $, then $
  \norm{\alpha}\gtrsim \exp(-H) $, and
  \begin{multline*}
	|V(x+\alpha,y+\beta)-V(x,y)|\ge |\cos(x+\alpha)-\cos x|-s|\cos(y+\beta)-\cos y|\\
    \gtrsim \exp(-H)-Cs\exp(-2H)\gtrsim \exp(-H).
  \end{multline*}  
\end{proof}

Let $ g(x,y,\alpha,\beta):=g_{V,h,1,2}(x,y) $, with $ g_{V,h,1,2} $ as
in \cref{defi:genericU}. Note that $ |g_{V,h,1,2}|=|g_{V,h,2,1}| $.

\begin{lemma}\label{lem:iii-Rz}
  If $ \norm{\alpha},\norm{\beta}\ge \exp(-2H) $, then there exists an
  absolute constant $ C_0\gg 1 $ such that
  \begin{gather*}
	\mes \{ y\in \T: \min_{x}(|V(x+\alpha,y+\beta)-V(x,y)|+|g(x,y,\alpha,\beta)|)
    <\exp(-C_0H) \}<2\exp(-H/2),\\
    \mes \{ x\in \T: \min_{y}(|V(x+\alpha,y+\beta)-V(x,y)|+|g(x,y,\alpha,\beta)|)<\exp(-C_0H) \}<2\exp(-H/2).
  \end{gather*}
\end{lemma}
\begin{proof}
  We only check the first estimate, the second one being completely
  analogous. Let 
  \begin{equation*}
    z=\exp(ix),\quad w=\exp(iy),\quad A=\exp(i\alpha),\quad B=\exp(i\beta).
  \end{equation*}
  Then
  \begin{equation}\label{eq:V-g-complexified}
    V(x+\alpha,y+\beta)-V(x,y)=\frac{1}{2zw} P_1(z,w),\quad g(x,y,\alpha,\beta) =-\frac{1}{4z w } Q_1(z,w),
  \end{equation}
  with
  \begin{gather*}
    P_1(z,w)=(A-1)z^2w+s(B-1)zw^2+(A^{-1}-1)w+s(B^{-1}-1)z,\\
    Q_1(z,w)=(B-sA)z^2w^2+(sA-B^{-1})z^2+(sA^{-1}-B)w^2+B^{-1}-sA^{-1}.
  \end{gather*}
  Let $ a_i $, $ b_i $ be the polynomials in $ w $ such
  that
  \begin{equation*}
	P_1(z,w)=a_2z^2+a_1z+a_0,\quad Q_1(z,w)=b_2z^2+b_1z+b_0.
  \end{equation*}
  Let
  \begin{equation*}
	R_1(w)=\Res_z (P_1,Q_1)= \det \begin{bmatrix}
      a_2 & 0 & b_2 & 0\\
      a_1 & a_2 & b_1 & b_2\\
      a_0 & a_1 & b_0 & b_1\\
      0 & a_0 & 0& b_0
    \end{bmatrix}=\sum_{k=0}^8 c_k w^k.
  \end{equation*}
  Analyzing the degrees of the terms from the Leibniz formula for the above
  determinant, one sees that the only term containing a monomial
  of degree $ 8 $  is
  \begin{equation*}
	a_1^2 b_2 b_0=[s(B-1)w^2+s(B^{-1}-1)]^2[(B-sA)w^2+sA-B^{-1}][(sA^{-1}-B)w^2+B^{-1}-sA^{-1}],
  \end{equation*}
  corresponding to the even permutation
  \begin{equation*}
	\begin{pmatrix}
      1 & 2 & 3 & 4\\
      2 & 3 & 1 & 4
    \end{pmatrix}.
  \end{equation*}
  It follows that
  \begin{equation*}
    c_8=s^2(B-1)^2(B-sA)(sA^{-1}-B)=s^2(B-1)^2(BA^{-1}-s)(s-AB),
  \end{equation*}
  and therefore
  \begin{equation*}
	|c_8|\gtrsim s^2\exp(-4H)|1-|s||^2>\exp(-5H).
  \end{equation*}
  Then, using \cref{lem:Cartan-P},
  \begin{equation*}
	|R_1(\exp(iy))|\ge \exp(-CH),
  \end{equation*}
  for $ y\in \T\setminus \cB_1 $, $ \mes(\cB_1)<\exp(-H/2) $, with $
  C $ an absolute constant. Applying \cref{lem:Cartan-P} again to $
  b_2(w)=(B-sA)w^2+sA-B^{-1} $, we get that
  \begin{equation*}
	|b_2(\exp(iy))|\ge \exp(-CH),
  \end{equation*}
  for $ y\in \T\setminus \cB_2 $, $ \mes(\cB_2)<\exp(-H/2) $. At the
  same time,
  \begin{equation*}
	|a_2(\exp(iy))|=|(A-1)\exp(iy)|\gtrsim \exp(-2H),
  \end{equation*}
  for any $ y\in \T $. Let $ r(\exp(iy)) $ be the maximum of the
  absolute values of the roots of $ P_1(\cdot,\exp(iy)) $ and $
  Q_1(\cdot,\exp(iy)) $.
  Using \cref{lem:Cauchy-bound} we have that the
   $r(\exp(iy)) \le  \exp(CH) $, for $ y\in
   \T\setminus \cB_2 $.

   Fix $ y\in \T\setminus \cB $,
   $ \cB:=\cB_1\cup \cB_2 $.
   It follows that
   \begin{equation*}
     R_1(\exp(iy))\ge 2|a_2(\exp(iy))|^2|b_2(\exp(iy))|^2 r(\exp(iy))^3 \delta,
   \end{equation*}
   where $ \delta=\exp(-CH) $, with $ C $ a sufficiently large
   absolute constant. By \cref{lem:resultant},
   \begin{equation*}
     \max (P_1(z,\exp(iy)),Q_1(z,\exp(iy)))\ge \min(|a_2(\exp(iy))|,|b_2(\exp(iy))|)\delta^2
     >\exp(-CH)
   \end{equation*}
   for any $ z $, and therefore
   \begin{equation*}
     |V(x+\alpha,y+\beta)-V(x,y)|+|g(x,y,\alpha,\beta)|\ge \exp(-CH),
   \end{equation*}
   for any $ x\in \T $ (recall \cref{eq:V-g-complexified}). The conclusion follows.
\end{proof}

Now condition (iii) follows from \cref{lem:simple-case} and
\cref{lem:iii-Rz}, by setting $ K=C_0H $, with $ C_0 $ as in
\cref{lem:iii-Rz}, and by taking $ \fc_0=1/C_0 $, $ \fc_1=1/2 $,
\begin{equation}\label{eq:fC0}
  \fC_0=C(C_0^2+C_0\max(\log |s|,\log |s|^{-1},\log |1-s^2|^{-1})),\quad C\gg 1.
\end{equation}

Finally, we check that condition (iv) holds for $ s\notin \{ -1,0,1 \}
$. Take $ H$ as in \cref{eq:H-lb}, $ \eta\in \R $, and $ h_0\in \R^2 $
a unit vector. With some abuse of notations we let
$ h_0=(\alpha,\beta) $, $ \alpha^2+\beta^2=1 $.

\begin{lemma}\label{lem:iv-simple-case}
  (a) If $ |\alpha|<\exp(-2H) $, then
  \begin{equation*}
	\mes \{ y\in \T : \min_{x} |\langle \nabla V(x,y),h_0 \rangle|< \exp(-2H)\}<\exp(-H).
  \end{equation*}
  \noindent (b) If $ |\beta|<\exp(-2H) $, then
  \begin{equation*}
	\mes \{ x\in \T : \min_{y} |\langle \nabla V(x,y),h_0 \rangle|< \exp(-2H)\}<\exp(-H).
  \end{equation*}
\end{lemma}
\begin{proof}
  (a) Since $ |\alpha|<\exp(-2H) $, we have $ |\beta|\ge
  (1-\exp(-4H))^{1/2}>1/2 $, and therefore
  \begin{equation*}
	|\langle \nabla V(x,y),h_0 \rangle|=|\alpha \sin x+s\beta \sin y|
    \ge \frac{1}{2}s\sin y-\exp(-2H)\ge \exp(-2H),
  \end{equation*}
  for all $ x\in \T $, and  $ y $ such that $ |\sin y|>\exp(-3H/2)
  $. The conclusion follows.  The proof for (b) is analogous.
\end{proof}

\begin{lemma}\label{lem:iv-Rz}
  (a) If $ |\alpha|\ge \exp(-2H) $, then there exists an absolute
  constant $ C_0\gg 1 $ such that
  \begin{equation*}
	\mes \{ y\in \T : \min_{x} (|V(x,y)-\eta|+|\langle \nabla V(x,y),h_0 \rangle|)
    < \exp(-C_0H)\}<\exp(-H/2).
  \end{equation*}
  \noindent (b) If $ |\beta|\ge \exp(-2H) $, then there exists an absolute
  constant $ C_0\gg 1 $ such that
  \begin{equation*}
	\mes \{ x\in \T : \min_{y}(|V(x,y)-\eta|+|\langle \nabla V(x,y),h_0 \rangle|)< \exp(-C_0H)\}
    <\exp(-H/2).
  \end{equation*}  
\end{lemma}
\begin{proof}
  We only prove (a), the proof of the second statement being analogous.
  By letting $ z=\exp(ix) $, $ w=\exp(iy) $, we have
  \begin{equation}\label{eq:V-eta-complexified}
    V(x,y)-\eta=\frac{1}{2zw} P_2(z,w),\quad \langle \nabla V(x,y),h_0
    \rangle=-\frac{1}{2izw} Q_2(z,w),
  \end{equation} with
  \begin{gather*}
    P_2(z,w)=z^2w+szw^2-2\eta
    zw+w+sz,\\ Q_2(z,w)=\alpha z^2 w+\beta zw^2-\alpha w-\beta z.
  \end{gather*}
  Let $ a_i $, $ b_i $ be the polynomials in $ w $ such
  that
  \begin{equation*}
	P_2(z,w)=a_2z^2+a_1z+a_0,\quad Q_2(z,w)=b_2z^2+b_1z+b_0.
  \end{equation*}
  In particular, $ a_2(w)=w $ and $ b_2(w)=\alpha w $.
  A direct computation yields
  \begin{multline*}
    R_2(w)=\Res_z(P_2,Q_2)=\sum_{k=0}^6 c_kw^k\\
    \kern-20em=w^{6} \left(- \alpha^{2} s^{2} +
      \beta^{2}\right) +w^{5} (4 \alpha^{2} \eta s) \\ + w^{4} \left(- 4
      \alpha^{2} \eta^{2} - 2 \alpha^{2} s^{2} + 4 \alpha^{2} - 2
      \beta^{2}\right) + w^{3}(4 \alpha^{2} \eta s) + w^{2} \left(-
      \alpha^{2} s^{2} + \beta^{2}\right)\\ \kern-10em =w^{6} \left(1-
      \alpha^{2} (1+s^{2}) \right) +w^{5} (4 \alpha^{2} \eta s) \\
    + w^{4}\left(\alpha^2(6-4\eta^2-2s^2)-2\right) + w^{3}(4 \alpha^{2} \eta s) +
    w^{2} \left(1- \alpha^{2}(1+ s^{2}) \right).
  \end{multline*}
  We will argue that not all of the coefficients of $ R_2 $ are too
  small. To this end, note  that
  \begin{equation*}
	2\alpha^{-2}c_6+\alpha^{-2}c_4=4-4s^2-4\eta^2.
  \end{equation*}
  If $ |\eta|<\exp(-H) $, then
  \begin{equation*}
	|2\alpha^{-2}c_6+\alpha^{-2}c_4|>4|1-s^2|-4\exp(-2H)>2|1-s^2|>\exp(-H),
  \end{equation*}
  and therefore, either
  \begin{equation*}
	|c_6|\gtrsim \exp(-H)\alpha^2\ge \exp(-5H) \qquad\text{or}\qquad
    |c_4|\gtrsim \exp(-H)\alpha^2 \ge \exp(-5H).
  \end{equation*}
  On the other hand, if $ |\eta|\ge \exp(-H) $, then
  \begin{equation*}
	|c_5|\gtrsim \alpha^2\exp(-H)s>\exp(-6H). 
  \end{equation*}
  Thus, $ \max_{k}|c_k|\gtrsim \exp(-6H) $. Then, using \cref{lem:Cartan-P},
  \begin{equation*}
	|R_2(\exp(iy))|\ge \exp(-CH),
  \end{equation*}
  for $ y\in \T\setminus \cB $, $ \mes(\cB)<\exp(-H/2) $. Let $ r(\exp(iy)) $ be the maximum of the
  absolute values of the roots of $ P_2(\cdot,\exp(iy)) $ and $
  Q_2(\cdot,\exp(iy)) $. By \cref{lem:Cauchy-bound}, $
  r(\exp(iy))<\exp(3H) $. Then
  \begin{equation*}
	|R_2(\exp(iy))|\ge 2|a_2(\exp(iy))|^2|b_2(\exp(iy))|^2 r(\exp(iy))^3\delta,
  \end{equation*}
  for $ y\in \T\setminus \cB $, with $ \delta=\exp(-CH) $. By
  \cref{lem:resultant},
  \begin{equation*}
	\max(P_2(z,\exp(iy)),Q_2(z,\exp(iy)))\ge \min(|a_2(\exp(iy))|,|b_2(\exp(iy))|)\delta^2>\exp(-CH)
  \end{equation*}
  for any $ z $, and $ y\in \T\setminus \cB $. The conclusion follows
  by recalling \cref{eq:V-eta-complexified}.
\end{proof}

Now condition (iv) follows from \cref{lem:iv-simple-case} and
\cref{lem:iv-Rz}, by setting $ K=C_0H $, with $ C_0 $ as in
\cref{lem:iv-Rz}, and by taking $ \fc_1=1/2 $ and $ \fC_0 $ as in
\cref{eq:fC0}, with the new $ C_0 $.
Obviously, we can arrange for both condition (iii) and (iv) to hold
with the same $ \fC_0 $.

\begin{remark}
  (a) It should be clear that for $ s\in \{ -1,0,1 \} $ not all of the
  conditions are satisfied. Indeed, we noted that conditions (i) and (ii) fail for $ s=0
  $, and  for $ s=\pm 1 $, for example, condition (iv) fails for $
  \eta=0 $ and $ h_0 $ proportional to $ (\pm 1,1) $.

  \smallskip\noindent
  (b) Due to the choices of $ \fC_0 $ in \cref{eq:fC0} and $ \lambda_0
  $ implied by the proof of \cref{thm:A} (recall \cref{eq:TV-final}),
  we have that as $ s $ approaches $ \{ -1,0,1 \} $, $ \lambda_0 $
  approaches $ \infty $, as claimed in \cref{rem:thmAquantify} (c).

\end{remark}

\def\cprime{$'$}

\end{document}